\documentclass[11pt]{article}
\usepackage{cite}
\usepackage{hyperref}
\usepackage{appendix}
\usepackage{thmtools}  
\usepackage{amssymb,amsmath,amsfonts,mathrsfs,amsthm,epsfig,latexsym,color}
\usepackage{enumitem,geometry}
\usepackage{fourier}
\usepackage{tikz-cd}
\usepackage[all]{xy}

\makeatletter
\newcommand{\subsectionruninhead}{\@startsection{subsection}{2}{0mm}
	{-\baselineskip}{-0mm}{\bf\large}}
\newcommand{\subsubsectionruninhead}{\@startsection{subsubsection}{3}{0mm}
	{-\baselineskip}{-0mm}{\bf\normalsize}}
\makeatother

\newtheorem*{theorem*}{Theorem}
\newtheorem*{proof*}{Proof}
\newtheorem*{proposition*}{Proposition}
\newtheorem*{notation*}{Notation}
\newtheorem*{corollary*}{Corollary}
\newtheorem*{claim*}{Claim}
\newtheorem*{remark*}{Remark}

\newtheorem{proposition}{Proposition}[section]
\newtheorem{theorem}[proposition]{Theorem}
\newtheorem{corollary}[proposition]{Corollary}
\newtheorem{lemma}[proposition]{Lemma}
\newtheorem{claim}[proposition]{Claim}

\theoremstyle{definition}
\newtheorem{definition}[proposition]{Definition}
\theoremstyle{remark}
\newtheorem{remark}[proposition]{Remark}

\numberwithin{equation}{section}
\newtheorem{theoalph}{\bf  Theorem}

\newtheorem{coralph}[theoalph]{\bf Corollary}

\def\NN{\mathbb{N}}
\def\RR{\mathbb{R}}
\def\TT{\mathbb{T}}
\def\ZZ{\mathbb{Z}}
\def\e{{\varepsilon}}

\def\vphi{\varphi}
\def\cL{{\cal L}}
\def\cO{{\cal O}}
\def\cF{{\cal F}}

\begin{document}
\title{\bf   The Teichm\"uller Space of  a $3$-Dimensional  Anosov Flow}
\author{Ruihao Gu \quad and \quad Yi Shi\footnote{Y. Shi was partially supported by National Key R\&D Program of China (2021YFA1001900) and
NSFC (12526207, 12571203).}}
\date{\today }
\maketitle

\begin{abstract}

For a  transitive Anosov flow $\Phi$ on $3$-dimensional closed manifold $M$, we realize its  Teichm\"uller space in the sense of smooth orbit-equivalence classes as a product of two function spaces. As an application, we show  the path-connectedness of the orbit-equivalence space of $3$-dimensional transitive Anosov flows  which gives a positive answer of Potrie \cite[Question 1]{P2025} in dimension 3.  Further,  in the space of $C^r$-smooth ($r\geq 1$) $3$-dimensional Anosov flows on $M$, we show that $\mathcal{A}^r(\Phi)$ the path component containing $\Phi$  is homotopy equivalent to the identity component of the diffeomorphism group of the manifold, namely,
\[ \mathcal{A}^r(\Phi)\simeq {\rm Diff}^r_0(M). \]
Moreover, we show the rigidity of time-preserving conjugacy for $3$-dimensional transitive  Anosov flows admitting $C^1$-smooth strong stable foliations, which gives partial answer of Gogolev-Leguil-Rodriguez Hertz \cite[Question 2.8]{GLRH2025}. 
\end{abstract}

\section{Introduction}

Let $M$ be a $C^{\infty}$-smooth closed Riemannian manifold. A $C^r$-smooth flow $\Phi$ on $M$ is a $C^r$-smooth map $\Phi:M\times \RR\to M$ denoted by $\Phi^t(x)=\Phi(x,t)$ such that  $\Phi^0(x)=x$ and $\Phi^{t_2+t_1}(x)=\Phi^{t_2}\circ \Phi^{t_1}(x)$, for all $x\in M$ and $t_1,t_2\in \RR$.  We call two flows $\Phi$ and $\Psi$ on $M$  are  \emph{orbit-equivalent}, if there is a homeomorphism $H$ (called \emph{orbit-equivalence}) of $M$ such that
 \[H(\cO_\Phi(x))=\cO_\Psi(H(x)), \qquad \forall x\in M, \] 
where $\cO_\Phi$ and $\cO_\Psi$ are (oriented) orbits  of $\Phi$ and $\Psi$, respectively. 

We call $\Phi$ and $\Psi$ are \emph{conjugate}, if there is a homeomorphism $H$ (called \emph{conjugacy}) of $M$ such that 
\[H\circ \Phi^t(x)=\Psi^t\circ H(x),  \qquad \forall x\in M,~\forall t\in\RR.\] 
If the orbit-equivalence or the conjugacy $H$ is a $C^r$-smooth diffeomorphism, we call $\Phi$ and $\Psi$ are $C^r$-smoothly orbit-equivalent or $C^r$-smoothly conjugate, respectively.

A $C^1$-smooth flow $\Phi$ on  $M$ is called an \emph{Anosov flow}, if there exist a $D\Phi$-invariant splitting 
\[TM=E^{ss}_{\Phi}\oplus T\cO_{\Phi}\oplus E^{uu}_{\Phi}, \] 
and constants $C, \lambda>1$ such that
\[ \|D\Phi^t|_{E^{ss}_{\Phi}(x)} \| \leq C\lambda^{-t} \quad {\rm and} \quad  \|D\Phi^t|_{E^{uu}_{\Phi}(x)} \| \geq C\lambda^t,\quad \forall x\in M,\ \forall t>0. \]
We denote \[E^s_{\Phi}=E^{ss}_{\Phi}\oplus T\cO_{\Phi}\quad {\rm and} \quad  E^u_{\Phi}=E^{uu}_{\Phi}\oplus T\cO_{\Phi}.\]  We call the subbundles $E^s_{\Phi}/E^{ss}_{\Phi}/E^u_{\Phi}/E^{uu}_{\Phi}$, the \emph{stable/strong stable/unstable bundle/strong unstable} bundle of $\Phi$ which uniquely integrate to $\Phi$-invariant foliations called the \emph{stable/strong stable/unstable bundle/strong unstable} foliation of $\Phi$, respectively. We  also refer to  the strong stable foliation and the strong unstable foliation as the \emph{strong hyperbolic foliations}.

One of the main purposes of present paper is to understand the topology of the space of Anosov flows, for example, the  path-connectedness asked by Potrie \cite[Question 1]{P2025} and the homotopy type of the space. Particularly, we will give the characterization of the smooth equivalence classes  in the orbit-equivalence space of  a transitive $3$-dimensional Anosov flow  via representation by function spaces.

 The other main purpose is to give both rigidity and classification of the Anosov flows on $3$-manifold with $C^1$-smooth strong hyperbolic foliations,  which  follows from the question in the work of Gogolev-Leguil-Rodriguez Hertz, see \cite[Question 2.9]{GLRH2025}.

\subsection{Realization of 3-dimensional Anosov flows}

For a $3$-dimensional Anosov flow $\Phi$, we denote the Jacobians along the strong stable and the strong unstable bundles  by $J_\Phi^s(x,t)$ and $J_\Phi^u(x,t)$, respectively, namely,
\[J_\Phi^s(x,t)=\log \Big|{\rm det}\big(D\Phi^{t}|_{E^{ss}_\Phi(x)}\big)\Big|
\quad {\rm and}\quad J_\Phi^u(x,t)=\log \Big|{\rm det}\big(D\Phi^{t}|_{E^{uu}_\Phi(x)}\big)\Big|.\]
We denote by ${\rm Per}(\Phi)$ the set of periodic points of flow $\Phi$, and by  $\tau(p,\Phi)$ the minimal positive period of  $p\in {\rm Per}(\Phi)$.  For short, we denote the Jacobians of the first return map of $p\in {\rm Per}(\Phi)$ by 
\[J^s(p,\Phi):=J^s_\Phi\big(p,\tau(p,\Phi)\big)\quad {\rm and}\quad  J^u(p,\Phi):=J^u_\Phi\big(p,\tau(p,\Phi)\big).\]
Rigidity of smooth conjugacy via matching periodic data for hyperbolic systems is extensively studied. There are a lot of evidences show that the periodic data is the only obstruction for smooth conjugacy, see for example \cite{DG2024,dLMM1986,dL1992,G2017,GRH2023,GLRH2025,GRH2022}. The following result is on the flexibility of these stable and unstable Jacobians at periodic points. For short, we denote $C^{1+{\rm H\"older}}$ by $C^{1+}$, i.e., the $C^1$-smoothness  with  H\"older continuous derivative for some  H\"older exponent $0<\alpha<1$.

\begin{theoalph}\label{thm flow flexibility 0}
	\emph{Let $\Phi$ be a transitive $C^{1+}$-smooth Anosov flow on $3$-manifold $M$. Let $f_\sigma:M\to \RR\ (\sigma=s,u)$ be two H\"older continuous functions with topological pressures $P_{f_\sigma}(\Phi)=0$.  Then there exists a $C^{1+}$-smooth Anosov flow $\Psi$ of manifold $M$ orbit-equivalent to $\Phi$ via a H\"older continuous homeomorphism $H$ such that for all periodic point $p\in {\rm Per}(\Psi)$,
		\[ J^u(p,\Psi)=-\int_0^{\tau\big(H^{-1}(p),\Phi\big)} f_u\big( \Phi^t\circ H^{-1}(p)\big) dt \quad {\rm and} \quad J^s(p,\Psi)= \int_0^{\tau\big(H^{-1}(p),\Phi\big)} f_s \big( \Phi^t\circ H^{-1}(p)\big) dt.\]
	Moreover, such a $C^{1+}$-smooth Anosov flow $\Psi$ is unique, up to $C^{1+}$-smooth orbit-equivalence.}
\end{theoalph}

In \cite{C1993}, Cawley gives the Teichm\"uller space of an Anosov diffeomorphism on $2$-torus through representing it by H\"older continuous functions on $\TT^2$, see Theorem \ref{thm Cawley} and  also see \cite{KQ2025}. The above Theorem \ref{thm flow flexibility 0} can be viewed as the case for transitive $3$-dimensional Anosov flows. In fact, we will also give the Teichm\"uller space of an Anosov $3$-flow in its orbit-equivalence class, see Corollary \ref{cor Teich space orbit-eq}. We refer to the works of Farrell-Gogolev \cite{FaG2014} and Gogolev-Leguil-Rodriguez Hertz \cite[Section 2.4]{GLRH2025}  for more discussion on the space of an Anosov diffeomorphism on $\TT^2$. 	 

\begin{remark}
	A main tool in Cawley's work  \cite{C1993} is  \emph{transverse measure} class of the stable and unstable foliations  of Anosov diffeomorphisms on $\TT^2$ associated with   H\"older functions, which is also called Radon-Nikodym realization, see \cite[Theorem]{C1993}. 
	This transverse  measure class  is a natural extension of the  Margulis measure \cite{M2004}, and has been formulated  in different settings, for example \cite{H1994,BM1977,Le2000,A2008,KQ2025}.  Our work also needs  the Radon-Nikodym realization for Anosov flows which is proved by Asaoka \cite{A2008}, see Theorem \ref{thm Asaoka 0}. We refer to  Haydn's work \cite{H1994} for an earlier version for Anosov flows, where the transverse measure is locally given as  conditional measure of  equilibrium state with respect to  strong hyperbolic foliations, which is different from  Theorem \ref{thm Asaoka 0}. Comparing with the works of Cawley  \cite{C1993} and Asaoka \cite{A2008,A2012}, we would like to explain more details on Theorem \ref{thm flow flexibility 0}:
	\begin{itemize}
		\item  Cawley deforms a toral Anosov automorphism  such that the new Anosov diffeomorphism has aimed stable and unstable Jacobians under the metric associated with the transverse measure. Her method needs $C^{1+}$-smooth stable and unstable foliations. Corresponding to our flow case, it needs the strong hyperbolic foliations being $C^{1+}$-smooth, however this happens rarely, see for example  \cite[Theorem G]{GLRH2025}.  This is the main reason that we cannot get the flexibility in the conjugacy class, as Cawley has done in \cite{C1993}.
		\item It is worth to point out that in both works  \cite{A2008, A2012},  Asaoka uses the  transverse  measure to deform Anosov flows. In particular, in \cite{A2012}, he starts from GA-actions with smooth hyperbolic splitting, hence there is no obstacle on the regularity of foliations as mentioned above. Our way to deform Anosov flows is inspired by the work of Farrell-Gogolev \cite{FaG2014} and is different from Asaoka's method in \cite{A2008,A2012}. We will discuss more details on the necessity of using our method and on the difference to Asaoka's method  in Remark \ref{rmk Asaoka 1}, Remark \ref{rmk Asaoka 2} and Remark \ref{rmk HA-flow reason}.
	\end{itemize} 
\end{remark}

By the above realization result, we can further show the path-connectedness of Anosov flow's orbit-equivalence space, which gives a positive answer of a question of Potrie \cite[Question 1]{P2025} in the category of transitive $3$-dimensional Anosov flows. 

\begin{theoalph}\label{thm connect}
	\emph{Any two  transitive $C^r\ (r\geq 1)$ Anosov flows on a $3$-manifold $M$, which are orbit-equivalent by an orbit-equivalence homotopic to identity, can be connected by  a path of\, $C^r$ Anosov flows.}
\end{theoalph}

Following the path-connectedness, we also consider the homotopy type of a path component of the space of Anosov flows.  For a closed manifold $M$ and $r\geq 1$, we denote its  $C^r$-diffeomorphism group and the identity component by ${\rm Diff}^r(M)$  and ${\rm Diff}_0^r(M)$, respectively.    We denote by $\mathfrak{X}^r(M)$ the set of $C^r$-smooth vector fields on $M$, and $\mathcal{A}^r(M)$ the set of $C^r$-smooth Anosov vector fields on $M$. Recall that   $\mathfrak{X}^r(M)$ is a separable infinite-dimensional manifold (Banach Space or Fr\'echet Space when $r<\infty$ or $r=\infty$). 
It is known that $\mathcal{A}^r(M)$ is open in  $\mathfrak{X}^r(M)$ with respect to the $C^r$-topology (actually $C^1$-open).   For a given  Anosov flow $\Phi$ on $M$ which is $C^r$-smooth as a vector field, we denote the path component  of $\mathcal{A}^r(M)$ containing $\Phi$ by $\mathcal{A}^r(\Phi)$.

\begin{theoalph}\label{thm component}
\emph{Let $r\geq 1$ and $M$ be a  closed $3$-dimensional manifold supporting a transitive  Anosov flow $\Phi$. Then,  $\mathcal{A}^r(\Phi)$  is homotopy equivalent to ${\rm Diff}_0^r(M)$.}
\end{theoalph}

We note that there exists a closed 3-dimensional manifold supporting both  transitive and nontransitive Anosov flows \cite{BBY2017}.  As special cases, we consider  $3$-dimensional Seifert manifolds,  solvmanifolds and hyperbolic manifolds.  We denote $M_1$ being homotopy equivalent to $M_2$,  by $M_1\simeq M_2$

\begin{coralph}\label{cor seifert solvable hyperbolic}
	\emph{Let $M$ be a $3$-dimensional closed manifold admitting a transitive Anosov flow $\Phi$. 
		\begin{enumerate}
			\item If $M$ is a Seifert manifold, then the whole space $\mathcal{A}^r(M)$ is homotopy equivalent to a circle, i.e., 
			\[\mathcal{A}^r(M)\simeq \mathbb{S}^1.  \] 
			\item  If $M$ is a solvmanifold, then the whole space $\mathcal{A}^r(M)$ is homotopy equivalent to two points, i.e.,
				\[\mathcal{A}^r(M)\simeq \big\{p, q\big\}.  \] 
			\item  If $M$ is a hyperbolic manifold, then the path component $\mathcal{A}^r(\Phi)$ is contractible, i.e., 
				\[\mathcal{A}^r(\Phi)\simeq \big\{{\rm Id}_M \big\}.  \] 
		\end{enumerate}
		}
\end{coralph}

\begin{remark}\label{rmk isom group}
	For closed $3$-dimensional manifold $M$, by the  resolution of the Generalized Smale conjecture, $\text{Diff}^r(M)$ is homotopy equivalent to  the isometry group $\text{Isom}(M)$, see for example \cite{B2023}. Explicitly, 
	\begin{enumerate}
		\item  For the case of Seifert manifolds, Ghys showed that $M$ is a finite cover of the unit tangent bundle of a hyperbolic surface (thus, Isom$_0(M)\cong\mathbb{S}^1$) and $\Phi$ is orbit-equivalent  to the geodesic flow of $M$ via an orbit-equivalence homotopic to identity \cite{G1984,B2005}. Thus,  Theorem \ref{thm connect} implies that the whole space $\mathcal{A}^r(M)$ is path-connected.  Then, the Seifert case of Corollary \ref{cor seifert solvable hyperbolic} follows from Theorem \ref{thm component}.
		\item  For the case  of solvmanifolds, Plante showed that $\Phi$ is orbit-equivalent to the suspension of an hyperbolic automorphism $A$  of $\TT^2$  \cite{P1981}. Hence,  $M=M_A$ is the mapping torus of $A$, $\pi_1(M_A)\cong \ZZ^2\rtimes_A\ZZ$ and Isom$_0(M_A)\cong \{{\rm Id}_M\}$.  In particular, let $\Psi$ be  any Anosov flow  on  $M_A$. By Plante's work again,   $\Psi$ is orbit-equivalent to  some suspension flow of hyperbolic automorphism $B\in {\rm GL}(2,\ZZ)$, then $\ZZ^2\rtimes_A\ZZ\cong \ZZ^2\rtimes_B\ZZ$. This implies that $B$ is conjugate to $A$ or $A^{-1}$ in ${\rm GL}(2,\ZZ)$. Thus, $\Psi$ is orbit-equivalent to  suspensions of $A$ or $A^{-1}$ via an orbit-equivalence homotopic to identity. By the path-connectedness result (Theorem \ref{thm connect}), the whole space $\mathcal{A}^r(M)$ has exactly two path components. Thus, $\mathcal{A}^r(M)\simeq \big\{p, q\big\}$.
		\item  For the case of hyperbolic manifolds, Gabai's work \cite{G2001} and the Mostow Rigidity Theorem  imply that  Diff$^r_0(M)$ is contractible \cite{G2001}. Then,  Corollary \ref{cor seifert solvable hyperbolic} follows from Theorem \ref{thm component}, directly.
	\end{enumerate}
	 We also note here that for a general $3$-manifold $M$ admitting Anosov flows, the identity component $\text{Isom}_0(M)$ is either trivial $\{{\rm Id}_M\}$ or  $\mathbb{S}^1$.  The cases of  Seifert manifolds and  hyperbolic manifolds follow  from the above first and third items.   In the case of Haken manifolds, the Hatcher-Ivanov Theorem \cite{H1976,I1979} shows that Diff$_0(M)$ is homotopy equivalent to $\text{Isom}_0(M)$ which  is trivial  or  $\mathbb{S}^1$, since $M$ admitting Anosov flow cannot be $\TT^3$. We refer to \cite[Section 1.3]{HKMR2012} for a detailed survey on  the Smale conjecture. 
\end{remark}

\begin{remark}
		In \cite{FaG2014}, Farrell-Gogolev shows that a homotopy class of  Anosov diffeomorphisms on $\TT^2$ is homotopy equivalent to $\TT^2$.   Instead of the homotopy class, we just consider the component of $\mathcal{A}^r(M)$ in Theorem \ref{thm component}.  The main reason of this difference is that any two homotopic toral Anosov diffeomorphisms are conjugate via a conjugacy homotopic to identity, then by the path-connectedness, a homotopic class of  the space of Anosov diffeomorphisms on $\TT^2$ coincides with a connected component. However, for Anosov flow $\Phi$, its homotopic class could have many orbit-equivalence classes, e.g. \cite{B1998,BBY2017}.  We will discuss the orbit-equivalence class of $\Phi$ is Subsection \ref{subsec intro teichmuller}.
\end{remark}

\subsection{3-dimensional Anosov flows with  $C^1$ strong hyperbolic foliations}

On the rigidity  of smooth conjugacy between Anosov flows on $3$-manifolds, Gogolev and Rodriguez Hertz surprisingly give remarkable results. In \cite{GRH2022}, they show that the conjugacy of two conservative $C^r\ (r>2)$ Anosov flows on $3$-manifold is automatically smooth, unless they are suspensions over Anosov diffeomorphisms on $2$-torus $\TT^2$ with constant roof-functions. 

In  \cite{GLRH2025}, Gogolev, Leguil and Rodriguez Hertz further show that 
the local rigidity of conjugacy being automatically smooth is $C^1$-open and $C^\infty$-dense in the class of $C^{\infty}$-smooth Anosov flows on $3$-manifolds \cite[Theorem C]{GLRH2025}. The main technical result in \cite{GLRH2025} classifies Anosov flows by admitting $C^{1}$-smooth strong hyperbolic foliations or not. In the case of flows without $C^1$-smooth strong hyperbolic foliation, they show the rigidity  (see for example \cite[Theorem E, Addendum F, Theorem G]{GLRH2025}). 

We give  rigidity results on the Anosov flows admitting $C^1$-smooth strong unstable foliations.

\begin{theoalph}\label{thm flow Phi sC1 rigid}
\emph{Let $\Phi$ and $\Psi$ be two conjugate $C^r$-smooth $(r>1)$ transitive Anosov flow on $3$-manifolds with $C^{1}$-smooth strong unstable foliations. 
	Then, at least one of the following  holds:
	\begin{itemize}
		\item  Either $\Phi$ and $\Psi$ are  suspensions with constant roof-functions,
		\item or the restriction of the conjugacy on each leaf of the stable foliation  is $C^{r_*}$-smooth, where $r_*=r-1+{\rm Lipschitz}$ when $r$ is an integer, otherwise $r_*=r$.
\end{itemize} }
\end{theoalph}
\begin{remark}
We consider more rigidity on conjugacy preserving smooth one-dimensional foliations, where the foliations may be not flow-invariant, see Theorem \ref{thm C1 rigid 1} and Theorem \ref{thm C1 rigid 2}.
\end{remark}

On the opposite side of rigidity, we classify the  Anosov flows admitting $C^1$-smooth strong unstable foliations via considering the  flexibility of unstable Jacobians,  in the conjugacy class rather than only in the orbit-equivalence class as in Theorem \ref{thm flow flexibility 0}. We denote the stable/unstable Jacobian functions of Anosov flow $\Phi$ by $J^{s/u}_\Phi(x):=\frac{d}{dt}|_{t=0}J^{s/u}_\Phi(x,t)$. Two functions $f$ and $g$  is called \emph{cohomologous}, if $f-g$ is $\Phi$-coboundary, see Subsection \ref{subsec Cocycle} for precise definition.

\begin{theoalph}\label{thm flow Phi sC1 flex}
	\emph{Let $\Phi$ be a transitive $C^{1+}$-smooth Anosov flow on $3$-manifold $M$ with $C^1$-smooth strong unstable foliation. Let $f:M\to \RR\ $ be a H\"older continuous  function with topological pressures $P_{f}(\Phi)=0$.  Then there exists a $C^{1+}$-smooth Anosov flow $\Psi$ of manifold $M$ conjugate to $\Phi$ via a H\"older continuous homeomorphism $H$ such that 
		\begin{enumerate}
			\item The strong unstable foliation of\, $\Psi$ is $C^1$-smooth.
			\item  The stable Jacobian function satisfies that $J^s_\Psi\circ H^{-1}$ is cohomologous to $J^s_\Phi$.
			\item The unstable Jacobian function satisfies that  $J^u_\Psi\circ H^{-1}$ is cohomologous to $-f$.
		\end{enumerate}
		Moreover, such a $C^{1+}$-smooth Anosov flow $\Psi$ is unique, up to $C^{1+}$-smooth conjugacy.}
\end{theoalph}

\begin{remark}
	We note that if $\Phi$ is not a suspension with constant roof-function, then the first two items of Theorem \ref{thm flow Phi sC1 flex} are equivalence. The necessary part is provided by  Theorem \ref{thm flow Phi sC1 rigid}, and the sufficient part follows from the main result in de la Llave's work \cite{dL1992}, see Theorem \ref{thm delaLlave}.
\end{remark}

Combining with Theorem \ref{thm flow Phi sC1 rigid} and Theorem \ref{thm flow Phi sC1 flex}, we get the classification of Anosov flows admitting $C^{1+}$-smooth strong hyperbolic foliations in the conjugacy class. This gives partial answer to \cite[Question 2.9]{GLRH2025} of Gogolev-Leguil-Rodriguez Hertz. 

In the rest part of this subsection, we state some corollaries of Theorem \ref{thm flow Phi sC1 rigid} and Theorem \ref{thm flow Phi sC1 flex}. 
Applying Theorem \ref{thm flow Phi sC1 flex} to  the function $f\equiv-h_{\rm top}(\Phi)$, we get the next one, directly (see Remark \ref{rmk proof of SRB=MME}). 
\begin{coralph}\label{cor mme=srb}
	\emph{Let $\Phi$ be a transitive $C^{1+}$-smooth Anosov flow on $3$-manifold $M$ with $C^1$-smooth strong unstable foliation. Then there is a $C^{1+}$-smooth Anosov flow $\Psi$ conjugate to $\Phi$ such that the measure of maximal entropy coincides with the Sinai-Ruelle-Bowen measure.}
\end{coralph}

By the Sinai-Ruelle-Bowen  property \cite{BR1975}, the pressures of Anosov flow $\Phi$ with respect to its stable and unstable Jacobians, i.e., $J^s_\Phi$ and $-J^u_\Phi$, are zero. That's the reason that we consider the $0$-pressure functions in Theorem \ref{thm flow flexibility 0} and Theorem \ref{thm flow Phi sC1 flex}. We also consider general H\"older functions in Section \ref{sec Space}.

\begin{remark}
	Here, we give some notes on some questions and works about the Sinai-Ruelle-Bowen (SRB) measure equal to the measure of maximal entropy (MME) for Anosov flows.  
	\begin{itemize}
		\item  Parry  shows that  a transitive $C^2$ Anosov flow  with $C^1$ strong unstable bundle is orbit-equivalent to a $C^1$ Anosov flow with SRB=MME, by adjusting the speed of flow \cite{Pa1986}. Corollary \ref{cor mme=srb} considers this in the conjugacy class, but we essentially need one-dimensional strong unstable foliation.
		\item On Katok's Entropy Conjecture for Anosov flow on $3$-manifold,  the work \cite{DLVY2020} shows that a $C^r$-smooth $(r\geq 5)$  Anosov flow with MME=volume is smooth conjugate to an algebraic flow.  Then the authors ask in  \cite[Question 2]{DLVY2020}  that for a smooth transitive Anosov $3$-flow with SRB=MME,  is it smoothly conjugate to an algebraic flow? This is a natural question extending the Katok's Entropy Conjecture, since a conservative Anosov flow's volume measure is the unique SRB measure \cite{BR1975}.  It is clear that Corollary \ref{cor mme=srb} gives a class of counterexamples for above question in $C^{1+}$-regularity, since the contact Anosov $3$-flows always admit $C^{1}$-smooth hyperbolic splittings and form a large class \cite{FoH2013}. We also note that  \cite{A2012} shows that in the class of $C^{\infty}$-smooth locally free GA-action,  the space of Anosov  flows conjugate to an algebraic one with SRB = MME can be represented as an open subset of the set of $C^{\infty}$-smooth closed one-forms quotient  cohomology. 
	\end{itemize}
\end{remark}

Recall that a  contact Anosov flow on $3$-manifold has $C^1$-smooth hyperbolic splitting and is transitive. Hence, as a corollary of Theorem \ref{thm flow Phi sC1 rigid}, we give a new proof of the rigidity result of Feldman-Ornstein for $3$-dimensional contact Anosov flows \cite{FO1987}. Such a rigidity for contact Anosov flow is extended by Gogolev-Rodriguez Hertz  \cite{GRH2024} to higher dimensions, also see other works of Gogolev-Rodriguez Hertz  such as \cite{GRH2025} for  rigidity of Anosov diffeomorphisms with $C^1$-smooth unstable foliations.

\begin{coralph}\label{cor rigid C1su}
	\emph{Let $\Phi$ and $\Psi$ be two $C^r$-smooth $(r>1)$ contact Anosov flows on $3$-manifolds.   If $\Phi$ and $\Psi$ are conjugate, then they are $C^{r_*}$-smoothly conjugate.}
\end{coralph}

\begin{remark}\label{rmk contact 1}
	We still do not know if it is possible to conjugate a contact Anosov flow $\Phi$ to an Anosov flow $\Psi$ admitting given stable and unstable Jacobians, simultaneously.  By Theorem  \ref{thm flow Phi sC1 flex}, we  can get an Anosov flow $\Phi_1$ with desired unstable Jacobian and conjugate to $\Phi$. However, by Theorem \ref{thm flow Phi sC1 rigid}, the strong stable foliation of $\Phi_1$ is not $C^1$-smooth any more, if the unstable Jacobian is really changed. Then, we cannot apply Theorem  \ref{thm flow Phi sC1 flex} again to the strong stable case of $\Phi_1$.
\end{remark}

In particular, we  get the following dichotomy.

\begin{coralph}\label{cor dicho}
\emph{	Let $\Phi$ and $\Psi$ be two conjugate $C^r$-smooth $(r>1)$ Anosov flows on $3$-manifolds with $C^{1}$-smooth strong hyperbolic  foliations. 	Then,  at least one of the following  holds:
	\begin{itemize}
		\item  Either $\Phi$ and $\Psi$ are  suspensions with constant roof-functions,
		\item  or $\Phi$ is $C^{r_*}$-smoothly conjugate to $\Psi$.
	\end{itemize}}
\end{coralph}

Corollary \ref{cor dicho} can be viewed as an adaption  of the main result of \cite{GRH2022} which we mentioned at the beginning of this subsection,  in lower regularity assumption, without conservative condition, but under additional hypothesis on strong hyperbolic foliations. 

Our proof is totally different from \cite{GRH2022}. In particular, we do not use the adapted transverse coordinates introduced by Hurder-Katok \cite{HK1990} or the adapted chart given by Tsujii \cite{T2018}  for $3$-dimensional conservative Anosov flow which is   developed by  Tsujii-Zhang \cite{TZ2023}  for dissipative $3$-dimensional Anosov flows and by Eskin-Potrie-Zhang \cite{EPZ2023} for  $3$-dimensional partially hyperbolic diffeomorphisms. 
See also Gogolev-Rodriguez Hertz \cite{GRH2022} and Gogolev-Leguil-Rodriguez Hertz \cite{GLRH2025} where the regularity of Anosov flows need to be no less than $2$.

\subsection{The Teichm\"uller space of Anosov $3$-flow}\label{subsec intro teichmuller}

As a summary of the previous two subsections, we give the Teichm\"uller space of an Anosov flow on $3$-manifold in its orbit-equivalence class, or  in its conjugacy class if we assume $C^1$-smoothness of strong hyperbolic foliations. 

Let  $\Phi$ be a $C^r$-smooth ($r\geq 1$) transitive Anosov flow of $3$-manifold $M$. We denote the orbit-equivalence class of $\Phi$ by $\mathcal{O}^{r}(\Phi)$, namely, 
\[ \mathcal{O}^{r}(\Phi):=\Big\{ \Psi\ | \ C^{r}\  \text{Anosov flow on}\  M\  \text{which is orbit-equivalent to  }\Phi  \Big\}.\]
For two flows $\Psi_1$ and $\Psi_2$ in $\cO^{r}(\Phi)$, we denote the equivalence relation $\Psi_1$ being $C^{r_*}$-smooth orbit-equivalent  to $\Psi_2$  by \[\Psi_1\sim^o \Psi_2.\]

Let $\mathbb{F}^{\rm H}(M)$ be the set of  H\"older continuous  functions on $M$. For functions $f_1, f_2\in \mathbb{F}^{\rm H}(M)$, denote by \[f_1\sim_\Phi f_2,\] for   flow $\Phi$ on $M$, if $f_1-f_2$ is almost $\Phi$-coboundary, namely, there is a constant $P$ and a H\"older continuous function $\beta$ such that
\[ \int_0^t (f_1-f_2-P)\circ \Phi^\tau(x)d\tau =\beta(x)-\beta\circ \Phi^t(x),\qquad \forall x\in M, \  \forall t\in\RR.  \]

By Theorem \ref{thm flow flexibility 0}, we will get the Teichm\"uller space of an Anosov $3$-flow in the orbit-equivalence class.
\begin{coralph}\label{cor Teich space orbit-eq}
	\emph{Let $\Phi$ be a $C^{1+}$ transitive Anosov flow on $3$-manifold $M$. Then there is a bijection 
	\begin{align*}
		\mathcal{O}^{1+}(\Phi)/_{\sim^o} \to \mathbb{F}^{\rm H}(M)/_{\sim_\Phi}\times  \mathbb{F}^{\rm H}(M)/_{\sim_\Phi}.
	\end{align*}}
\end{coralph}

\begin{remark}
	Let $\cO^r_0(\Phi)$ be the identity component of $\cO^r(\Phi)$, i.e., the set of Anosov flows orbit-equivalent to $\Phi$ via an orbit-equivalence homotopic to Id$_M$.   By the path-connectedness  (Theorem \ref{thm connect}) and the structural stability,  $\cO^r_0(\Phi)$  coincides with the path component $\mathcal{A}^r(\Phi)$, if we just consider  $\cO^r_0(\Phi)$ consisting of $C^r$-smooth vector fields. Hence, $\cO^r_0(\Phi)$ is also homotopy equivalent to Diff$^r_0(M)$. 
\end{remark}

We consider the conjugacy class of Anosov flows with $C^1$-smooth strong hyperbolic foliations, 
\begin{align*}
	\mathcal{H}^{1+}_{s,u}(\Phi):=\Big\{ \Psi \ | \ C^{1+}\ \text{smooth} &\text{ Anosov flow on} \ M\  \text{conjugate to } \Phi   \\ & \text{with}\ 
	C^1 \ \text{strong stable and strong unstable  foliations}  \Big\}.  
\end{align*}
We denote  $\Psi_1$ being smoothly  conjugate to $\Psi_2$  by \[\Psi_1\sim \Psi_2.\]

Similar to the flow case,  for a  map $A:M\to M$, we define
\[f_1\sim_A f_2,\] 
if $f_1-f_2$ is almost $A$-coboundary, namely, there is a constant $P$ and a H\"older continuous function $\beta$ such that $f_1(x)-f_2(x)-P=\beta(x)-\beta\circ A(x)$,  for all $x\in M$.

Corollary \ref{cor dicho} and the work  of Cawley \cite{C1993} imply  the following characterization  of the space $\mathcal{H}^{1+}_{s,u}(\Phi)$.

\begin{coralph}\label{cor Techmuller space suC1}
	\emph{Let $\Phi$ be a $C^{1+}$-smooth transitive Anosov flow on $3$-manifold $M$ with $C^1$-smooth strong stable and strong unstable foliations. Then there is a dichotomy:
	\begin{itemize}
		\item either $\Phi$ is a constant-roof suspension over $A:\TT^2\to \TT^2$ and  there is a nature bijection: \[\mathcal{H}^{1+}_{s,u}(\Phi)/_{\sim} \to \mathbb{F}^{\rm H}(\TT^2)/_{\sim_A}\times  \mathbb{F}^{\rm H}(\TT^2)/_{\sim_A}, \]
		\item or the space $\mathcal{H}^{1+}_{s,u}(\Phi)/_{\sim}$ is trivial, namely, it has just an element.
	\end{itemize}}
\end{coralph}

We will also discuss the space of transitive Anosov $3$-flows just admitting smooth strong unstable foliations, see Corollary \ref{cor teichmulle uC1}.

\vskip 0.5 \baselineskip

\noindent {\bf Organization of this paper:}
In section \ref{sec Preliminaries}, we recall some general properties of Anosov flows on $3$-manifolds including the Radon-Nikodym realization. In section \ref{sec HA-flow}, we introduce H\"older  flows admitting $C^{1+}$-smooth weak stable and weak unstable foliations which we called HA-flows. We will show the flexibility of the Jacobians induced by these foliations for HA-flows. 
In section \ref{sec Space}, we smoothing the  above HA-flows to get  Theorem \ref{thm flow flexibility 0} the Jacobian flexibility result of transitive Anosov $3$-flows and  Corollary \ref{cor Teich space orbit-eq} the Teichm\"uller space in the sense of orbit-equivalence. Then we  prove  the path-connectedness of the the space of Anosov flows (Theorem \ref{thm connect})  and the homotopy type of the space (Theorem \ref{thm component}).   In section \ref{sec C1 Foliation}, we consider Anosov $3$-flows admitting $C^1$-smooth strong hyperbolic foliations, including the classification Theorem \ref{thm flow Phi sC1 flex}, Corollary \ref{cor Techmuller space suC1} and the rigidity Theorem \ref{thm flow Phi sC1 rigid}.

\vskip 0.5 \baselineskip

\section{Preliminaries}\label{sec Preliminaries}

In this section, we discuss the regularity of foliations and flows, and recall basic properties of Anosov flows including the Radon-Nikodym realization, i.e., the family of transversal invariant measures for Anosov-flows.

\subsection{Foliations and flows}

Recall that a partition $\cF$ of a closed Riemannian $d$-manifold $M$ is a $C^r$-smooth $l$-dimensional foliation, if there exists a local chart $\{ (\vphi_i,U_i) \}_{1\leq i\leq k}$, called \emph{foliation chart}, such that 
\begin{itemize}
	\item The map $\vphi_i:D^l\times D^{d-l}\to U_i$ has $\vphi_i\big(D^l\times \{y\} \big)\subset \cF\big(\vphi_i(0,y)\big)$, for all $y\in D^{d-l}$, where $D^l$ and $D^{d-l}$ are open $l$-disk and $(d-l)$-disk of $\RR^d$, respectively.
	\item For $U_i\cap U_j\neq \emptyset$, the map $(\vphi_j^{-1}\circ \vphi_i)|_{\vphi_i^{-1}(U_i\cap U_j)}$ is $C^r$-smooth.
\end{itemize}
Particularly, for  each $p\in M$, the local leaf $\cF_{\rm loc}(p)$ is a $C^r$-smooth embedded submanifold of $M$ and  the whole leaf $\cF(p)$ is a $C^r$-smooth immersed submanifold of $M$.  For a  subset $U$ of $M$ and a point $p\in U$, we denote by 
\[ \cF(p, U):=\  \text{the component of}\  \cF(p)\cap U \ \text{containing the point} \ p. \] 
We call $U$ an $\cF$-\emph{foliation box}, if \[U=\bigcup_{x\in \Sigma}\cF_{\rm loc}(x),\] where $\Sigma\subset U$ is $(d-l)$-dimensional submanifold of $M$ transverse to $\cF$.

Let $\Sigma_1$ and $\Sigma_2$ be two smooth transversals of foliation $\cF$ close enough. One may define the holonomy map induced by $\cF$  as 
\[{\rm Hol}^{\cF}:\Sigma_1\to \Sigma_2,\quad  {\rm Hol}^{\cF}(x)= \cF_{\rm loc}(x)\cap \Sigma_2, \]
which is actually a homeomorphism from $\Sigma_1$ to $\Sigma_2$. 

We denote the homeomorphism $H$ mapping foliation $\cF$ to be foliation $\cL$ by $\cL=H(\cF)$,  specifically, $\cL(x)=H(\cF(H^{-1}(x)))$. The following Journ\'e Lemma is  useful.

\begin{lemma}[Journ\'e Lemma \cite{J1988}]\label{lem journe}
	Let $\cF_1$ and $\cF_2$ be transverse foliations of $d$-manifold $M$ with dimension $d_1$ and $d_2$, respectively, where $d_1+d_2=d$. Let $\cL_1$ and $\cL_2$ be also transverse foliations of $M$ with dimension $d_1$ and $d_2$, respectively.  Assume that  a homeomorphism $H:M\to M$ satisfies that for each $x\in M$ and $i=1,2$, $H(\cF_i)=\cL_i$ and the restrictions $H|_{\cF_i(x)}:\cF_i(x)\to \cL_i(H(x))$ are uniformly (with respect to $x$) $C^r$-smooth, then $H$ is a $C^{r_*}$-smooth diffeomorphism, where $r_*=r$ if $r\notin \NN$, and $r_*=r-1+{\rm Lipschtiz}$  if $r\in \NN$.  	
\end{lemma}

Recall that $\Phi$ is a $C^r$-smooth flow on $M$, if we view it as a  $C^r$-smooth map 
\[ \Phi: M\times \RR \to M,\quad (x,t)\mapsto\Phi^t(x).\]  
The orbit foliation of $\Phi$ is denoted by $\cO_\Phi$, namely, each leaf $\cO_\Phi(x)$ is $\big\{\Phi^t(x)\big\}_{t\in\RR}$, the orbit of $x\in M$.  In this paper, we always assume that the flow has no singularity. Then, every orbit $\cO_\Phi(x)$ is homeomorphic to $\RR^1$ or $\mathbb{S}^1$ (periodic orbit).  It is clear that  $\cO_\Phi$ is a $C^r$-smooth foliation.

\subsubsection{Regularity of foliations and flows}\label{subsec regularity}

In this paper, we will consider three types of regularity of foliation given by  foliation chart, by leaves and holonomy maps, and by tangent  plane field. It is clear that a $C^r$-foliation has $C^r$-smooth leaves and holonomy maps. Applying the Journ\'e Lemma, the regularity of foliation is almost decided  by ones of its leaves and holonomy maps, see for example  \cite[Section 6]{PSW1997}.  A foliation generated by $C^r$ plane field is  automatically $C^r$-smooth. The converse is false, in general. However, a $C^r$-smooth foliation can be $C^r$-smoothly diffeomorphic  to a foliation generated by a $C^r$ plane field \cite{H1983}. Here we collect the above facts and properties as the following proposition. 

\begin{proposition}[\cite{H1983,PSW1997} ]\label{prop foliation regularity}
	
	Let $\cF$ be a foliation of $M$. Then 
	\begin{itemize}
		\item If $\cF$ is generated by a $C^r$-smooth plane field, then $\cF$ is a $C^r$-smooth foliation.
		\item If $\cF$ is a $C^r$-smooth foliation, then there exists a $C^r$-diffeomorphism $H:M\to M$ such that the foliation $H(\cF)$ is generated by a $C^r$-smooth plane field and $H$ can be chosen $C^r$-close to  ${\rm Id}_M$.
		\item If $\cF$ is a $C^r$-smooth foliation, then the leaves and the holonomy maps are uniformly $C^r$-smooth.
		\item If the leaves of $\cF$ and the holonomy maps induced by $\cF$ are uniformly $C^r$-smooth, then the foliation $\cF$ is $C^{r_*}$-smooth, where $r_*$ is defined as Lemma \ref{lem journe}.
	\end{itemize}
\end{proposition}

Applying the second item of  Proposition \ref{prop foliation regularity} to a flow $\Phi$ with $C^r$-smooth orbit foliation on $M$,   there is a $C^r$-smooth diffeomorphism $H:M \to M$ such that  $DH(T\cF)$ is a $C^r$-smooth subbundle of $TM$. Obviously,  the unit bundle with respect to a $C^r$-smooth metric $a(\cdot,\cdot)$, of $DH(T\cF)$  induces a $C^r$-flow $\Psi$ on $M$, and $H$ is a $C^r$-smooth orbit-equivalence between $\Phi$ and $\Psi$. Moreover, Proposition \ref{prop foliation regularity} also implies that 
a $C^r$-smooth flow can be $C^r$-smoothly conjugate to a flow generated by $C^r$-smooth vector field. For convenience, we state these as follow.

\begin{proposition}[\cite{A2008,H1983}]\label{prop flow regularity}
	Let $\Phi$ be a flow of $M$ without singularity. Then
	\begin{itemize}
		\item If the orbit foliation  $\cO_\Phi$ is $C^r$-smooth, then there exist a $C^r$-diffeomorphism $H:M\to M$  and a flow $\Psi$ of $M$ generated by $C^r$-vector field such that $\Phi$ is orbit-equivalent to $\Psi$ via $H$.
		\item If $\Phi$ is a $C^r$-smooth flow, then there exists a $C^r$-smooth diffeomorphism $H:M\to M$ such that the flow $\Psi^t(x):=H\circ \Phi^t\circ H^{-1}(x)$ is generated by $C^r$-vector field.
	\end{itemize}
	In both cases, $H$ can be arbitrarily $C^r$-close to  the identity map ${\rm Id}_M$.
\end{proposition}

\subsubsection{Cocycles over flows}\label{subsec Cocycle}

Let $\Phi:M\to M$ be a $C^r$-smooth $(r\geq0)$  flow. A $C^k$-smooth $(0\leq k\leq r)$ function $\alpha:M\times \RR \to \RR$ is called a $C^k$-smooth \emph{cocycle} over flow $\Phi$, if $\alpha(x,0)=0$ and \[\alpha(x,t_1+t_2)=\alpha(x,t_1)+\alpha(\Phi^{t_1}(x),t_2),\quad \forall x\in M \ {\rm and}\ \forall t_1,t_2\in \RR. \]
We call two cocycles $\alpha_1$ and $\alpha_2$ over $\Phi$ are $C^k$-smoothly \emph{cohomologous}, if there is a $C^k$-smooth function $\beta:M\to \RR$ such that 
\[ \alpha_1(x,t)=\alpha_2(x,t)+\beta(x)-\beta\circ \Phi^t(x), \quad \forall x\in M \ {\rm and}\ \forall t\in \RR.  \]
In particular, for a $C^k$ function $f:M\to \RR$, the function
\[\alpha_\Phi(x,t,f):=\int_0^tf\circ \Phi^\tau(x)d\tau \]
is a $C^k$-smooth cocycle over $\Phi$. 
We call two function $f_1, f_2:M\to \RR$ are $C^k$-smoothly \emph{cohomologous} (with respect to flow $\Phi$), if the cocycles $\alpha_{\Phi}(x,t,f_1)$  and $\alpha_{\Phi}(x,t,f_2)$ are  $C^k$-smoothly cohomologous.

 Let  $\cF$ be a $ C^{k+1}$-smooth $(k\geq 0)$ codimension-one foliation of $M$ with  tangent bundle $T\cF$.  Assume that $\cF$ is  subfoliated by $\cO_\Phi$ the orbit foliation  of $\Phi$.  Let \[TM=T\cF\oplus E\] be a direct sum splitting, where $E$ is a one-dimensional bundle transverse to $T\cF$. Since $\cF$ is smooth, its holonomy ${\rm Hol}^{\cF}$ naturally induces the map $D{\rm Hol}^{\cF}_{x,y}:E(x)\to T_yM$. Let $\pi:TM=T\cF\oplus E\to E$ be the natural projection, and
\[D{\rm Hol}^{\cF,E}_{x,y}: E(x)\to E(y),\]
by $D{\rm Hol}^{\cF,E}_{x,y}(v)=\pi\circ D{\rm Hol}^{\cF,E}_{x,y}(v)$.
Particularly, the holonomy of $\cF$ induces a flow $\Phi_\cF$ of $E$, given by 
\[\Phi^t_\cF(x,v)=\big(\Phi^t(x), D{\rm Hol}^{\cF,E}_{x,\Phi^t(x)}(v) \big), \quad \forall (x,v)\in E, \forall t\in \RR.  \]
For  a Riemannian metric $a(\cdot,\cdot)$ of $M$, we denote \[\alpha_\Phi(x,t,\cF,E,a):=\log \|D{\rm Hol}^{\cF,E}_{x,\Phi^t(p)}\|_a.\]
It is clear that the map $\alpha_\Phi(\cdot,\cdot,\cF,E,a):M\times \RR\to \RR$ is a cocycle over the flow $\Phi$, namely, 
\[ \alpha_\Phi(x, t_1+t_2, \cF,E,a)=\alpha_\Phi(x, t_1, \cF,E,a)+\alpha_\Phi(\Phi^{t_1}(x), t_2, \cF, E,a). \]

In the following of this paper, we will actually focus on the cohomologous class of the cocycle $\alpha_\Phi(x,t,\cF,E,a)$.  The next lemma claims that different metric $a'(\cdot,\cdot)$ and transversal $E'$ give a cocycle $\alpha_\Phi(x,t,\cF,E',a')$ cohomologous to $\alpha_\Phi(x,t,\cF,E,a)$.

\begin{lemma}\label{lem metric to change bundle}
	Let $M$ be a Riemannian manifold with $C^{k}$-smooth $(k\geq0)$ metrics $a_1(\cdot,\cdot)$ and $a_2(\cdot,\cdot)$. Let $\cF$ be a codimension-one  foliation generated by $C^k$-smooth plane field and  subfoliated by the orbit foliation of a flow $\Phi$ on $M$.   Let $E_1, E_2$ be one-dimensional $C^k$-smooth  bundles transverse to $T\cF$.  Then, 
	\begin{enumerate}
		\item The cocycles $\alpha_\Phi(x,t,\cF,E_1,a_1)$ and $\alpha_\Phi(x,t,\cF,E_2,a_2)$ are $C^k$-smoothly cohomologous;
		\item If $\alpha_\Phi(x,t,\cF,E_1,a_1)$ is $C^k$-smoothly cohomologous  to cocycle $\alpha(x,t)$ over flow $\Phi$.  Then, there is a $C^k$-smooth metric $a(\cdot,\cdot)$ such that $(E_1)^\perp_{a_1}=(E_1)^\perp_{a}$, $a|_{(E_1)^\perp_{a}}=a_1|_{(E_1)^\perp_{a_1}}$  and 
		\[ \alpha_\Phi(x, t, \cF,E_1,a_1) = \alpha(x,t),\]
		where $(E_1)^\perp_a$ is the orthogonal complement of $E_1$ with respect to metric $a$.
	\end{enumerate}
\end{lemma}

\begin{proof}
	
	It is clear that given a $C^k$-smooth transversal $E_0$ of $T\cF$,  different metrics $a_i(\cdot,\cdot)\ (i=1,2)$ give the same $C^k$-cohomologous class of cocycles. Indeed, let $v_x\in E_0$ such that $\|v_x\|_{a_1}=1$, and $\gamma(x)=\log\|v_x\|_{a_2}$. Then, for all $x\in M$ and $t\in \RR$,
	\begin{align}
			 \alpha_{\Phi}(x,t,\cF,E_0,a_1)- \alpha_{\Phi}(x,t,\cF,E_0,a_2)=\gamma(x)-\gamma\circ\Phi^t(x).  \label{eq. coho metric}
	\end{align}

	Let $E$  be the bundle orthogonal to $F$ with respect to $a_1(\cdot,\cdot)$.  Since $F$ and $a_1(\cdot,\cdot)$ are $C^k$-smooth, so is  $E$.  Let $\theta_1(x):=\angle_{a_1}\big(E(x),E_1(x)\big)$.  Then by direct computation,
		\[ \alpha_{\Phi}(x,t,\cF,E_1,a_1)= \alpha_{\Phi}(x,t,\cF,E,a_1)+\log \big(\cos \theta_1(x)\big)- \log \big(\cos \theta_1(\Phi^t(x))\big).\]
	The above two cohomologous equations implies that the following $C^k$-smoothly cohomologous cocycles chain: the cocycles given by $(E_2, a_2)$, $(E_2,a_1)$,  $(E,a_1)$ and  $(E_1,a_1)$.

  Conversely, 	if $\alpha_\Phi(x,t,\cF,E_1,a_1)=\alpha(x,t)+\beta(x)-\beta\circ \Phi^t(x)$, for some $C^k$-smooth function $\beta:M\to \RR$.  
  Let $a(\cdot,\cdot)$ be the $C^k$-smooth metric such that $E_1$ and $(E_1)^\perp_{a_1}$ are still orthogonal  and 
	\[a|_{(E_1)^\perp_{a}}=a_1|_{(E_1)^\perp_{a_1}} \quad {\rm and} \quad a|_{E}=e^{\beta}\cdot a_1|_{E}.\]
	The by \eqref{eq. coho metric}, $\alpha_{\Phi}(x,t,\cF,E_1,a) =\alpha(x,t)$.
\end{proof}

\begin{remark}\label{rmk Holder metric}
	In this paper, we will focus on  $C^{1+}$-smooth foliation $\cF$,  H\"older continuous bundle $E$ transverse to $T\cF$, H\"older continuous metric $a(\cdot,\cdot)$ and the H\"older-cohomologous class of the cocycle given by $(\cF, E, a)$.  Hence, by Lemma \ref{lem metric to change bundle}, one can define the H\"older-cohomologous class,
	\[<\alpha_\Phi(x,t,\cF)>:=<\alpha_\Phi(x,t,\cF,E,a)>.\]
	When we just concern the metric, we denote 
	\begin{align}
		\alpha_\Phi(x,t,\cF,a):=\alpha_\Phi(x,t,\cF,(T\cF)_a^\perp ,a), \label{eq. cocycle no E}
	\end{align}
	where $(T\cF)_a^\perp$ is the orthogonal complement of the tangent bundle $T\cF$ of $\cF$ with respect to metric $a$. 
\end{remark}

\subsection{Anosov flows}

Let $\Phi$ be a $C^r$-smooth $(r\geq 1)$ Anosov flow on Riemannian manifold $(M,a)$.
It is known that the subbundles $E^\sigma_{\Phi} \ (\sigma=s/ss/u/uu)$ are all H\"older continuous. 
Recall that these bundles uniquely integrate to $\Phi$-invariant foliations, called the stable/strong stable/unstable bundle/strong unstable foliation of $\Phi$ denoted by  $\cF^s_{\Phi}/\cF^{ss}_{\Phi}/\cF^u_{\Phi}/\cF^{uu}_{\Phi}$, respectively. 
It is clear that the foliations $\cF^{ss/uu}_{\Phi}$ and $\cO_{\Phi}$ subfoliate each leaf of $\cF^{s/u}_{\Phi}$. 
We define the Jacobians on the strong stable/unstable bundle of a $C^{1+}$-smooth Anosov flow $\Phi$ respectively  by
\[J^s_\Phi(x,t)=\log \Big|{\rm det}_a\big(D\Phi^t|_{E^{ss}_\Phi(x)}\big)\Big|
\quad {\rm and}\quad J^u_\Phi(x,t)=\log \Big|{\rm det}_a\big(D\Phi^t|_{E^{uu}_\Phi(x)}\big)\Big|.\]
Let 
\[J^u_{\Phi}(x):=\frac{d}{dt}|_{t=0}\log \Big|{\rm det}_a\big(D\Phi^t|_{E^{uu}_\Phi(x)}\big)\Big|\quad  {\rm and}\quad  J^s_{\Phi}(x):=\frac{d}{dt}|_{t=0}\log \Big|{\rm det}_a\big(D\Phi^t|_{E^{ss}_\Phi(x)}\big)\Big|. \]
Then
\[ J^s_\Phi(x,t)=\alpha_{\Phi}(x,t,J^s_\Phi) \quad {\rm and} \quad  J^u_\Phi(x,t)=\alpha_{\Phi}(x,t,J^u_\Phi), \]
are cocycles over flow $\Phi$. Moreover, if $\Phi$ is $C^{1+}$-smooth, the cocycles $J^s_\Phi(x,t)$ and $J^u_\Phi(x,t)$ are H\"older continuous.
For short, for a periodic point $p\in {\rm Per}(\Phi)$, we denote 
\[ J^{u/s}(p,\Phi):=J^{u/s}_\Phi\big(p,\tau(p,\Phi)\big)=\int_0^{\tau(p,\Phi)}J^{u/s}_\Phi\circ \Phi^\tau(x)d\tau. \] 
In the following of this paper, these notations are easy to distinguish according to the text content.

In the rest part of present paper, we focus on Anosov flows on $3$-dimensional manifold $M$. A series of works show that the homotopy on such a manifold $M$ implies isotopy. Indeed, it is known that $M$ is irreducible and covered by $\RR^3$, see for example \cite[Corollary 5]{Pal1978}.  Then, by the works of Gabai-Meyerhoff-Thurston for hyperbolic manifolds case\cite{GMT2003}, Boileau-Otal for Seifert manifolds case \cite{BO1991} and Waldhausen for Haken manifolds case \cite{W1968}, we have the following  isotopy theorem.

\begin{theorem}[Isotopy Theorem]\label{thm Isotopy}
	Let $M$ be a smooth closed $3$-manifold supported an Anosov flow. Let $H:M\to M$ be a diffeomorphism homotopic to ${\rm Id}_M$ the Identity map of $M$. Then $H$ is isotopic to ${\rm Id}_M$ .
\end{theorem}

\subsubsection{The foliations of Anosov flows}
Let $\Phi$ be a $C^{1+}$-smooth Anosov flow of $3$-manifold $(M,a)$.  We discuss some useful  properties on  $C^{1+}$-smooth foliations in this subsection. Firstly, we collect some well-known results as follow.

\begin{proposition}\label{prop foliation C1+}
	Let $\Phi$ be a $C^{1+}$-smooth Anosov flow on $3$-manifold. Then
	\begin{enumerate}
		\item {\rm \cite[Corollary 19.1.11]{KH1995} and \cite[Corollary 9.4.11]{FiH2019}}  The foliations $\cF^u_{\Phi}$ and $\cF^s_{\Phi}$ are $C^{1+}$-smooth. 
		\item {\em \cite[Theorem 7.1]{P2004}} The foliation $\cF^{u}_{\Phi}$ is $C^{1+}$-smoothly subfoliated by $\cF^{uu}_{\Phi}$. Similarly, it holds for $\cF^s_{\Phi}$ and $\cF^{ss}_{\Phi}$.
	\end{enumerate}  	
\end{proposition}

By Proposition \ref{prop foliation C1+}, it is clear that \[ J^{s/u}_\Phi(x,t)=\alpha_\Phi(x,t,\cF^{u/s}_\Phi,E^{ss/uu}_\Phi,a),\]
are H\"older cocycles over $\Phi$. In the rest,  for a foliation $\cF^{u/s}_\Phi$, we will only consider the induced cocycles over flow $\Phi$ rather than other flows. Thus,  for short, we omit the subscripts of the cocycles 
\begin{align}
	\alpha(x,t,\cF^{u/s}_\Phi,E,a):=\alpha_\Phi(x,t,\cF^{u/s}_\Phi,E,a).\label{eq. cocycle no flow}
\end{align}

Although our main rigidity theorems are under the assumption of $C^1$-smooth strong hyperbolic foliations, we will actually deal with the $C^{1+}$-regularity case. Indeed, $C^1$-smooth strong hyperbolic foliations is automatically $C^{1+}$-smooth.

 \begin{proposition}\label{prop foliation C1 to C1+}
	Let $\Phi$ be a $C^{1+}$-smooth Anosov flow on $3$-manifold $M$. If the strong stable foliation $\cF_\Phi^{ss}$ is $C^1$-smooth, then it is $C^{1+}$-smooth.
\end{proposition}
\begin{proof}
	
	Let $x,y\in M$ close enough. Since $\cF^{ss}_\Phi$ has $C^1$-smooth holonomy map ${\rm Hol}^{ss}_{y,x}:\cF^u_\Phi(y)\to \cF^u_\Phi(x)$, for each curve $[z_1,z_2]$ lying on $\cF^{uu}_\Phi(z_1)$ with endpoints $z_1,z_2\in \cF^u_\Phi(y)$, one has that ${\rm Hol}^{ss}_{y,x}\big( [z_1,z_2] \big)$ is a $C^1$-smooth curve lying on $\cF^u_\Phi(x)$, generally it is not a local strong unstable manifold. We denote the curve by $\cF(w_1):={\rm Hol}^{ss}_{p,q}\big( [z_1,z_2] \big)$, where $w_i={\rm Hol}^{ss}_{y,x}(z_i)$ for $i=1,2$. 
	\begin{lemma}\label{lem C1 to C1+}
		The curve $\cF(w_1)$ is actually a $C^{1+}$-smooth submanifold.
	\end{lemma}
	\begin{proof}[Proof of Lemma \ref{lem C1 to C1+}]
		
		By the $C^1$-smooth holonomy map ${\rm Hol}^s_{y,x}$, there exist constants $\e_0,\theta_0>0$ and $C_1>1$ depending on $\Phi$ only,  such that 	if $d^{uu}(z_1,z_2)\leq \e_0$ and $d^{ss}(z_1,w_1)\leq \e_0$, one has
		\[ \angle \big( E^{uu}_\Phi(w), T_w \cF(w_1)  \big) \leq \theta_0,\quad \forall w\in \cF(w_1),\]
		and by the $C^1$-smooth holonomy map, \[C_1^{-1}\cdot d^{uu}(z_1,z_2) \leq d(w_1,w_2)\leq C_1\cdot d^{uu}(z_1,z_2),\]
		where $d^{uu}(\cdot,\cdot)$ and $d^{ss}(\cdot,\cdot)$ are the distance induced by the metric restricted on the leaves of $\cF^{uu}_\Phi$ and $\cF^{ss}_\Phi$, respectively. We consider points $y_1\in \cF^u_\Phi(y)$, $y_2\in\cF^{uu}_\Phi(y_1)$ and points $x_i={\rm Hol}^{ss}_{y,x}(y_i)$ for $i=1,2$ such that $d^{uu}(y_1,y_2)\ll\e_0$ and $d^{ss}(y_1,x_1)\leq \e_0$. We denote $d=d^{uu}(y_1,y_2)$ and
		\[\lambda_+:=\sup_{z\in M} \frac{d}{dt}\Big|_{t=0} \log \| D\Phi^t|_{E^{uu}_\Phi(z)} \|\quad {\rm and} \quad \lambda_-:=\inf_{z\in M} \frac{d}{dt}\Big|_{t=0} \log \| D\Phi^t|_{E^{uu}_\Phi(z)} \|. \] 
		Then the time $t^*>0$ such that $d^{uu}\big(\Phi^{t^*}(y_1),\Phi^{t^*}(y_2)\big)=\e_0$ satisfies that 
		\[ t^*\geq \lambda_+^{-1}\cdot \ln \frac{\e_0}{d}.  \] 
		Let $z_i=\Phi^{t^*}(y_i)$ and $w_i={\rm Hol}^{ss}_{p,q}(z_i)$, for $i=1,2$. Note that $w_i=\Phi^{t^*}(x_i)$, for $i=1,2$, and
		\begin{align*}
			\angle\big( E^{uu}_\Phi(x_i), T_{x_i}\cF(x_1)  \big) &= \angle\Big( D\Phi^{t^*}\big(E^{uu}_\Phi(w_i)\big), D\Phi^{t^*}  \big(T_{w_i}\cF(w_1) \big) \Big)\\&\leq C_2 \cdot e^{-\lambda_- \cdot t^*} \angle \big( E^{uu}_\Phi(w), T_w \cF(w_1)  \big)\\
			&\leq C_2 \cdot e^{-\lambda_- \cdot t^*} \cdot \theta_0.
		\end{align*}
		for some constant $C_2$ depending only on $\theta_0$ and $\Phi$. Combining the last two formulas, we have
		\[ \angle\big( E^{uu}_\Phi(x_i), T_{x_i}\cF(x_1)  \big) \leq C_2\cdot \theta_0\cdot \e_0^{-\frac{\lambda_-}{\lambda_+}} \cdot d^{\frac{\lambda_-}{\lambda_+}}\leq  C_2\cdot \theta_0\cdot \e_0^{-\frac{\lambda_-}{\lambda_+}} \cdot C_1^{\frac{\lambda_-}{\lambda_+}}\cdot d^{\frac{\lambda_-}{\lambda_+}}(x_1,x_2).   \]
		Let $C_3=C_2\cdot \theta_0\cdot \e_0^{-\frac{\lambda_-}{\lambda_+}} \cdot C_1^{\frac{\lambda_-}{\lambda_+}}$ and $0<\alpha_1=\frac{\lambda_-}{\lambda_+}\leq 1$. We have
		\[ \angle\big( E^{uu}_\Phi(x_i), T_{x_i}\cF(x_1)  \big) \leq C_3\cdot d^{\alpha_1}(x_1,x_2),\quad \forall i=1,2.\]
		Since the subbundle $E^{uu}_\Phi$ is H\"older continuous, there exist constants $C_4>1$ and $0<\alpha_2<1$ such that 
		\[ \angle\big( E^{uu}_\Phi(x_1), E^{uu}_\Phi(x_2)  \big) \leq C_4\cdot d^{\alpha_2}(x_1,x_2).\]
		Thus, let $C=3\cdot\max\{C_3,C_4\}$ and $\alpha=\min\{\alpha_1, \alpha_2\}$, we have
		\[ \angle\big( T_{x_1}\cF(x_1), T_{x_2}\cF(x_1)  \big) \leq C\cdot d^{\alpha}(x_1,x_2). \]	
		This shows that the curve $\cF(x_1)$ is $C^{1+}$-smooth.
	\end{proof}
	
	We continue the proof of Proposition \ref{prop foliation C1 to C1+}. Since the restriction ${\rm Hol}^{ss}_{y,x}|_{\cF^{uu}_\Phi(z_1)}:\cF^{uu}_\Phi(z_1)\to \cF(w_1)$ coincides with the holonomy map ${\rm Hol}^s_{y,x}:\cF^{uu}_\Phi(z_1)\to \cF(w_1)$ induced by foliation $\cF^s_\Phi$, and since the curves $\cF^{uu}_\Phi(z_1)$ and $\cF(w_1)$ are both $C^{1+}$-smooth, the $C^{1+}$-regularity of ${\rm Hol}^{ss}_{y,x}|_{\cF^{uu}_\Phi(z_1)}:\cF^{uu}_\Phi(z_1)\to \cF(w_1)$ follows from one of ${\rm Hol}^s_{y,x}$.  On the other hand, by Proposition \ref{prop foliation C1+}, the restriction ${\rm Hol}^{ss}_{y,x}|_{\cO_{\Phi}(z)}: \cO_{\Phi}(z_1) \to \cO_\Phi(w_1)$ is $C^{1+}$-smooth. Note that $\cF^u_\Phi(y)$ is subfoliated by $\cF^{uu}_\Phi$ and $\cO_\Phi$, and locally,  $\cF^u_{\Phi,{\rm loc}}(x)$ is subfoliated  by foliations $\cO_\Phi$ and ${\rm Hol}^{ss}_{y,x}(\cF^{uu}_{\Phi,{\rm loc}})$. By the Journ\'e Lemma (see Lemma \ref{lem journe}), the holonomy map ${\rm Hol}^{ss}_{y,x}$ is $C^{1+}$-smooth and by Proposition \ref{prop foliation regularity}, the foliation $\cF^{ss}_\Phi$ is $C^{1+}$-smooth.
\end{proof}

\subsubsection{The Livschitz theorem}

Recall that a flow $\Phi$ of $M$ is called \emph{transitive}, if there is a point $x\in M$ such that $\cO_{\Phi}(x)$ is dense in $M$.

\begin{proposition}[Livschitz Theorem\cite{L1971}]
	Let $\Phi$ be a transitive Anosov flow on $M$ and $\alpha$ be a H\"older cocycle over $\Phi$. If $\alpha(p,\tau(p,\Phi))=0$ for all $p\in {\rm Per}(\Phi)$, then $\alpha$ is H\"older cohomologous to function $0$, namely, there is a H\"older function $\beta:M\to \RR$ such that 
	\[\alpha(x,t)=\beta(x)-\beta\circ \Phi^t(x),  \quad \forall x\in M, \  \forall t\in \RR. \]
	Moreover, the function $\beta$ is unique up to an additive constant.
\end{proposition}

\begin{remark}\label{rmk half-Livsic}
	An inequality version of  Livschitz's type theorem due to Ma\~n\'e-Conze-Guivarc'h (see for example \cite{LT2005}) shows that if $\alpha(p,\tau(p,\Phi))\leq0$ for all $p\in {\rm Per}(\Phi)$, then there is a H\"older function $\beta:M\to \RR$ such that $\alpha(x,t)\leq \beta(x)-\beta\circ \Phi^t(x)$,   for all$x\in M$ and $t\in \RR$. Moreover, the function $\beta$ is smooth along the flow direction.
\end{remark}

Livschitz Theorem is useful in rigidity issue on periodic data of Anosov systems. It is well known that two conjugate Anosov $3$-flows admitting same stable and unstable Jacobians at corresponding periodic points are smoothly conjugacy, see independent works of de La Llave \cite{dL1992} and  Pollicott \cite{Po1990}. For convenience, we state the  following case. 

\begin{theorem}\cite{dL1992,Po1990}\label{thm delaLlave}
	Let $\Phi,\Psi$ be two  $C^{r}$-smooth $(r>1)$ Anosov flow on $3$-manifold $M$ conjugate via $H$. If for all periodic point $p$ of\, $\Phi$,
	\[J^u(p,\Phi)=J^u(H(p),\Psi)\]
	then $H$ is $C^{r_*}$-smooth  along each leaf of $\cF^u_\Phi$. 
\end{theorem}

\begin{remark}
	The $C^{r_*}$-regularity of $H$ in Theorem \ref{thm delaLlave} follows from $C^r$-smoothness along the strong unstable foliation and along the orbit, and the Journ\'e Lemma \ref{lem journe}.   Similarly, if $J^s(p,\Phi)=J^s(H(p),\Psi)$, then $H$ is $C^{r_*}$-smooth  along each leaf of $\cF^s_\Phi$. By the Journ\'e Lemma again,  when  the stable and unstable  Jacobians of $\Phi$ and $\Psi$ on corresponding period points coincide respectively, $H$ is  $C^{r_*}$-smooth.
\end{remark}

By a similar argument as \cite{dL1992}, one can get the following smooth rigidity of orbit-equivalence. We refer to \cite[Appendix A]{GLRH2025} for a proof.

\begin{proposition}[\cite{dL1992,GLRH2025}]\label{prop Jac smooth orbit-equiv}
	Let\, $\Phi$ and $\Psi$ be two $C^{r}$-smooth $(r>1)$ Anosov flows on $3$-manifold $M$ orbit-equivalent via $H$. Assume that for all periodic point $p\in M$ of\, $\Phi$,
	\[J^u(p,\Phi)= J^u(H(p),\Psi) \quad {\rm and} \quad J^s(p,\Phi)= J^s(H(p),\Psi).\]
	Then $\Phi$ and $\Psi$ are $C^{r_*}$-smoothly orbit-equivalent, here the smooth orbit-equivalence may be not $H$.
\end{proposition}

 \subsubsection{Equilibrium states of Anosov flows}

Denote the set of $\Phi$-invariant probability measures on $M$ by $\mathcal{M}_{\Phi}(M)$.

\begin{proposition}[\cite{B2001,BR1975,C2002}]\label{prop pressure}
	Let $\Phi$ be a $C^{1+}$-smooth transitive Anosov flow and $f:M\to \RR$ be a H\"older continuous function. Then there exists a unique $\mu_f\in \mathcal{M}_{\Phi}(M)$ such that 
	\[ P_f(\Phi)=\sup_{\mu\in \mathcal{M}_{\Phi}(M)}\big\{h_\mu(\Phi)+\int_M fd\mu   \big\} =  h_{\mu_f}(\Phi)+\int_M fd\mu_f.    \]
	We call $\mu_f$  the \emph{equilibrium state} of the potential $f$ and $P_f(\Phi)$ is the pressure.	Moreover, for two H\"older functions $f_1$ and $f_2$, their  equilibrium states $\mu_{f_1}$ and $\mu_{f_2}$ coincide, if and only if $f_1-P_{f_1}(\Phi)$ and $f_2-P_{f_2}(\Phi)$ are H\"older cohomologous, namely, their corresponding cocycles
	\[\alpha_{i} (x,t):= \int_0^t \big( f_i\circ \Phi^\tau(x)-P_{f_i}(\Phi) \big)d\tau, \ (i=1,2)  \]  are H\"older continuously cohomologous.
\end{proposition}

In particular, if $f$ is constant, then $\mu_f$ is the measure of maximal entropy for the flow $\Phi$. 
We call $\mu$ is the \emph{SRB measure} (respectively  the \emph{inverse SRB measure}), if it is the unique equilibrium state of the potential $-J^u_\Phi(x)$ (respectively $J^s_\Phi(x)$).
It is well known that \cite{BR1975} a $C^{1+}$-smooth transitive Anosov flow $\Phi$ satisfies that 
\[ P_{-J^u_{\Phi}}(\Phi)=0 \quad {\rm and}\quad P_{J^s_{\Phi}}(\Phi)=0. \]
Hence, the SRB measure $\mu$ has 
\[ h_\mu(\Phi)=\int_MJ^u_\Phi(x)d\mu.\]

\begin{remark}\label{rmk proof of SRB=MME}
	Let $\Phi$ be a $C^{1+}$-smooth transitive Anosov flow on $M$. Proposition \ref{prop pressure} and the Livschitz Theorem implies that  the SRB measure coincides with the MME, if and only if  for each periodic point $p$ with period $\tau(p,\Phi)$, one has 
	$J^u(p,\Phi)= h_{\rm top}(\Phi)\cdot \tau(p,\Phi)$. Thus, Corollary \ref{cor mme=srb} follows from Theorem \ref{thm flow Phi sC1 flex}.
\end{remark}

Applying the inequality type of Livschitz Theorem (Remark \ref{rmk half-Livsic}) and the uniqueness of equilibrium state (Proposition \ref{prop pressure}), we will show that any H\"older potential is cohomologous (up to a constant) to a negative function, so that this new function can be the target Jacobians. 

\begin{lemma}\label{lem function 0 pressure}
	Let $\Phi$ be a $C^{1+}$-smooth transitive Anosov flow on $M$ and $f:M\to \RR$ be a H\"older function. Let $P=P_{f}(\Phi)$. Then,
	\begin{itemize}
		\item  there are constants $T_0>0$ and $\e_0>0$ such that 
		\begin{align}
			\frac{1}{T}\int^T_0(f-P)\circ \Phi^t(x)dt<-\e_0, \quad \forall x\in M,\  \forall T>T_0,\label{eq. final negative}
			 \end{align} 
		 \item there is H\"older function $g:M\to \RR$ such that $g<0$ and the cocycles 
		 \[\alpha_{f-P}(x,t):=\int_0^t (f-P)\circ \Phi^\tau(x)d\tau \quad {\rm and} \quad  \alpha_{g}(x,t):=\int_0^t g\circ \Phi^\tau(x)d\tau \]
		 are H\"older cohomologous. In particular, $P_g(\Phi)=0$.
	\end{itemize} 
\end{lemma}

\begin{proof}
	
	For every periodic measure $p\in {\rm Per}(\Phi)$ with period $\tau(p,\Phi)$,  we denote by $\mu_p$ the periodic measure supported on $\cO_\Phi(p)$, i.e., the Lebesgue measure on the single periodic orbit $\cO_\Phi(p)$. It is clear that $\mu_p$ is $\Phi$-invariant with $0$-entropy, $h_{\mu_p}(\Phi)=0$. Thus, by the definition of pressure, 
	$\int_M(f-P)d\mu_p\leq 0$.
	Moreover,  for all $p\in{\rm Per}(\Phi)$,
	\begin{align}
		\int_M(f-P)d\mu_p< 0.\label{eq. lem coho <0}
	\end{align}
	Indeed, if there is $p\in{\rm Per}(\Phi)$ such that $\int_M(f-P)d\mu_p= 0$, then $P=h_{\mu_p}(\Phi)+\int_M\vphi d\mu_p$. It contradicts with the uniqueness of equilibrium state and the fact that $\mu_\vphi$ given by  Proposition \ref{prop pressure} is full supported. We claim that 
	\begin{align}
		P>\sup_{\mu\in \mathcal{M}_{\Phi}(M)}\big\{\int_M f d\mu  \big\}:=C_f. \label{eq. lem coho <0, C}
	\end{align}
	Indeed, by the Shadowing Lemma and \eqref{eq. lem coho <0}, we already have $P\geq C_\vphi$. If the equality $P=C_\vphi$ holds, then by shadowing again, there is a sequence of periodic points $\{p_n\}$ such that 
	\[ \int_M\vphi d\mu_{p_n}\to P,\quad n\to +\infty. \]
	Since $\mathcal{M}_\Phi(M)$ is compact, up to take a subsequence we can assume that $\mu_{p_n}\to \mu \in\mathcal{M}_\Phi(M)$ with respect to the weak-star topology. Hence, $\int_M\vphi d\mu=P$. By the definition of pressure, $\mu$ is the unique equilibrium state and
	\[h_\mu(\Phi)=0.\]  This contradicts with the well known fact that equilibrium state of a transitive Anosov flow and a H\"older potential has positive entropy \cite{R1976}, see also \cite[Theorem 1.25]{B1975}. Hence, \eqref{eq. lem coho <0, C} holds.
	
	By \eqref{eq. lem coho <0, C},  let $\e_0=(P-C_f)/2>0$.  Then, the uniform time $T_0$ of formula \eqref{eq. final negative} follows from shadowing property and transitivity. The proof is standard and we just give a quick overview here.
	For any point $x\in M$ and time $T>0$, by transitivity of $\Phi$, one can find a orbit starts from a point $y$ in a  neighborhood of $\Phi^T(x)$ and ends at  a point $\Phi^{t_y}(y)$ neighborhood of $x$. Connecting these two local orbits, we get a compact pseudo orbit of $\Phi$. Then by Shadowing lemma, one can approach the pseudo orbit by a periodic orbit of $\Phi$. The key point is that for a fixed size of shadowing, hence a fixed size of neighborhood for applying transitivity, the time $t_y$ is uniformly upper bounded for all $x$ and $T$. Hence, when $T_0$ is big enough, the piece of orbit from $x$ to $\Phi^T(x)$  constitutes a extremely large proportion of the pseudo orbit. Then one can control the error among the time average of the integral of $f-P$ along the orbit of $x$, along the the pseudo orbit and along the periodic orbit, and get \eqref{eq. final negative}.

	Now, we prove the existence of function $g$. By the definition of $C_f$,  
	\[ \alpha_{f-C_f}\big(p,\tau(p,\Phi)\big)\leq 0, \quad \forall p\in{\rm Per}(\Phi).\]
	By Livschitz's type theorem (see Remark \ref{rmk half-Livsic}) there is a H\"older function $\beta:M\to \RR$ such that 
	\[\int_0^t \big(f\circ \Phi^\tau(x) -C_f\big)d\tau\leq \beta(x)-\beta\circ \Phi^t(x),\quad \forall x\in M,  \forall t\in \RR.\] 
	Hence, 
	\begin{align}
		\alpha_{f-P}(x,t) \leq (C_f-P)\cdot t+ \beta(x)-\beta\circ \Phi^t(x). \label{eq. lem coho <0, P}
	\end{align}
	Recall that $\beta$ is smooth along the flow direction, let $L_X\beta(x)=\frac{d}{dt}\big|_{t=0} \beta\circ \Phi^t(x)$. Then the formula \eqref{eq. lem coho <0, P} is equivalent to 
	\[ \int_0^t \big(f-P+L_X\beta \big)\circ \Phi^\tau(x)d\tau\leq (C_\vphi-P)\cdot t, \quad \forall x\in M, \forall t\in \RR. \]
	By the last formula and \eqref{eq. lem coho <0, C},  $g=f-P+L_X\beta$ is  the H\"older  function we desired.
\end{proof}

\subsection {The Radon-Nikodym realization}\label{subsec measure}

In this subsection, we recall a  main tool of this paper, the \emph{transverse  measures.} Roughly speaking, it is a family of measures supported on the transversals of foliations and is coherent with  the dynamics of the systems and holonomies. For example, the Margulis measures  \cite{M2004} of an Anosov flow supported on the strong unstable leaves such that the flow is conformal with respect to this family of measures.  One can naturally consider  extending the Margulis measures to the case that the corresponding dynamic is matching with a H\"older potential and the measures are supported on not only the strong unstable foliations but also any transversals of stable foliations, see for example \cite{C1993,H1994,Le2000,A2008}. 

\begin{definition}
	We call  a family of measures $\{\mu_x\}_{x\in M}$  \textit{subordinated} to foliation $\cF$, if it satisfies
	\begin{enumerate}
		\item $\mu_x$ is a non-atomic and locally finite Borel measure on $\cF(x)$ and positive on  every non-empty open set of $\cF(x)$.
		\item $\mu_x=\mu_y$,  if $\cF(x)=\cF(y)$.
	\end{enumerate}
\end{definition}

We will apply the following  family of  transverse  measures which is called the Radon-Nikodym realization theorem due to Asaoka \cite{A2008}. 

\begin{theorem}[\cite{A2008}]\label{thm Asaoka 0}
	Let $\Phi$ be a $C^{1+}$-smooth transitive Anosov flow on $3$-manifold $M$ and $g:M\to \RR$ be  a H\"older function with $P_g(\Phi)=0$.  Then there is  a family  $\{\nu^{uu}_x\}_{x\in M}$ subordinated to $\cF^{uu}_{\Phi}$ such that 
	\begin{enumerate}
		\item For any $x\in M$, $t\in \RR$  and $x'\in \cF^{uu}_{\Phi}(x)$,
		$$ 
		\log \frac{d (\Phi^t)_*(\nu^{uu}_{\Phi^t(x)})}{d \nu^{uu}_x}(x')=\alpha_\Phi(x,t,-g)=-\int_0^tg\circ \Phi^\tau(x')d\tau.
		$$ 
		\item For any $x, y\in M$ close enough and $x'\in \cF^{uu}_{\Phi}(x)$, the Radon-Nikodym derivative $\frac{d({\rm Hol}^{\Phi,s}_{x,y})_*(\nu^{uu}_y)}{d\nu^{uu}_x}(x')$ is H\"older continuous with respect to $x, y$ and $x'$.
		\item For any $x\in M$ and $y\in \cF^{uu}_\Phi(x)$, the function $y\mapsto \nu^{uu}_x\big([x,y] \big)$ is H\"older continuous with respect to $y$, where $[x,y]$ is the curve lying on $\cF^{uu}_\Phi(x)$ with endpoints $x,y$.
	\end{enumerate}
\end{theorem}

\begin{remark}
	If the function $g$ in Theorem \ref{thm Asaoka 0} is the constant $-h_{\rm top}(\Phi)$, then the family $\{\nu^{uu}_x\}_{x\in M}$ is so called the Margulis measure. 	One can also construct a family of measure $\{\nu^{ss}_x\}_{x\in M}$ subordinated to $\cF^{ss}_\Phi$ associate with function $g$ such that the Radon-Nikodym derivative under $\cF^u_\Phi$-holonomy is H\"older continuous, $v_x^{ss}\big([x,\cdot]\big)$ varies H\"older continuously, and the logarithm of the flow action is $\alpha_{\Phi}(x,t,g)$.
\end{remark}

\begin{remark}\label{rmk function final negative}
	Theorem \ref{thm Asaoka 0} follows from \cite[Theorem 3.1 and Lemma 4.2]{A2008}. We notice that \cite[Theorem 3.1]{A2008} further assumes that $g$ is negative. By Lemma \ref{lem function 0 pressure}, up to a H\"older cohomologous, we can assume that the function $g$ in Theorem \ref{thm Asaoka 0} is negative. Moreover, by the proof of Theorem \ref{thm Asaoka 0} (see \cite[Lemma 3.6 and Formula (3.12)]{A2008}), the negativity of function $g$ is only used to provide that there are constants $T_0>0$ and $\e_0>0$ such that 
	\begin{align}
		\frac{1}{T}\int^T_0g\circ \Phi^t(x)dt<-\e_0, \quad \forall x\in M,\  \forall T>T_0,\label{eq. final negative 0}
	\end{align} 
 By Lemma \ref{lem function 0 pressure}, \eqref{eq. final negative 0} holds for any   H\"older potential $g$ with $P_g(\Phi)=0$. Hence, Theorem \ref{thm Asaoka 0} holds. 
\end{remark}

It is helpful to recall the construction of the family $\{\nu_p^{uu}\}_{p\in M}$ in \cite{A2008}. In particular, we restate a key lemma given by  \cite[Formula (3.5) and (3.21) ]{A2008}
\begin{lemma}[\cite{A2008}]\label{lem Asaoka}
	Let $\Phi$ be a $C^{1+}$-smooth transitive Anosov flow on $3$-manifold. Let $g:M\to \RR$ be a H\"older continuous function such that $P_{g}(\Phi)=0$.  Let function $u:M\times M\to \RR$ be given by 
	\[ u(x,y)=\int_{0}^{+\infty} \Big( g\circ \Phi^{\eta(x,y)+t}(x) -g\circ \Phi^t(y) \Big)dt+\int_{0}^{\eta(x,y)}g\circ \Phi^t(x)dt, \]
	where $x\in M, y\in \cF^s_{\Phi,{\rm loc}}(x)$ close to $x$ and $\Phi^{\eta(x,y)}(x)\in \cF^{ss}_{\Phi}(y)$. Then 
	\begin{align}
		u(x,y)+u(y,w)=u(x,w),\quad \forall y,w\in \cF^s_{\Phi,{\rm loc}}(x) \label{eq. measure 1}
	\end{align}
	and 
	\begin{align}
		\log \frac{d({\rm Hol}^{\Phi,s}_{x,y})_*(\nu^{uu}_y)}{d \nu^{uu}_x} (x')= u\big(  {\rm Hol}^{\Phi,s}_{x,y}(x'),x'  \big),\quad  \forall x'\in \cF^{uu}_{\Phi,{\rm loc}}(x).  \label{eq. measure 2}
	\end{align}
\end{lemma}

\begin{remark}[\bf{[Continuous dependence of measure]}]\label{rmk measure path}
	We note that the family of measures $\{\nu^{uu}_p\}_{p\in M}$ given in Theorem \ref{thm Asaoka 0} varies continuously in weak-star topology with respect to the H\"older function $g\in C^0(M)$ equipped with the $C^0$-norm. Indeed, by the proof of Theorem \ref{thm Asaoka 0}, the family  $\{\nu^{uu}_p\}_{p\in M}$ can be \emph{a posteriori} given by first projecting the Gibbs measure (corresponding to function $g$) of the subshift of finite type $\sigma$ coding $\Phi$ to one local strong unstable  leaf in a rectangle of the Markov partition, then  sending the projection to everywhere via the stable holonomy maps and by \eqref{eq. measure 2}. Since the Gibbs measure of $\sigma$ is continuous with respect to the H\"older potential, e.g.,  see\cite[Theorem 4.2.11]{K1998} and the stable holonomy is uniformly continuous, the family  $\{\nu^{uu}_p\}_{p\in M}$  is also continuous with respect to $g$.
\end{remark}

Recall that \cite{C1993} Cawley proves an original version of the  previous Radon-Nikodym realization for Anosov diffeomorphisms on $2$-torus. Using this family of transverse  measure, she further considers the Teichm\"uller space of Anosov diffeomorphism on $\TT^2$. For convenience, combining Lemma \ref{lem function 0 pressure} for Anosov diffeomorphism case, we restate her main result as follow.

\begin{theorem}[\cite{C1993}]\label{thm Cawley}
	Let $A:\TT^2\to\TT^2$ be a $C^{1+}$-smooth Anosov diffeomorphism and $\phi_1, \phi_2:\TT^2\to \RR$ be H\"older continuous functions. Then there are negative functions $\psi_i: \TT^2\to \RR_-\ (i=1,2)$ and a $C^{1+}$-smooth Anosov diffeomorphism $A':\TT^2\to \TT^2$ conjugate to $A$ via homeomorphism $h:\TT^2\to \TT^2$ such that 
	\begin{itemize}
		\item For $i=1, 2$, $\psi_i$ is H\"older cohomologous to $\phi_i-P_{\phi_i}(A)$ where $P_{\phi_i}(A)$ is the topological pressure, namely, there is H\"older function $u_i:\TT^2\to \RR$ such that 
		\[ \psi_i(x)=\phi_i(x)-P_{\phi_i}(A)+u_i(x)-u_i\circ A(x),\quad \forall x\in \TT^2.\]
		\item The stable and unstable  Jacobians $J^s_{A'}=\log \|DA'|_{E^s_{A'}}\|$ and $J^u_{A'}=\log \|DA'|_{E^u_{A'}}\|$ have
		\[J^s_{A'}=\psi_1\circ h^{-1}\quad {\rm and}\quad J^u_{A'}=-\psi_2\circ h^{-1}. \]
	\end{itemize}
\end{theorem}

\section{Jacobian Flexibility for HA-Flows}\label{sec HA-flow}

This section provides technical preliminaries for Theorem \ref{thm connect}, Theorem \ref{thm flow flexibility 0} and Theorem \ref{thm flow Phi sC1 flex}.  Instead of Anosov flows considered in the above three theorems, we will focus on a kind of H\"older continuous flows derived from Anosov flows which we called the \emph{HA-flow}. 

Let  $\Phi$ be a continuous flow on $M$ conjugate to an Anosov flow $\Phi_0$ on $M$ via  conjugacy $H_0$, i.e., $\Phi^t\circ H_0=H_0\circ \Phi^t_0$. One can define
the (strong) stable/unstable foliations of $\Phi$ by 
\[ \cF^\sigma_\Phi(x):=H_0\big( \cF^\sigma_{\Phi_0}(H_0^{-1}(x))\big), \quad \forall x\in M,\   \sigma=s,u,ss,uu. \] 
It is clear that
\[\cO_\Phi=H_0(\cO_{\Phi_0}) = H_0(\cF^u_{\Phi_0}\cap \cF^s_{\Phi_0})=\cF^u_\Phi\cap \cF^s_\Phi.  \]
Note that $\cF^\sigma_\Phi$  and $\cO_\Phi$ are continuous foliations and independent with the choice of $(\Phi_0,H_0)$. 

\begin{definition}\label{def HA} 
	A continuous flow $\Phi$  of $M$ is called \emph{HA-flow}, if it satisfies the following two items,
	\begin{enumerate}
		\item There is a $C^{1+}$-smooth transitive Anosov flow $\Phi_0$ and a bi-H\"older  homeomorphism $H_0:M\to M$ such that  
		\[\Phi^t\circ H_0=H_0\circ \Phi^t_0, \quad \forall t\in \RR,\]
		\item The stable/unstable foliations $\cF^{s/u}_\Phi$ are $C^{1+}$-smooth. Particularly,  $\cO_\Phi$ is $C^{1+}$-smooth and one can define the stable and unstable bundles by $E^{s/u}_\Phi:=T\cF^{s/u}_\Phi$. Recall that the foliations $\cF^{s/u}_\Phi$ induce H\"older cocycles over flow $\Phi$, 
		\[\alpha(x,t,\cF^{s/u}_\Phi,a):=\alpha_{\Phi}(x,t,\cF^{s/u}_\Phi,a), \]
		for some H\"older continuous metric $a(\cdot,\cdot)$, we refer to \eqref{eq. cocycle no E} and \eqref{eq. cocycle no flow} for notation conventions.
	\end{enumerate}
\end{definition}

\begin{remark}
	A natural way to get  HA-flow on $3$-manifold $M$ is  time change.  Let $\Phi_0$ be a transitive $C^{1+}$-smooth Anosov flow  generated by vector-field $Y$ on $M$. Let $u$ be a positive H\"older continuous function on $M$ cohomologous to $1$, i.e., $\alpha_{\Phi_0}(x,t,u)=t+\beta(x)-\beta\circ \Phi_0^t(x)$ for  some H\"older function $\beta$. Let flow $\Phi$ be generated by $\frac{1}{u}\cdot Y$, particularly, $\Phi_0^t(x)=\Phi^{\alpha_{\Phi_0}(x,t,u)}(x)$.  Then $\Phi$ is an HA-flow, since $H_0(x):=\Phi^{\beta(x)}(x)$ is a H\"older conjugacy from $\Phi_0$ to $\Phi$ and the foliations $\cF^{s/u}_\Phi=\cF^{s/u}_{\Phi_0}$ are $C^{1+}$-smooth. However, for an HA-flow,  there may be no smooth flow as a time change in its conjugacy class. That is why we will consider within  the orbit-equivalence class to find a smooth time change in Section \ref{sec Space}.
\end{remark}

In  this section, we deform a HA-flow  on $3$-manifold to get another HA-flow whose stable  and unstable holonomy maps  associated with objective functions.

\begin{theorem}\label{thm HA-flow flexibility}
	Let $\Phi$ be a HA-flow on $3$-manifold $(M,a)$. Let $f_\sigma:M\to \RR\ (\sigma=s,u)$ be two H\"older functions with topological pressure $P_\sigma=P_{f_\sigma}(\Phi)$.  Then there exist an HA-flow $\Phi_*$ on $M$, a H\"older continuous metric $a_*(\cdot,\cdot)$ and a bi-H\"older continuous homeomorphism $H$ of $M$ such that 
	\begin{enumerate}
		\item $\Phi_*$ is conjugate to $\Phi$ via $H$, i.e., $H\circ \Phi^t=\Phi_*^t\circ H$.
		\item The $C^{1+}$-smooth foliations $\cF^{\sigma}_{\Phi_*}:=H(\cF^{\sigma}_{\Phi})\ (\sigma=s, u)$  induce cocycles such that for all $x\in M$, $t\in \RR$,
		\[\alpha\big(x,t, \cF^u_{\Phi_*},a_*\big)= \alpha_{\Phi}\big(H^{-1}(x),t, f_s-P_s \big) \quad {\rm and} \quad \alpha\big(x,t, \cF^s_{\Phi_*},a_*\big)=-\alpha_{\Phi}\big(H^{-1}(x),t, f_u-P_u \big).\]
	\end{enumerate}
\end{theorem}

\begin{remark}\label{rmk HA-flow unique 1}
	We note that  such a continuous flow $\Phi_*$ in Theorem \ref{thm HA-flow flexibility} is unique, up to $C^{1+}$-smooth orbit-equivalence. However, on the one hand this uniqueness follows from the uniqueness of Theorem \ref{thm flow flexibility 0}, on the other hand the existence part of Theorem \ref{thm flow flexibility 0} follows from Theorem \ref{thm HA-flow flexibility}. To avoid confusion, we do not state it in the above theorem, and we
	will prove this uniqueness after the proof of Theorem \ref{thm flow flexibility 0} , see Remark \ref{rmk HA-flow unique 2}.
\end{remark}

Theorem \ref{thm HA-flow flexibility} will be deduced by the following  one-side flexibility case.
\begin{proposition}\label{prop ami jacobian 1}
	Let $\Phi$ be an HA-flow on $3$-manifold $(M,a)$ and $f:M\to \RR$ be a H\"older  function  with topological pressure $P=P_f(\Phi)$.
	Then there exist an HA-flow $\check{\Phi}$ on $M$, a H\"older continuous metric $\check{a}(\cdot,\cdot)$ and a  bi-H\"older  continuous homeomorphism $\check{H}:M\to M$ such that
	\begin{enumerate}
		\item The flow $\check{\Phi}$ is conjugate to $\Phi$ via $\check{H}$, i.e., $\check{H}\circ \Phi^t=\check{\Phi}^t\circ \check{H}$.
		\item The restriction $\check{H}|_{\cF^s_\Phi(x)}:\cF^s_\Phi(x)\to \cF^s_{\check{\Phi}}(H(x))$ is $C^{1+}$-smooth, for all $x\in M$.
		\item The cocycles induced by $C^{1+}$-smooth foliations $\cF^{u/s}_{\check{\Phi}}:=\check{H}(\cF^{s/u}_\Phi)$ satisfy that for all  $x\in M$ and $t\in \RR$,
		\[ 
		\alpha\big(x,t, \cF^u_{\check{\Phi}},\check{a}\big)=\alpha\big(\check{H}^{-1}(x),t, \cF^u_\Phi,a\big)\quad {\rm and} \quad \alpha\big(x,t, \cF^s_{\check{\Phi}},\check{a}\big)=-\alpha_\Phi\big(\check{H}^{-1}(x),t,f-P\big).\]
	\end{enumerate}
\end{proposition}

\begin{remark}\label{rmk Asaoka 1}
	We mention that, if the flow $\Phi$ in Proposition \ref{prop ami jacobian 1}  is a $C^{1+}$-smooth Anosov flow,  one may apply the method of Asaoka in \cite[Section 4.2]{A2008} or in \cite[Section 2.2]{A2012} to get a similar result. Specifically, in both \cite{A2008,A2012}, observing the $C^{1+}$-smooth flow under the $C^{1+}$-smooth local chart,  Asaoka deforms $\Phi$ via changing the $C^{1+}$-smooth structure of the manifold  to match the new $C^{1+}$-smooth structure with the transverse  measures given by Theorem \ref{thm Asaoka 0}.   Then he gets a conjugacy which is an identity as a set map, and a continuous flow  satisfying a similar conclusion of  Proposition \ref{prop ami jacobian 1}.  Our method is to adjust the metric and flow synchronously such that 
	\begin{itemize}
		\item the new metric ``discretely" matches the transverse  measures (see Proposition \ref{prop adapted metric}),
		\item under the new metric, the new flow's dynamics "continuously" matches the measure family.
	\end{itemize}
	This way will keep the smooth structure of the manifold. Hence, our conjugacy could not be an identity, but it can be arbitrarily $C^0$-close to the identity map, see Remark \ref{rmk H close to id}.
\end{remark}

\begin{remark}\label{rmk Asaoka 2}
  To get Theorem \ref{thm HA-flow flexibility}, also Theorem \ref{thm connect} and Theorem \ref{thm flow flexibility 0},  we need apply Proposition \ref{prop ami jacobian 1} twice. Even if we start from a $C^{1+}$-smooth Anosov flow $\Phi$ in Proposition \ref{prop ami jacobian 1}, both our and Asaoka's method can only get a continuous flow satisfying the desired cocycles.   It is worth to point out that both Asaoka and Proposition \ref{prop ami jacobian 1}  essentially deform $\Phi$ along  a $C^{1+}$-smooth one-dimensional  foliation close to $\cF^{uu}_\Phi$.  Moreover, we will show that  this kind of deformation of a $C^{1+}$-smooth Anosov flow  cannot be a  $C^{1+}$-smooth Anosov flow anymore, if one really changes the cocycles induced by stable holonomy, see  Remark \ref{rmk no flex}. This is a reason that we introduce the HA-flow, see Remark \ref{rmk HA-flow reason} for the other reason.
\end{remark}

Finally, we will improve Proposition \ref{prop ami jacobian 1} such that it also holds for a path of HA-flows matching a path of functions. It will be useful for proving the path-connectedness of the orbit-equivalence  space of an Anosov flow on $3$-manifold, i.e., Theorem \ref{thm connect}.  

\begin{proposition}\label{prop path-connect 1}
	Let $\Phi$ be an HA-flow on $3$-manifold $(M,a)$ generated by a H\"older vector-field $v_\Phi\cdot Y_\Phi$, where $Y_\Phi$ is  a $C^{1+}$-smooth vector-field and $v_\Phi$ is positive H\"older function.  Let $f_\kappa:M\to \RR,\ (\kappa\in[0,1])$ be a path of H\"older  functions with $\alpha(x,t,\cF^u_\Phi,a)=\alpha_\Phi(x,t,f_0)$ and the topological pressure $P_{f_\kappa}(\Phi)=P_\kappa$, for all $\kappa\in[0,1]$. 
	Then there exist a path of  HA-flows $\check{\Phi}_\kappa$ on $M$, a path of  H\"older continuous metrics $\check{a}_\kappa(\cdot,\cdot)$ and a path of bi-H\"older continuous homeomorphisms $\check{H}_\kappa:M\to M$ such that
	\begin{enumerate}
		\item $\check{\Phi}_0=\Phi$ and $\check{H}_0={\rm Id}_M$.
		\item The flow $\check{\Phi}_\kappa$ is conjugate to $\Phi$ via $\check{H}_\kappa$, i.e., $\check{H}_\kappa\circ \Phi^t=\check{\Phi}_\kappa^t\circ \check{H}_\kappa$.
		\item The cocycles induced by foliations $\cF^{u/s}_{\check{\Phi}_\kappa}$ satisfy that for all  $x\in M$ and $t\in \RR$,
		\[ 
		\alpha\big(x,t, \cF^u_{\check{\Phi}_\kappa},\check{a}_\kappa\big)=\alpha\big(\check{H}^{-1}_{\kappa}(x),t, \cF^u_\Phi,a\big)\quad {\rm and} \quad \alpha\big(x,t, \cF^s_{\check{\Phi}_\kappa},\check{a}_\kappa\big)=-\alpha_\Phi\big(\check{H}_\kappa^{-1}(x),t,f_\kappa-P_{\kappa}\big).\]
		\item There are a path of $C^{1+}$-smooth vector-field $\check{Y}_\kappa$ and a path of positive H\"older function $v_\kappa$ such that $\check{Y}_0=Y_\Phi$, $v_0=v_\Phi$ and  $\check{\Phi}_\kappa$ is generated by vector field $v_\kappa\cdot\check{Y}_\kappa$.
	\end{enumerate}
\end{proposition}

\subsection{The Radon-Nikodym realization for HA-flows}

Before we prove Theorem \ref{thm HA-flow flexibility}, Proposition \ref{prop ami jacobian 1} and Proposition \ref{prop path-connect 1}, we give an adaption of Radon-Nikodym realization for  HA-flows. Moreover, we give a new metric on the manifold such that it ``discretely" matches  this family of  transverse  measures, see Proposition \ref{prop adapted metric}.  

Let $\Phi$ be an HA-flow on $3$-manifold $M$ conjugate to a $C^{1+}$-smooth Anosov flow $\Phi_0$ on $M$ via bi-H\"older homeomorphism $H$ of $M$. Firstly, by the H\"older conjugacy $H$, one can get the existence and uniqueness of equilibrium state for $\Phi$. 

\begin{lemma}\label{lem pressure new}
	Let $f:M\to \RR$ be a H\"older  function. Then there exists a unique $\mu_f\in \mathcal{M}_{\Phi}(M)$ such that 
	\begin{align}
		P_f(\Phi)=\sup_{\mu\in \mathcal{M}_{\Phi}(M)}\big(h_\mu(\Phi)+\int_M fd\mu   \big) =  h_{\mu_f}(\Phi)+\int_M fd\mu_f.    \label{eq. equilibrium}
	\end{align}
	We call $\mu_f$ the equilibrium of the potential $f$.   Particularly, $P_f(\Phi)=P_{f\circ H}(\Phi_0)$.
\end{lemma}
\begin{proof}
	Let $\psi=H^*(f)=f\circ H$ and apply Proposition \ref{prop pressure} to $\Phi_0$ and H\"older function $\psi$, one get $\mu_\psi$ the unique equilibrium of $\psi$ with respect to $\Phi_0$. Let 
	\[ \mu_f=H_*(\mu_\psi), \quad {\rm i.e.,}\quad \mu_f(A)=\mu_\psi(H^{-1}(A)),\ \ \text{for any measurable set}\  A.  \]
	It is clear that $H_*$ is a bijection between $\mathcal{M}_{\Phi_0}(M)$ and $\mathcal{M}_{\Phi}(M)$,  and for each $\mu\in \mathcal{M}_{\Phi_0}(M)$, one has 
	\[h_{\mu}(\Phi_0)=h_{H_*(\mu)}(\Phi) \quad {\rm and}\quad \int_M\psi d\mu=\int_M f dH_*(\mu). \]
	Hence  $\mu_f=H_*(\mu_\psi)$ is the unique measure satisfying \eqref{eq. equilibrium} and $P_f(\Phi)=P_\psi(\Phi_0)$.
\end{proof}

We will consider a smooth foliation $\cL$, instead of the H\"older foliation $\cF^{uu}_\Phi$.

\begin{lemma}\label{lem C1+ L}
	Let $\Phi$ be a HA-flow on\, $3$-manifold $M$. Then there exists a $C^{1+}$-smooth one-dimensional foliation $\cL$ subfoliating $\cF^u_\Phi$ and transversely intersecting with $\cF^s_\Phi$.
\end{lemma}
\begin{proof}
	Since $\cF^u_\Phi$ is $C^{1+}$-smooth,  Proposition \ref{prop foliation regularity} provides a $C^{1+}$-smooth diffeomorphism $H$ of $M$ such that  $TH(\cF^u_\Phi)$ is a $C^{1+}$-smooth bundle. Let  $L'$  be a $C^{1+}$-smooth one-dimensional subbundle  of $TH(\cF^u_\Phi)$, and  integrable to a $C^{1+}$-smooth foliation $\cL'$ subfoliating $H(\cF^u_\Phi)$ and transversely intersecting with $H(\cF^s_\Phi)$. Then, $\cL=H^{-1}(\cL')$ is  the $C^{1+}$-smooth foliation satisfying the lemma.
\end{proof}

Applying Theorem \ref{thm Asaoka 0}, we can get the  Radon-Nikodym realization to $\cL$.
\begin{lemma}\label{lem measure L}
	Let $\cL$ be a one-dimensional $C^{1+}$-smooth foliation  subfoliating $\cF^u_\Phi$   and transversally intersecting with $\cF^s_\Phi$. 	Let $f:M\to \RR$ be a H\"older  function with $P_f(\Phi)=P$.  Then there exists  a family of measures $\{\mu_x\}_{x\in M}$ subordinated to foliation $\cL$ such that 
	\begin{enumerate}
		\item 	For any $x, y\in M$ close enough and $x'\in \cL(x)$, the Radon-Nikodym derivative $\frac{d({\rm Hol}^{s,\cL}_{x,y})_*(\mu_y)}{d\mu_x}(x')$ is H\"older continuous with respect to $x, y$ and $x'$, where ${\rm Hol}^{s,\cL}_{x,y}:\cL(x)\to \cL(y)$ is the holonomy map induced by $\cF^s_\Phi$.
		\item For any $x\in M$ and $y=\Phi^t(x)$, the Radon-Nikodym derivative $\log \frac{d({\rm Hol}^{0}_{x,y})_*(\mu_y)}{d\mu_x}(x)=-\alpha_\Phi(x,t,f-P)$, where  ${\rm Hol}^{0}_{x,y}:\cL(x)\to \cL(y)$ is the holonomy map induced by the orbit foliation $\cO_{\Phi}$ inside $\cF^u_\Phi(x)$.
		\item For any $x\in M$ and $y\in \cL(x)$, the measure $\mu_x([x,y]^{\cL})$ is H\"older continuous with respect to $y$, where $[x,y]^{\cL}$ is the curve lying on $\cL(x)$ with endpoints $x$ and $y$.
		\item 	Let $\Sigma_1$ and $\Sigma_2$ be two $C^{1+}$-smooth $2$-dimensional submanifolds of $M$ intersecting transversally with  $\cL$ such that the holonomy map ${\rm Hol}^{\cL}:\Sigma_1\to \Sigma_2$  is well-defined. Then the function \[x\mapsto\mu_x\big([x,{\rm Hol}^{\cL}(x)]^{\cL}\big), \] is H\"older continuous with respect to $x\in \Sigma_1$.
	\end{enumerate}
\end{lemma}

\begin{proof}
	Let $\cF=H^{-1}(\cL)$. Then $\cF$ is a H\"older continuous foliation of $M$, subfoliates $\cF^u_{\Phi_0}$ and topologically transverse to $\cF^s_{\Phi_0}$. Considering flow $\Phi$ and function $\tilde{f}:=f\circ H$, we will construct  a family of measures $\{ \nu_x\}_{x\in M}$ subordinated to foliation $\cF$ satisfying analogous properties in the statement.   Then the measure $\mu_x:=H_*(\nu_{H^{-1}(x)})$ automatically satisfies this lemma.
	
By the proof of Lemma \ref{lem pressure new},  $P_{\tilde{f}}(\Phi_0)=P_{f}(\Phi)=P$. Let H\"older function $\tilde{g}:=\tilde{f}-P$. Then $P_{\tilde{g}}(\Phi_0)=0$, and hence $\tilde{g}$ satisfies Theorem \ref{thm Asaoka 0} and Lemma \ref{lem Asaoka}. Let $\{\nu^{uu}_x\}_{x\in M}$  be the  measures subordinated to $\cF^{uu}_{\Phi_0}$ given by Theorem \ref{thm Asaoka 0}. For $x\in M, y\in \cF^s_{\Phi_0,{\rm loc}}(x)$ close to $x$, let function $u(x,y)$ be given by Lemma \ref{lem Asaoka} with respect to $\Phi_0$ and $\tilde{g}$, namely,
	\[ u(x,y)=\int_{0}^{+\infty} \Big( \tilde{g}\circ \Phi_0^t(y)-\tilde{g}\circ \Phi_0^{\eta(x,y)+t}(x)  \Big)dt-\int_{0}^{\eta(x,y)}\tilde{g}\circ \Phi_0^t(x)dt.\]
	For a measurable set $A\subset \cF_{\rm loc}(x)$, let 
	\[ \nu_x(A):=\int_{y\in A} e^{u\big(y,  {\rm Hol}^{s}_{x} (y) \big)} d\nu^{uu}_x\big(  {\rm Hol}^{s}_{x}(y) \big),  \]
	where $ {\rm Hol}^{s}_{x}: \cF_{\rm loc}(x)\to \cF^{uu}_{\Phi_0,{\rm loc}}(x)$ is the holonomy map induced by foliation $\cF^s_{\Phi_0}$. Note that $\nu_x$ is well defined on each leaf $\cF(x)$ and is independent with the choice of base point $x$. Indeed, let $x'\in A$. By \eqref{eq. measure 1} and \eqref{eq. measure 2},
	\begin{align*}
		\nu_{x'}(A)&=\int_A e^{u\big(y,  {\rm Hol}^{s}_{x'} (y) \big) }d\nu^{uu}_{x'}\big(  {\rm Hol}^{s}_{x'}(y) \big), \\
		&=\int_A e^{u\big(y,  {\rm Hol}^{\Phi_0,s}_{x,x'} \circ {\rm Hol}^{s}_{x} (y) \big)   \big) }d\nu^{uu}_{x'}\big(  {\rm Hol}^{\Phi_0,s}_{x,x'} \circ {\rm Hol}^{s}_{x} (y) \big),\\
		&= \int_A e^{u\big(y,  {\rm Hol}^{\Phi_0,s}_{x,x'} \circ {\rm Hol}^{s}_{x} (y) \big)   \big) } \cdot e^{u\big(  {\rm Hol}^{\Phi_0,s}_{x,x'} \circ {\rm Hol}^{s}_{x} (y),  {\rm Hol}^{s}_{x} (y) \big)}d\nu^{uu}_x\big(  {\rm Hol}^{s}_{x} (y) \big),\\
		&=\int_A e^{u\big(q,  {\rm Hol}^{s}_{x} (y) \big)} d\nu^{uu}_x\big(  {\rm Hol}^{s}_{x}(y) \big)=v_x(A).
	\end{align*}
	Calculating in the same way, one can prove that for any $w\in M$,
	\begin{align}
		\nu_x(A)=\int_{y\in A}e^{u\big( y, {\rm Hol}^{s}_{x,w}(y) \big)} d\nu^{uu}_w\big(  {\rm Hol}^{s}_{x,w}(y) \big), \label{eq. measure 3}
	\end{align}
	where ${\rm Hol}^{s}_{x,w}:\cF_{\rm loc}(x)\to \cF^{uu}_{\Phi_0,{\rm loc}}(w)$ is the holonomy map induced by $\cF^{s}_{\Phi_0}$.  Hence, $\{v_x\}_{x\in M}$ is well defined on the whole manifold $M$ and it is subordinated to foliation $\cF$.
	
	Denote by ${\rm Hol}^{s,\cF}_{x,y}:\cF(x)\to \cF(y)$, the holonomy map from $\cF(x)$ to $\cF(y)$ induced by foliation $\cF^{s}_{\Phi_0}$.  Let $x,y\in M$ be close and $x'\in \cF(x)$.  For short, we denote 
	\[x'_\#={\rm Hol}^s_p(x')\in \cF^{uu}_{\Phi_0}(x),\quad y'={\rm Hol}^{s,\cF}_{x,y}(x')\in \cF(y)\quad {\rm and}\quad   y'_\#= {\rm Hol}^s_y(y')\in \cF^{uu}_{\Phi_0}(y). \] 
	Then by the construction of $\{\nu_x\}_{x\in M}$, 
	\begin{align*}
		\log	\frac{d({\rm Hol}^{s,\cF}_{x,y})_*(\nu_y)}{d\nu_x}(x')&= 
		\log \frac{ d \big( {\rm Hol}^s_y\circ {\rm Hol}^{s,\cF}_{x,y} \big)_* (\nu^{uu}_y) }{d \big( {\rm Hol}^s_x \big)_* (\nu^{uu}_x)}(x') +u(y', y'_\#)- u(x', x'_\#)\\
		& =u\big( y'_\#, x'_\#   \big) +u(y', y'_\#)- u(x', x'_\#)= u(y',x')\\
		&=u\big({\rm Hol}^{s,\cF}_{x,y}(x'),x'   \big)
	\end{align*}
	is H\"older continuous with respect to $x,y$ and $x'$. By push-forward via H\"older continuous homeomorphism $H$,  the measure $\mu_x:=H_*(\nu_{H^{-1}(x)})$ satisfies the first item.
	
	Let $y=\Phi_0^t(x)$ and ${\rm Hol}^{0,\cF}_{x,y}:\cF(x)\to \cF(y)$ be the holonomy map induced by the orbit foliation $\cO_{\Phi_0}$. It is clear that ${\rm Hol}^{0,\cF}_{x,y}={\rm Hol}^{s,\cF}_{x,y}$, and by the definition of function $u$, 
	\begin{align*}
		\log	\frac{d({\rm Hol}^{0,\cF}_{x,y})_*(\nu_y)}{d\nu_x}(x)&=\log	\frac{d({\rm Hol}^{s,\cF}_{x,y})_*(\nu_y)}{d\nu_x}(x)\\
		&=u\big({\rm Hol}^{s,\cF}_{x,y}(x),x   \big)
		=u(y,x)
		=\int_{0}^{t}\tilde{g}\circ \Phi^\tau(x)d\tau.
	\end{align*}
	Thus, the second item also holds for $\mu_x=H_*(\nu_{H^{-1}(x)})$.
	
	Let $x'\in \cF(x)$ and $[x,x']^{\cF}$ is the curve lying on $\cF(x)$ with endpoints $x,x'$. By definition, 
	\[v_x\big([x,x']^{\cF} \big)=\int_{x}^{x'} e^{u\big(y,  {\rm Hol}^{s}_{x} (y) \big)} d\nu^{uu}_x\big(  {\rm Hol}^{s}_{x}(y) \big).  \]
	Recall that by Theorem \ref{thm Asaoka 0}, for $x$ and $w\in \cF^{uu}_{\Phi_0}(x)$, the measure $\nu^{uu}_x\big([x,w]^{\cF^{uu}_{\Phi_0}}\big)$ is H\"older continuous with respect to $w$. Since the function $u(\cdot,\cdot)$ and holonomy map ${\rm Hol}^s_x$ are H\"older continuous, we get that $v_x\big([x,x']^{\cF} \big)$ is H\"older with respect to $x'$. Thus, $\mu_x=H_*(\nu_{H^{-1}(x)})$ satisfies the third item.
	
	Let $\Sigma_i^*=H^{-1}(\Sigma_i)\ (i=1,2)$, where $\Sigma_i$ is given by the forth item. We denote by ${\rm Hol}^{\cF}:\Sigma_1^*\to \Sigma_2^*$ be the holonomy map induced by foliation $\cF$. Let $x,y\in \Sigma_1^*$. Denote
	\[x'={\rm Hol}^{\cF}(x),\quad y'={\rm Hol}^{\cF}(y),\quad x_\#={\rm Hol}^{s,\cF}_{y,x}(y') \quad  {\rm and} \quad y_\#={\rm Hol}^{s,\cF}_{x,y}(x).\] 
	Since ${\rm Hol}^{s,\cF}_{x,y}\big( [x,x_\#]^\cF \big)=[y',y_\#]^\cF$, by \eqref{eq. measure 3}, there exist constants $C,\alpha>0$ such that  \[\Big| \nu_x\big( [x,x_\#]^{\cF}  \big)- \nu_y\big( [y',y_\#]^{\cF}  \big)    \Big|\leq Cd^\alpha(x,y).\]
	Since  the points $x_\#$ and $y_\#$ H\"older continuously vary  with respect to $x$ and $y$, by the third item we proved above, there exist constants (without losing of generality, we still use the same constants as above) $C, \alpha>0$ such that  
	\[\Big| \nu_x\big( [x,x']^{\cF}  \big)- \nu_x\big( [x,x_\#]^{\cF}  \big)    \Big| \leq Cd^\alpha(x,y)\quad {\rm and}\quad   \Big| \nu_y\big( [y,y']^{\cF}  \big)- \nu_y\big( [y',y_\#]^{\cF}  \big)    \Big|\leq  Cd^\alpha(x,y). \]
	Combining the last three formulas,  we get 
	\begin{align*}
		\Big| \nu_x\big( [x,x']^{\cF}  \big)- \nu_y\big( [y,y']^{\cF}  \big)    \Big|&\leq  \Big| \nu_x\big( [x,x_\#]^{\cF}  \big) - \nu_y\big( [y',y_\#]^{\cF}  \big)    \Big| \\ &\qquad \quad + \Big| \nu_x\big( [x,x']^{\cF}  \big)- \nu_x\big( [x,x_\#]^{\cF}  \big)    \Big|  + \Big| \nu_y\big( [y,y']^{\cF}  \big)- \nu_y\big( [y',y_\#]^{\cF}  \big)    \Big|\\
		&\leq 3Cd^{\alpha}(x,y). 
	\end{align*}
	Hence, the forth  item also holds for the measure $\mu_x=H_*(\nu_{H^{-1}(x)})$.
\end{proof}

\begin{remark}[\bf{[Continuous dependence of measure]}]\label{rmk measure path 1}
	By the proof of the above lemma and Remark \ref{rmk measure path}, it is clear that the family of measures $\{\mu_x\}_{x\in M}$ in Lemma \ref{lem measure L} varies continuously with respect to H\"older functions with $0$-pressures. 
\end{remark}

Let $\cL$ be given in Lemma \ref{lem measure L}, which is transverse to $\cF^s_\Phi$. Let $U$ be a $\cL$-foliation box. We denote the upper and lower boundaries of $U$ by $U^+$ and $U^-$ respectively, i.e.,
\[U^+:=\bigcup_{x\in \Sigma} \sup\cL(x, U)\quad {\rm and} \quad U^-:=\bigcup_{x\in \Sigma} \inf\cL(x, U), \]
for some transversal $\Sigma\subset U$ of foliation $\cL$, where the sign $+/-$ and supremum/infimum coincide with the local orientation of $\cL$ in $U$.

\begin{definition}\label{def regular box}
	A $\cL$-foliation box $V$ is called \emph{s-regular}, if 
		\begin{enumerate}
		\item $V$ is \emph{proper}, i.e., $V=\overline{{\rm int}(V)}$.
		\item $V^+$ and $V^-$ are local leaves of $\cF^s_\Phi$.
		\item Each local leaf $\cL(x,V)$ is intersecting with both $V^+$ and $V^-$, for all $x\in V$.
	\end{enumerate}
	A finite family of $\cL$-foliation boxes $\{V_i \}_{1\leq i\leq k}$ is called \emph{s-regular}, if each $V_i$ is an $s$-regular $\cL$-foliation box, for all $1\leq i\leq k$.
\end{definition}

Similarly, one can define \emph{u-regular} $\cL$-foliation box family, if $\cL$ is transverse to $\cF^u_\Phi$.

\begin{remark}
	It is clear that $M$ can be covered by an  $s$-regular family of finitely many $\cL$-foliation boxes, if $\cL$ is transverse to $\cF^s_\Phi$.  We note that by the coherence of the foliation $\cF^s_\Phi$,	for all $\sigma,\tau=+,-$ and $1\leq i\neq  j\leq k$, either $V_i^\tau\cap V_j^\sigma=\emptyset$, or $V_i^\tau\cup V_j^\sigma$ is contained in a local leaf of $\cF^s_\Phi$.
\end{remark}

\begin{proposition}\label{prop adapted metric}
	
	Let $\Phi$ be a HA-flow on $3$-manifold $(M,a)$. Let foliation $\cL$ and  the family measures  $\{\mu_p\}_{p\in M}$ subordinated to $\cL$ be given by Lemma \ref{lem measure L}. Let $(T\cL)_a^\perp$ be the  orthogonal complement of $T\cL$ with respect to the metric $a$.   For any $s$-regular family of $\cL$-foliation box $\{V_i\}_{1\leq i\leq k}$, 
	there exists a H\"older continuous metric $\tilde{a}(\cdot,\cdot)$ of $M$ such that $(T\cL)_{\tilde{a}}^\perp=(T\cL)_a^\perp$ and $\tilde{a}|_{(T\cL)_a^\perp}=a|_{(T\cL)_a^\perp}$, and
	for $1\leq i\leq j\leq k$ with $V_i\cap V_j\neq \emptyset$, 
	\[l_{\tilde{a}}( [x,y])=\mu_p( [x,y]),\quad \forall x\in V_i^\pm \ \ {\rm and}\ \ y\in V_j^\pm \cap \cL(x, V_i\cap V_j),\]  where $l_{\tilde{a}}$ is the length induced by $\tilde{a}$ and $[x,y]$ is the curve lying on $\cL(x)$ with endpoints $x$ and $y$.
\end{proposition}

Notice that $i$ could equal to $j$ in the above proposition. The proof of this proposition is a bit lengthy and actually independent with the dynamics, for coherence, we leave it to the appendix. 

\begin{remark}[\bf{[Continuous dependence of metric]}]\label{rmk metric path}
Following the proof of Proposition \ref{prop adapted metric}, when the family $\{\mu_p\}_{p\in M}$ varies continuously with respect to the function, so does the metric $\tilde{a}(\cdot,\cdot)$. 
\end{remark}

\subsection{The deformation of HA-flows}\label{subsec deformation HA}

In this subsection, we prove Proposition \ref{prop ami jacobian 1} and Theorem \ref{thm HA-flow flexibility}. 

We first prove Proposition \ref{prop ami jacobian 1}.  Let $\Phi$ be an HA-flow on $3$-manifold $(M,a)$ and $f:M\to \RR$ be a H\"older  function  with topological pressure $P=P_f(\Phi)$.   By the previous subsection, we can assume that
\begin{itemize}
	\item $\cL$ is a one-dimensional $C^{1+}$-smooth foliation of $M$, $\cL$ subfoliates $\cF^u_\Phi$ and transversally intersects with $\cF^s_\Phi$, provided by Lemma \ref{lem C1+ L}.
	\item    $\{V_i\}_{1\leq i\leq k}$ is an $s$-regular  family of $\cL$-foliation boxes cover $M$.
	\item   $\{\mu_x\}_{x\in M}$ is the family of measures given by Lemma \ref{lem measure L} with respect to $f$ and $\cL$.
	\item $\tilde{a}$ is the H\"older  metric  provided by Proposition \ref{prop adapted metric} such that $(T\cL)_{\tilde{a}}^\perp=(T\cL)_a^\perp$ and
	$\tilde{a}|_{(T\cL)_a^\perp}=a|_{(T\cL)_a^\perp}$, and  $l_{\tilde{a}}([x,y])=\mu_x([x,y])$, for all  $x\in V_i^\pm$ and $y\in V_j^\pm \cap \cL(x, V_i\cap V_j)$, if $V_i\cap V_j\neq \emptyset$.
\end{itemize}   

\begin{proof}[Proof of Proposition \ref{prop ami jacobian 1}]
	For each $1\leq i\leq k$,	let $H_i:V_i \to V_i$ be a homeomorphism defined by 
	\begin{align}
			 H_i\big(\cL(x,V_i)\big)=\cL(x,V)\quad {\rm and}\quad  l_{\tilde{a}}\big([x,H_i(z)] \big)=\mu_x\big([x,z]\big),\label{eq. def H}
	\end{align}
	for every $x\in V_i^-$ and $z\in \cL(x,V_i)$.

		\begin{claim}\label{claim H well-def}
		For $z\in V_i\cap V_j$, one has that $H_i(z)=H_j(z)$. 	Hence, we actually get  a  homeomorphism $\check{H}:M\to M$ such that $\check{H}|_{V_i}=H_i$. 	
	\end{claim}
	\begin{proof}[Proof of Claim \ref{claim H well-def}]
		Let $x_i=\cL(z,V_i)\cap V_i^-$ and $x_j=\cL(z,V_j)\cap V_j^+$. Without loss of generality, we can assume that the local orientations of $\cL$ in $V_i$ and $V_j$ coincide. Then, $\cL(z,V_i\cup V_j)$ is the local leaf $\cL(z)$ with endpoints $x_i$ and $x_j$.
		By the definition of $H_i$, one has 
		\[l_{\tilde{a}}\big([x_i,H_i(z)]\big)=\mu_z([x_i,z]). \]
		On the other hand, one has 
		\[l_{\tilde{a}}\big([x_j,H_j(z)]\big)=\mu_z([x_j,z]). \]
		Indeed, let $x_j'=\cL(z,V_j)\cap V_j^-$. Then by the definition of $H_j$ and Proposition \ref{prop adapted metric}, 
		\begin{align*}
			l_{\tilde{a}}\big([x_j,H_j(z)]\big)&=l_{\tilde{a}}([x_j,x_j'])-l_{\tilde{a}}\big([x'_j,H_j(z)]\big)\\
			&=\mu_z([x_j,x_j'])-\mu_z([x_j',z])=\mu_z([x_j,z]).
		\end{align*} 
		Thus, $l_{\tilde{a}}\big([x_i,H_i(z)]\big)+l_{\tilde{a}}\big([x_j,H_j(z)]\big)=\mu_z([x_j,z])+\mu_z([x_i,z])=\mu_z([x_i,x_j])=l_{\tilde{a}}([x_i,x_j])$. Since $H_i(z)$ and $H_j(z)$ are in the curve $[x_i,x_j]$, we get that $H_i(z)=H_j(z)$.
	\end{proof}

	Let $\check{\Phi}$ be a continuous flow on $M$ given by $\check{\Phi}^t=\check{H}\circ \Phi^t\circ \check{H}^{-1}$.
	Let  $\cF^{s/u}_{\check{\Phi}}(x):=\check{H}\big(\cF^{s/u}_\Phi(\check{H}^{-1}(x))\big)$.

\begin{lemma}\label{lem ami jacobian 0}
	Let $V=V_i$, for any $1\leq i\leq k$. Then 
	\begin{enumerate}
		\item  $\check{H}$ is bi-H\"older continuous. 
		\item The foliations $\cF^{s/u}_{\check{\Phi}}\Big|_V$ are $C^{1+}$-smooth. Particularly, $\cF^{u}_{\check{\Phi}}(x,V)=\cF^u_\Phi(x,V)$, for all  $x\in V$.
		\item The restriction  $\check{H}:\cF^s_\Phi|_V\to \cF^s_{\check{\Phi}}|_V$ is $C^{1+}$-smooth.
	\end{enumerate}
In particular, $\check{\Phi}$ is a HA-flow, $\check{H}$ is bi-H\"older continuous on whole $M$ and $C^{1+}$-smooth along $\cF^s_\Phi$.
\end{lemma}

\begin{proof}[Proof of Lemma \ref{lem ami jacobian 0}]
	By the definition of $\check{H}$, it is clear that
	\[ \check{H}|_{V^\sigma}={\rm Id}_{V^\sigma},\quad \forall  \sigma=+,-.\]
	Since $\cF^u_\Phi$ is subfoliated by $\cL$ and $\check{H}$ is a deformation along each leaf of $\cL$,  one has that $\cF^u_{\check{\Phi}}$   coincides with  $\cF^u_\Phi$.  In particular, the local foliation $\cF^u_{\check{\Phi}}$ is $C^{1+}$-smooth.

	\begin{claim}\label{claim C1+ foliation 1}
		The  foliation $\cF^{s}_{\check{\Phi}}\Big|_V$ is $C^{1+}$-smooth.
	\end{claim}
	
	\begin{proof}[Proof of Claim \ref{claim C1+ foliation 1}]
		By the first item of Lemma \ref{lem measure L} and the above construction of $H$, the holonomy map of local foliation $\cF^s_{\check{\Phi}}$,
		\[ {\rm Hol}^{\cF^s_{\check{\Phi}},\cL}_{x,y}:\ \cL_{\rm loc}(x)\to \cL_{\rm loc}(y),\]
		has H\"older derivative, for any $x, y\in V$ with $y\in \cF^s_{\check{\Phi},\rm loc}(x)$.  Hence, the holonomy of $\cF^s_{\check{\Phi}}$ is $C^{1+}$-smooth.  Moreover, if it is necessary, one can slightly perturb $\cL$ to be a $C^{1+}$-smooth one-dimensional foliation $\cF$ in $V$. Then the derivative of holonomy map	\[ {\rm Hol}^{\cF^s_{\check{\Phi}},\cF}_{x,y}:\ \cF_{\rm loc}(x)\to \cF_{\rm loc}(y),\]
		has that 
		\[ D_p{\rm Hol}^{\cF^s_{\check{\Phi}},\cF}_{x,y}=\frac{\cos\theta(y)}{\cos\theta(x)}\cdot D_x{\rm Hol}^{\cF^s_{\check{\Phi}},\cL}_{x,y},\]
		where $\theta(x)=\angle\big(T\cL(x),T\cF(x)\big)$. Since both $T\cL$ and $T\cF$ are H\"older continuous, the holonomy map $ {\rm Hol}^{\cF^s_{\check{\Phi}},\cF}$ is also $C^{1+}$-smooth.
		
		Since $\check{H}(V^-)=V^-$  is a local leaf of $\cF^s_\Phi$,  the $C^{1+}$-smooth submanifold $V^-$ is also a local leaf of $\cF^s_{\check{\Phi}}$. Note that for $x\in V$, the local leaf $\cF^s_{\check{\Phi}}(x)$  can be viewed as a graph of a $C^{1+}$-map from $V^-$ to $\cF(x)$, in the coordinate system formed by  $C^{1+}$-smooth manifold $V^-$ and the $C^{1+}$-foliation $\cF$. Thus $\cF^s_{\check{\Phi}}(x)$ is a $C^{1+}$-smooth submanifold.  Consequently, the leaves and holonomy maps of $\cF^s_{\check{\Phi}}$ are all $C^{1+}$-smooth, by Proposition \ref{prop foliation regularity}, the foliation $\cF^s_{\check{\Phi}}$ is $C^{1+}$-smooth.
	\end{proof}

	\begin{claim}\label{claim H cs smooth}
		The restriction $\check{H}|_{\cF^s_\Phi(x,V)}: \cF^s_\Phi(x,V)\to \cF^s_{\check{\Phi}}(\check{H}(x),V)$ is $C^{1+}$-smooth, for all  $x\in V$.
	\end{claim}
\begin{proof}[Proof of Claim \ref{claim H cs smooth}]
		Let $x\in V$ and $x':=\cL(x,V)\cap V^-$.  Then the restriction
	$\check{H}|_{\cF^s_\Phi(x,V)}$ has 
	\begin{align}
		\check{H}|_{\cF^s_\Phi(x,V)}=      {\rm Hol}^{\cL}_{	\check{H}(x'),	\check{H}(x)}      \circ 	\check{H}|_{V^-}\circ {\rm Hol}^{\cL}_{x,x'}, \label{eq H along cs}
	\end{align}
	where  ${\rm Hol}^{\cL}_{x,x'}:\cF^s_\Phi(x,V)\to \cF^s_\Phi(x',V)$ and ${\rm Hol}^{\cL}_{	\check{H}(x'),	\check{H}(x)} :\cF^s_{\check{\Phi}}(	\check{H}(x'),V)\to \cF^s_{\check{\Phi}}(	\check{H}(x'),V) $  are holonomy maps induced by  $\cL$. Since $	\check{H}|_{V^-}={\rm Id}|_{V^-}$ and $\cL$ is $C^{1+}$-smooth, the restriction $	\check{H}|_{\cF^s_\Phi(x,V)}(y)$ is $C^{1+}$-smooth  with respect to $y\in \cF^s_\Phi(x,V)$. 
\end{proof}
	
	\begin{claim}\label{claim H holder}
		The homeomorphism $\check{H}:V\to V$ is bi-H\"older continuous.
	\end{claim}
	\begin{proof}[Proof of Claim \ref{claim H holder}]
		Recall that  Claim \ref{claim H cs smooth} has given the bi-H\"older continuity of $\check{H}$ along $\cF^s_\Phi$.  By the Journ\'e Lemma,  we need just prove that $\check{H}$ is bi-H\"older along each leaf of the  foliation $\cL$. 
	 
	 Let $x\in V$ and $z\in \cL(x,V)$. We have
	 \begin{align}
	 		l_{\tilde{a}}\big([\check{H}(x),\check{H}(z)]\big)=\mu_{x'}([x,z]).\label{eq. H property}
	 \end{align}
		By the third item of Lemma \ref{lem measure L}, the measure $\mu_{x'}([x,z])$ is H\"older continuous with respect to $z\in \cL(x,V)$, so is the length $l_{\tilde{a}}\big([H(x),H(z)]\big)$. Hence, the restriction $H|_{\cL(x,V)}$ is bi-H\"older continuous.
	\end{proof}
By Claim \ref{claim H well-def},  the properties of $\check{H}$ and $\check{\Phi}$ in $V$ also holds for $\check{H}$ and $\check{\Phi}$ on whole manifold $M$. In particular, $\check{\Phi}$ is a HA-flow.  This completes the proof of Lemma \ref{lem ami jacobian 0}.
\end{proof}

	Now, we continue the proof of the proposition.
Since $\Phi$ and $\check{\Phi}$ are conjugate via $\check{H}$,  $\alpha^{u/s}(x,t):=\alpha\big(\check{H}(x),t,\cF^{u/s}_{\check{\Phi}},\tilde{a}\big)$ are H\"older cocycles over flow $\Phi$.
Since  the restriction $\check{H}|_{\cF^s_\Phi}$ is $C^{1+}$-smooth,
\begin{align}
	< \alpha(\check{H}(x),t,\cF^u_{\check{\Phi}},\tilde{a})>_\Phi= <\alpha(x,t,\cF^u_\Phi,a)>_\Phi, \label{eq. coho 1}
\end{align}
recall that $<\cdot>_\Phi$ is the cohomologous class of cocycle over flow $\Phi$. Indeed, let $S$ be a one-dimensional H\"older continuous subbundle of  $E^s_\Phi$. And $S':=DH|_{\cF^s_\Phi}(S)$. Then, one has that
\[\alpha(\check{H}(x),t,\cF^u_{\check{\Phi}},S', \tilde{a})=\alpha(x,t,\cF^u_{\Phi},S, \tilde{a})-\log\|DH|_{S(x)}\|+\log\|DH|_{S(\Phi^t(x))}\|. \]
By Lemma \ref{lem metric to change bundle}, different transverse bundle of $\cF^u_{\Phi/\check{\Phi}}$ and different metric determine the same cohomologous class, hence we get \eqref{eq. coho 1}. On the other hand, 
by \eqref{eq. H property} and the second item of Lemma \ref{lem measure L},  \[\alpha\big(\check{H}(x),t, \cF^s_{\check{\Phi}},T\cL, \tilde{a}\big)=-\alpha_\Phi(x,t,f-P),\]
and hence 
\[ < \alpha(\check{H}(x),t,\cF^s_{\check{\Phi}},\tilde{a})>_\Phi= <-\alpha_\Phi(x,t,f-P)>_\Phi. \]
Finally, we can adjust the metric such that the cocycles induced by foliations $\cF^{u/s}_{\check{\Phi}}$ exactly equal to $\alpha(x,t,\cF^u_\Phi,a)$ and $-\alpha_\Phi(x,t,f-P)$, respectively. 
\begin{claim}\label{claim change metric}
	There is a H\"older metric $\check{a}$ such that \[ \alpha(\check{H}(x),t,\cF^u_{\check{\Phi}},\check{a})= \alpha(x,t,\cF^u_\Phi,a)\quad {\rm and} \quad   \alpha(\check{H}(x),t,\cF^s_{\check{\Phi}},\check{a})= -\alpha_\Phi(x,t,f-P). \]
\end{claim}
\begin{proof}[Proof of Claim \ref{claim change metric}]
	Let $\tilde{a}_1(\cdot,\cdot)$ be a metric such that the bundles
	\[ S'(=DH|_{\cF^s_\Phi}(S)\subseteq E^s_{\check{\Phi}}),\quad  T\cO_{\check{\Phi}}(=E^s_{\check{\Phi}}\cap E^u_{\check{\Phi}} )\quad {\rm and} \quad T\cL(\subseteq E^u_{\check{\Phi}}=E^u_\Phi)\] are orthogonal with each other and $\tilde{a}_1|_{E}=\tilde{a}|_E$, for $E=S',T\cO_{\check{\Phi}}$ and $T\cL$. Then one has that $S'=(E^u_{\check{\Phi}})^\perp_{\tilde{a}_1}$, $T\cL=(E^s_{\check{\Phi}})^\perp_{\tilde{a}_1}$ and 
\[ < \alpha(\check{H}(x),t,\cF^u_{\check{\Phi}},\tilde{a}_1)>_\Phi= <\alpha(x,t,\cF^u_\Phi,a)>_\Phi\quad {\rm and} \quad  < \alpha(\check{H}(x),t,\cF^s_{\check{\Phi}},\tilde{a}_1)>_\Phi= <-\alpha_\Phi(x,t,f-P)>_\Phi.\]
By the second item of Lemma \ref{lem metric to change bundle}, one can change the metric $\tilde{a}_1$ by scaling some positive H\"older function along $S'$ and $T\cL$ to get a metric $\check{a}$ satisfying the claim.
\end{proof}
This completes the proof of this proposition.
\end{proof}
\begin{remark}\label{rmk H close to id}
	In the proof of Proposition \ref{prop ami jacobian 1}, we can take the family of boxes $\{V_i\}_{1\leq i\leq k}$ with arbitrarily small size. Then, the conjugacy $\check{H}$ is $C^0$-close to the identity map of $M$.
\end{remark}

Applying Proposition \ref{prop ami jacobian 1} twice to functions $f_s$ and $f_u$, we can get Theorem \ref{thm HA-flow flexibility}.
\begin{proof}[Proof of Theorem \ref{thm HA-flow flexibility}]
	 
	 We apply Proposition \ref{prop ami jacobian 1} to flow $\Phi$ and function $f_u$ directly, then we get a triple $(\check{\Phi},\check{H},\check{a})$ such that $\check{H}\circ \Phi^t=\check{\Phi}^t\circ \check{H}$ and 
	 \[ 
	 \alpha\big(x,t, \cF^u_{\check{\Phi}},\check{a}\big)=\alpha\big(\check{H}^{-1}(x),t, \cF^u_\Phi,a\big)\quad {\rm and} \quad \alpha\big(x,t, \cF^s_{\check{\Phi}},\check{a}\big)=-\alpha_\Phi\big(\check{H}^{-1}(x),t,f_u-P_u\big).\]
	 By Lemma \ref{lem pressure new},  $P_{f_s\circ \check{H}^{-1}}(\check{\Phi})=P_{f_s}(\Phi)=P_s$.
	 We apply Proposition \ref{prop ami jacobian 1} again to the HA-flow $\check{\Phi}$  by deforming $\check{\Phi}$ to match the family of measures with respect to function $f_s\circ \check{H}^{-1}$ and subordinated to a $C^{1+}$-smooth foliation subfoliating $\cF^s_{\check{\Phi}}$. Then, we get a triple $(\hat{\Phi},\hat{H},\hat{a})$ such that $\hat{H}\circ \check{\Phi}^t=\hat{\Phi}^t\circ \hat{H}$ and 
	  \[ \alpha\big(x,t, \cF^u_{\hat{\Phi}},\hat{a}\big)=\alpha_{\check{\Phi}}\big(\hat{H}^{-1}(x),t,f_s\circ\check{H}^{-1}-P_s\big) \quad {\rm and} \quad  
	 \alpha\big(x,t, \cF^s_{\hat{\Phi}},\hat{a}\big)=\alpha\big(\hat{H}^{-1}(x),t, \cF^s_{\check{\Phi}},\check{a}\big).\]
	Let $H_*=\hat{H}\circ \check{H}$, $\Phi_*=\hat{\Phi}$ and $a_*=\hat{a}$.  Then,  $\Phi_*$ is a HA-flow conjugate to $\Phi$ via bi-H\"older conjugacy   $H_*$, 
	  \[ \alpha\big(x,t, \cF^u_{\Phi_*},a_*\big)=\alpha_\Phi\big(H_*^{-1}(x),t,f_s-P_s\big) \quad {\rm and} \quad  
	\alpha\big(x,t, \cF^s_{\Phi_*},a_*\big)=-\alpha_\Phi\big(H_*^{-1}(x),t,f_u-P_u\big).\]
	This completes the proof of Theorem \ref{thm HA-flow flexibility}.
\end{proof}

\subsection{The Path of HA-flows}

In this subsection, we prove Proposition \ref{prop path-connect 1}.  We recall the assumption. 
\begin{itemize}
	\item Let  $\Phi$ be an HA-flow  on $3$-manifold $(M,a)$. Furthermore, it is generated by a H\"older vector-field $v_\Phi\cdot Y_\Phi$, where $Y_\Phi$ is  a $C^{1+}$-smooth vector-field and $v_\Phi$ is positive H\"older function.
	\item  There is a  path of H\"older continuous functions $f_\kappa:M\to \RR,\ (\kappa\in[0,1])$ with $\alpha(x,t,\cF^u_\Phi,a)=\alpha_\Phi(x,t,f_0)$ and the topological pressure $P_{f_\kappa}(\Phi)=P_{\kappa}$.  Note that $P_\kappa$ is continuous with respect to $\kappa\in[0,1]$, since the pressure is continuous on the potential see for example \cite[Theorem 9.5]{W1982}.
	\item Let  $\cL$ be a $C^{1+}$-smooth foliation subfoliating $\cF^u_\Phi$, and $\{V_i\}_{i\leq i\leq k}$ be a  family of $s$-regular $\cL$-foliation boxes covering $M$. 
\end{itemize}	
We will modify the proof of Proposition \ref{prop ami jacobian 1}.   The key point here is that when we apply  the construction of  Proposition \ref{prop ami jacobian 1} for different  function $f_\kappa$, we do not change the foliation $\cL$ and family $\{V_i\}_{i\leq i\leq k}$.

\begin{proof}[Proof of Proposition \ref{prop path-connect 1}]
	Let $\{\mu_x^\kappa\}_{x\in M}$ be family of measures subordinated to $\cL$  and with respect to function $f_\kappa$,  provided by Lemma \ref{lem measure L}.  Let  $\tilde{a}_\kappa(\cdot,\cdot)$  be the metric associated with $\{\mu_x^\kappa\}_{x\in M}$ and $\{V_i\}_{1\leq i\leq k}$ given by Proposition \ref{prop adapted metric}. Then,  the metric $\tilde{a}_\kappa(\cdot,\cdot)$  related to the measures $\{\mu_p^\kappa\}_{p\in M}$ is also a continuous path with respect to $\kappa\in [0,1]$, see Remark \ref{rmk measure path} and Remark \ref{rmk metric path}. 
	
	Let $H_\kappa$ be defined by the same way of  $\check{H}$ in \eqref{eq. def H}, and  $\Phi^t_\kappa:=H_\kappa\circ \Phi^t\circ H_\kappa^{-1}$.  Then,  the conjugacy $H_\kappa$  and HA-flow $\Phi_\kappa$ given are continuously  vary with respect to $\kappa\in[0,1]$. Let H\"older metric $a_\kappa$ be given as the same way as the metric $\check{a}$ in Claim \ref{claim change metric}. Particularly, the bundles $T\cL$ and $S\subset E^s_\Phi$ in  Claim \ref{claim change metric} are fixed, and $S_\kappa'=DH_\kappa|_{\cF^s_\Phi}(S)$ is continuous with respect to $\kappa$, by \eqref{eq H along cs}.  Hence, the metric $a_\kappa$ is continuous with respect to $\kappa$. Then the path $(\Phi_\kappa,H_\kappa,a_\kappa)_{\kappa\in[0,1]}$ 
satisfy  Proposition \ref{prop path-connect 1}, except for the forth item.

	For proving the forth item, we modify the proof of Proposition \ref{prop foliation regularity} \cite{A2008,H1983} which states that for a $C^{1+}$-smooth foliation $\cF$, there is a $C^{1+}$-diffeomorphism $h$ of $M$ such that $h(\cF)$ is generated by $C^{1+}$-smooth subbundle. Here, we need show that for each
	$C^{1+}$-smooth orbit foliation $\cO_\kappa:=\cO_{\Phi_\kappa}$, the  $C^{1+}$-smooth diffeomorphism $h_\kappa$ of $M$ such that $Th_\kappa(\cO_\kappa)$ being $C^{1+}$-smooth can be chosen continuously depending on  $\kappa\in [0,1]$. Note  that though $T\cO_\kappa=DH_\kappa|_{\cF^s_\Phi}(T\cO_0)$, we cannot get that $T\cO_\kappa$ is $C^{1+}$-smooth.
	
	Recall that $V_i$ is an $s$-regular $\cL$-foliation box. By the construction, the boundary $\partial(V_i)$ consists of six pieces in which two pieces are transverse to $\cO_\kappa$ and the other four are subfoliated by  $\cO_\kappa$, for all $\kappa\in [0,1]$.  Hence,  as a family of $\cO_0$-foliation boxes, we can take $\{V_i\}_{1\leq i\leq k}$,  such that it admits a foliation chart $\big\{(\vphi^0_i, V_i)\big\}_{1\leq i\leq k}$ such that for each $1\leq i\leq k$,
	\begin{itemize}
		\item $\vphi^0_i:V_i\to \RR^3$ is  $C^{1+}$-smooth such that $\vphi^0_i(V_i)\supset(-10,10)^3$ and $M=\bigcup_{1\leq i\leq k}( \vphi^0_i)^{-1}((-1,1)^3)$.
		\item Let $\pi_1:\RR\times \RR^2\to \RR$, $\pi_2:\RR\times \RR^2\to \RR^2$ and $\pi_3:\RR^2\times \RR\to \RR^2$ be natural projections. Then \[\vphi^0_i\big(\cO_0(x,V_i)\big)\subset \RR\times\{\pi_2\circ\vphi^0_i(x) \} \quad{\rm and}\quad   \vphi^0_i\big( \cL(x,V_i)\big)\subset \{ \pi_3\circ \vphi^0_i(x) \}\times \RR,    \quad  \forall x\in V_i.\]
		\item Let $x_i\in V_i$ such that $\vphi^0_i(x_i)=0$. There is a surface $\Sigma_i\subset V_i$ subfoliated by $\cL$ and transversally intersecting with each leaf of  $\cO_0|_{V_i}$ such that $x_i\in \Sigma_i$  and $\vphi^0_i(\Sigma_i)\subset \{\pi_1\circ \vphi^0_i(x_i)\}\times \RR^2$.
	\end{itemize}
	For each $\kappa\in[0,1]$, we denote by ${\rm Hol}_i^{\cO_\kappa}:V_i\to \Sigma_i$ the holonomy induced by $\cO_\kappa|_{V_i}$, and denote by 
	\[{\rm Hol}^\kappa_i: \vphi^0_i(V_i) \to \vphi^0_i(\Sigma_i) \]
	the holonomy map in $\vphi^0_i(V_i)\subset \RR^3$ induced by the $C^{1+}$-foliation $\vphi^0_i(\cO_\kappa)$.    Then
	\[ {\rm Hol}^\kappa_i=\vphi^0_i\circ {\rm Hol}_i^{\cO_\kappa}\circ (\vphi^0_i)^{-1} \quad {\rm and}\quad  {\rm Hol}_i^{\cO_\kappa}=\check{H}_\kappa\circ {\rm Hol}_i^{\cO_0}\circ \check{H}_\kappa^{-1}.\] In particular, 
	${\rm Hol}^\kappa_i$ varies  continuously with respect to $\kappa\in[0,1]$. 
	Since $\cO_\kappa|_{V_i}$ is  transversally intersecting with $\Sigma_i$ for all $\kappa\in[0,1]$,  we can define the $C^{1+}$-smooth diffeomorphism 
	\[ \psi^\kappa_i:\vphi^0_i(V_i)\to \RR^3,\quad  \psi^\kappa_i(z)=\big(\pi_1(z), \pi_2\circ {\rm Hol}^\kappa_i(z)  \big),\]
	and let \[\vphi_i^\kappa=\psi^\kappa_i\circ \vphi^0_i:V_i\to \RR^3.\] 
	It is clear that $\psi^0_i={\rm Id}$. Thus,  $\{(\vphi^\kappa_i,V_i)\}_{1\leq i\leq k}$ is a family of  $C^{1+}$-smooth  $\cO_\kappa$-foliation charts covering $M$ and varies continuously with respect to $\kappa\in[0,1]$. Without loss of generality, we still assume that $M=\bigcup_{1\leq i\leq k}( \vphi^\kappa_i)^{-1}((-1,1)^3)$, for all $\kappa\in[0,1]$.
	
	For a $C^{1+}$-smooth one-dimensional foliation $\cF$ on $M$. We call a $C^{\infty}$-smooth diffeomorphism $\phi:V\to \RR^3$ is a  $(\cF,\e)$-\emph{adapted coordinate} for some $\e>0$ and set $V\subset M$, if  $\phi(V)=(-9,9)^3$ and 
	\[ D\phi(T\cF(x))\subset \big\{ v_1\oplus v_2\in T_{\phi(x)}\RR^3=\RR\oplus \RR^2 \ \big|\ \|v_2\|\leq \e\|v_1\|  \big\}, \]
	for all $x\in V$.  By standard  smoothing methods, for each $1\leq i\leq k$, there is a path $H_i(x,\kappa):V_i\times [0,1]\to \RR^3$ such that the map $\phi^\kappa_i:=H_i(\cdot,\kappa):V_i\to \RR^3$ satisfies
	\begin{itemize}
		\item $\phi^0_i=\vphi^0_i:V_i\to \RR^3$ is a $C^{1+}$-smooth diffeomorphism, for $1\leq i\leq k$.
		\item  $\phi^\kappa_i:V_i\to \RR^3$ is $C^{\infty}$-smooth and $C^{1+}$-close to $\vphi^\kappa_i:V_i\to \RR^3$ for $0<\kappa\leq 1$. 
		\item $M=\bigcup_{1\leq i\leq k}( \phi^\kappa_i)^{-1}([-1,1]^3)$.
	\end{itemize} In particular, $\phi^\kappa_i:V_i\to \RR^3$ is a $C^{\infty}$-smooth $(\cO_\kappa,1/2)$-adapted coordinate.

	Denote by $U^{1+\alpha}(\cF)$, the set of points $p$ such that $T\cF$ is $C^{1+}$-smooth on a neighborhood of $p$.
	\begin{lemma}\label{lem C1+ bundle}
		Let $\{\cF^\kappa\}_{\kappa\in[0,1]}$ be a family of one-dimensional $C^{1+}$-smooth foliation on $M$.	Let $\phi^\kappa:V\to \RR^3$ be $C^\infty$-smooth $(\cF_\kappa,1)$-adapted coordinate.   Assume that
		\begin{itemize}
			\item  $U^{1+\alpha}(\cF^0)=M$ and $(\phi^0,V)$ is a foliation chart of $\cF^0$ on $V$. 
			\item  $\cF^\kappa$ and $\phi^\kappa$ vary continuously with respect to $\kappa$.
		\end{itemize}
		Then for any compact (could be empty) subset $V'\subset\bigcap_{\kappa\in[0,1]} U^{1+\alpha}(\cF^\kappa)$, there is a continuous path $h(x,\kappa):M\times[0,1] \to M$  such that the map $h^{\kappa}:=h(\cdot,\kappa):M\to M$ satisfies that
		\begin{enumerate}
			\item $h^0={\rm Id}_M$ and $h^\kappa:M\to M$ is a $C^{1+}$-smooth diffeomorphism $C^1$-close to ${\rm Id}_M$, for $0<\kappa\leq 1$.
			\item For any $\kappa\in[0,1]$, $V'\cup (\phi^\kappa)^{-1}([-1,1]^3)\subset U^{1+\alpha}(h^\kappa(\cF^\kappa))$.
		\end{enumerate}
	\end{lemma}
	
	\begin{proof}[Proof of Lemma \ref{lem C1+ bundle}]
		The proof is an adaption of \cite[Lemma A.1]{A2008}. Let map $F^\kappa:[-4,4]^3\to (-9,9)^2$ be given by for any $z=(z_1,z_2)\in \RR\times \RR^2$,
		\[ F^\kappa(z)= \pi_2\Big( \cF^\kappa\big(( \phi^\kappa)^{-1}(0,z_2),V \big) \cap \{z_1\}\times \RR^2 \Big), \]
		namely, the point $\big(z_1, F^\kappa(z)\big)$ is the unique intersection point of the local leaf $ \cF^\kappa\big(( \phi^\kappa)^{-1}(0,z_2),V \big)$ and  the transversal $ \{z_1\}\times \RR^2 $.  Since $\phi^\kappa$ is $(\cF^\kappa,1)$-adapted, one has $\|F(z)-z_2\|\leq \|z_1\|$.  By the assumption of the lemma, we have that 
		\begin{itemize}
			\item $F^0$ is the identity map.
			\item $F^\kappa$ is a $C^{1+}$-smooth map and continuous with respect to $\kappa\in[0,1]$. 
		\end{itemize}
		By standard smoothing methods, the path consisting of $F^\kappa$ can be  approximate by a path formed by $C^\infty$-smooth maps, except for keeping the starting point  $F^0$. Namely,
		there is a continuous path $W(x,\kappa):[-4,4]^3\times [0,1]\to (-9,9)^2$ such that the map $w^\kappa:=W(\cdot,\kappa):[-4,4]^3\to (-9,9)^2$ satisfies that $w^0=F^0$ and $w^\kappa$ is a $C^{\infty}$-smooth map $C^1$-close to $F^\kappa$, for $0<\kappa\leq 1$.

		Let $\lambda:[-4,4]^3\to [0,1]$ be a $C^\infty$-smooth bump function such that $\lambda(z)=1$ on $[-3,3]^3$, and $\lambda(z)=1$ on $[-4,4]^3\setminus[-3.5,3.5]^2$. Let maps $G^\kappa, G^\kappa_0:[-4,4]^3\to (-9,9)^3$ and $h^\kappa:M\to M$ be given by
		\[ G^\kappa(z)=\Big( z_1\ ,\    \lambda(z)w^\kappa(z)+\big(1-\lambda(z)\big)F^\kappa(z)    \Big)\quad {\rm and}\quad G^\kappa_0(z)=(z_1,F^\kappa(z)), \]
		and 
		\[ h^\kappa(x)=\Big\{ \begin{array}{lr}
			(\phi^\kappa)^{-1}\circ G^\kappa\circ (G^\kappa_0)^{-1}\circ \phi^\kappa(x), \; \ x\in U_\kappa\\ \qquad \qquad \quad  x,\qquad \qquad \qquad\quad \   x\notin U_\kappa,
		\end{array}   \]
		where $U_\kappa:=(\phi^{\kappa})^{-1}\big( \big\{ (z_1,F^\kappa(z)) \ |\ z\in[-4,4]^3    \big\}   \big)\subset M$. It is clear that when $w^\kappa$ is $C^1$-close to $F^\kappa$, the map $h^\kappa$ is a $C^{1+}$-smooth diffeomorphism $C^1$-close to Id$_M$. Moreover, $h^\kappa$ varies continuously with respect to $\kappa$, and $h^0={\rm Id}_M$. By the same proof of \cite[Lemma A.1]{A2008}, we get $V'\cup (\phi^\kappa)^{-1}([-1,1]^3)\subset U^{1+\alpha}(h^\kappa(\cF^\kappa))$.
	\end{proof}
	
	Applying Lemma \ref{lem C1+ bundle} inductively, we can finish the proof of Proposition \ref{prop path-connect 1} as follow. We claim  that for $1\leq i\leq k-1$, there is  $C^{1+}$-smooth diffeomorphism $h^\kappa_i$ continuously depending on $\kappa$ such that
	\begin{itemize}
		\item $h^0_i={\rm Id}_M$ and $\bigcup_{1\leq j\leq i}(\phi^\kappa_j)^{-1}([-1,1]^3)\subset U^{1+\alpha}\big(h^\kappa_i(\cO_\kappa)\big)$.
		\item  For $1\leq i\leq k$, $\phi^\kappa_i$ is a $\big( h^\kappa_1(\cO_\kappa), \frac{k+i}{2k}\big)$-adapted coordinate.
	\end{itemize} 
	Indeed, for the case of $i=1$, we take  $\cF^\kappa=\cO_\kappa$, $\phi^\kappa=\phi^\kappa_1$ the above $(\cO_\kappa,1/2)$-adapted coordinate $\phi^\kappa_1$ and $V'=\emptyset$ in Lemma \ref{lem C1+ bundle},  then we get  $C^{1+}$-smooth diffeomorphism $h^\kappa_1$ continuously depending on $\kappa$ satisfying the claim. By induction of $i$, we suppose that the claim holds for $1\leq i\leq k-1$. 
	Let $\cF^\kappa=h^\kappa_i(\cO_\kappa)$,  $\phi^\kappa=\phi^\kappa_{i+1}$ and $V'=\bigcup_{1\leq j\leq i}(\phi^\kappa_j)^{-1}([-1,1]^3)$  in Lemma \ref{lem C1+ bundle}. Then, there is a  $C^{1+}$-smooth diffeomorphism $h^\kappa_{i+1}$ continuously depending on $\kappa$ satisfying the claim. Thus, we get the above claim. In particular, the $C^{1+}$-smooth diffeomorphisms $h_\kappa:=h^\kappa_k$ for $\kappa\in[0,1]$ form a continuous path such that $h_0={\rm Id}_M$ and $\check{Y}_\kappa:=Dh_\kappa(T\cO_\kappa)$ is the $C^{1+}$-smooth bundle varying continuously with respect to $\kappa$.

	Finally, we replace the previous  flow $\Phi^t_\kappa$ and conjugacy $H_\kappa$ by 
	\[\check{\Phi}_\kappa^t:=h_\kappa\circ \Phi^t_\kappa\circ h^{-1}_\kappa\quad  {\rm and} \quad \check{H}_\kappa:=h_\kappa\circ H_\kappa. \]
	It is clear that $\check{\Phi}_\kappa$ is tangent to $\check{Y}_\kappa$. Since $\check{H}_\kappa$ is $C^{1+}$-smooth along the flow direction and varies continuously with respect to $\kappa$,	the speed $v_\kappa(x)$ of  $\check{\Phi}_\kappa$ is H\"older continuous  on  $x\in M$ and varies continuously  on $\kappa$. 
	By the similar construction of Claim \ref{claim change metric}, one can change  the metric $\tilde{a}_\kappa$ continuously (with respect to $\kappa$) to get metric  $\check{a}_\kappa$ such that $(\check{\Phi}_\kappa,\check{H}_\kappa,\check{a}_\kappa)$ satisfying  the proposition. 
\end{proof}

\section{The Orbit-Equivalence Space of Anosov Flow}\label{sec Space}

In this section, we prove  Theorem \ref{thm flow flexibility 0} the realization of Anosov flow, Theorem \ref{thm connect} the path-connectedness of the orbit-equivalence space of  Anosov flows, and Theorem \ref{thm component}  the  homotopy type of the Anosov flow space, and give the Teichm\"uller space of an Anosov flow on $3$-manifold, i.e., Corollary \ref{cor Teich space orbit-eq}. 

Firstly, we show the realization  (Theorem \ref{thm flow flexibility 0}), i.e., we will deform a transitive $C^{1+}$-smooth Anosov flow  on $3$-manifold to get another one in its orbit-equivalence class whose  stable and unstable Jacobians  associate with objective functions. We generalize this  to  the following one.  

\begin{theorem}\label{thm flow flexibility 1}
	Let $\Phi$ be a transitive $C^{1+}$-smooth Anosov flow on $3$-manifold $(M,a_0)$. Let $f_\sigma:M\to \RR\ (\sigma=s,u)$ be two H\"older function with $P_{f_\sigma}(\Phi)=P_\sigma$. Then there exist two H\"older function $g_\sigma:M\to \RR_-$, a $C^{1+}$-smooth Anosov flow $\Psi$ of manifold $M$ and a bi-H\"older continuous homeomorphism $H$ such that 
	\begin{enumerate}
		\item The function $g_\sigma<0$ is  H\"older cohomologous to function $f_\sigma-P_\sigma$ and $P_{g_\sigma}(\Phi)=0$, $\sigma=s,u$.
		\item $\Psi$ is orbit-equivalent to $\Phi$ via $H$.
		\item For all periodic point $p\in M$ of\, $\Psi$, one has that
		\[ J^s(p,\Psi)=\int_0^{\tau\big(H^{-1}(p),\Phi\big)} g_s\big( \Phi^t\circ H^{-1}(p)\big) dt \quad {\rm and} \quad J^u(p,\Psi)= -\int_0^{\tau\big(H^{-1}(p),\Phi\big)} g_u \big( \Phi^t\circ H^{-1}(p)\big) dt,\]
		where $\tau\big(H^{-1}(p),\Phi\big)$ is the period of periodic point $H^{-1}(p)$ with respect to the flow $\Phi$.
	\end{enumerate}
	Moreover, such a $C^{1+}$-smooth Anosov flow $\Psi$ is unique, up to $C^{1+}$-smooth orbit-equivalence.
\end{theorem}

	Recall that we can deform flow $\Phi$ to be a HA-flow $\Phi_*$ with desired Jacobians, see Theorem \ref{thm HA-flow flexibility}. By Proposition \ref{prop flow regularity}, we can $C^{1+}$-smoothly orbit-equivalent  $\Phi_*$ to be a $C^{1+}$-smooth flow $\Psi$.   We will show that $\Psi$ is the flow we want.   Moreover, for getting the cocycles exactly matching the $0$-pressure functions (not only in their cohomologous classes), we must use the H\"older continuous metric.  Here, we will take the smooth metric $a_0$ again, since we just consider the periodic Jacobians in Theorem \ref{thm flow flexibility 1}.

\begin{proof}[Proof of Theorem \ref{thm flow flexibility 1}]
	
   By Lemma \ref{lem function 0 pressure}, there are functions $g_s$ and $g_u$ satisfying the first item. 
 Applying  Theorem \ref{thm HA-flow flexibility} to flow $\Phi$ and functions $g_s$ and $g_u$, we get a HA-flow  $\Phi_0$ conjugate to $\Phi$ via bi-H\"older conjugacy $H_0$ such that for all $x\in M$ and $t\in \RR$,
 \begin{align}
 		<\alpha(x,t, \cF^u_{\Phi_0})>_{\Phi_0}=<\alpha_{\Phi_0}(x,t,g_s\circ H_0^{-1})>_{\Phi_0}, \  <\alpha(x,t, \cF^s_{\Phi_0})>_{\Phi_0}=<-\alpha_{\Phi_0}(x,t,g_u\circ H_0^{-1})>_{\Phi_0}. \label{eq. flex 1}
 \end{align}

 Recall that $\cO_{\Phi_0}$ is a $C^{1+}$-smooth foliation.  By Proposition \ref{prop flow regularity}, there is a $C^{1+}$-smooth diffeomorphism $H_1$ such that $DH_1(T\cO_{\Phi_0})$ is $C^{1+}$-smooth. Let $Y$ be the unit bundle of $DH_1(T\cO_{\Phi_0})$ under a smooth metric.  Let $H=H_1\circ H_0$.  Then 
 \begin{itemize}
 	\item The flow $\Phi^t_1:=H_1\circ \Phi_0^t\circ H_1^{-1}$ is a HA-flow generated by vector field $v\cdot Y$, where $v:M\to \RR_+$ is  H\"older continuous.  The H\"older speed $v$ follows from the fact that $\Phi_1$ is conjugate to $\Phi$ by conjugacy $H$ which is $C^{1+}$-smooth along the flow direction. 
 	\item 
 	Since $H_1$ is smooth, by \eqref{eq. flex 1} and Lemma \ref{lem metric to change bundle},  there exists a H\"older  metric $a(\cdot,\cdot)$ such that
 	\begin{align}
 		\alpha(x,t,\cF^u_{\Phi_1},a)=\alpha_{\Phi_1}(x,t,g_s\circ H^{-1})\quad {\rm and}\quad 	\alpha(x,t,\cF^s_{\Phi_1},a)=-\alpha_{\Phi_1}(x,t,g_u\circ H^{-1}) \label{eq. flex 2} 
 	\end{align}
 	\item The vector field $Y$ generates a $C^{1+}$-smooth flow $\Psi$.  Then, $\Psi$ is $C^{1+}$-smoothly orbit-equivalent to  $\Phi_0$ via $H_1$, and $\Psi$ is  $C^{\infty}$-smoothly orbit-equivalent to $\Phi_1$ via Id$_M$.  In particular, $\Psi$ is orbit-equivalent to $\Phi$ via $H$.
 \end{itemize}
 
Let $\cF^{s/u}_{\Psi}:=\cF^{s/u}_{\Phi_1}$. Then  $\cF^{s/u}_{\Psi}$  are $C^{1+}$-smooth foliations and invariant under the flow $\Psi$. Then one can consider the H\"older cocycle classes $<\alpha_{\Psi}(x,t,\cF^s_{\Psi})>_{\Psi}$ and $<\alpha_{\Psi}(x,t,\cF^u_{\Psi})>_{\Psi}$.  
	
	Since $\Phi_1$ is generated by $v\cdot Y$, for any point $x$, the time $T(x)$ from $x$ to $\Psi^1(x)$ under flow $\Phi_1$ is upper and lower bounded. Namely, there exists constant $C>1$ such that 
	\[C^{-1}<T(x)<C,\quad \forall x\in M, \]
	where $\Phi_1^{T(x)}=\Psi^1(x)$ and $\Psi^1$ is the time-one map of $\Psi$. Thus, by \eqref{eq. flex 2},  for every $x\in M$, 
	\[\alpha(x,1, \cF^u_{\Psi},a)=\alpha\big(x,T(x), \cF^u_{\Phi_1},a\big)=\int^{T(x)}_0g_s\circ H^{-1}\circ \Phi_1^t(x)dt, \]
	and 
	\[\alpha(x,1, \cF^s_{\Psi},a)=\alpha\big(x,T(x), \cF^s_{\Phi_1},a\big)=-\int^{T(x)}_0g_u\circ H^{-1}\circ \Phi_1^t(x)dt,\]
	with uniformly bounded time $T(x)$. This implies that $\Psi$ is a $C^{1+}$-smooth Anosov flow.  
	
	Denote by $E^{ss}$ and $E^{uu}$  the strong hyperbolic bundles of $\Psi$. Note that $\cF^{s/u}_\Psi=\cF^{s/u}_{\Phi_0}$, thus $E^{ss}$ and $E^{uu}$ are transverse to $\cF^u_{\Phi_1}$ and $\cF^s_{\Phi_1}$, respectively.	Let $a_0$ be the original smooth metric. By Lemma \ref{lem metric to change bundle}, we have
	\[ <\alpha(x,t,\cF^u_{\Phi_1},E^{ss},a_0)>_{\Phi_1}=<\alpha_{\Phi_1}(x,t,g_s\circ H^{-1})>_{\Phi_1}, \]
	and 
	\[<\alpha(x,t,\cF^s_{\Phi_1},E^{uu},a_0)>_{\Phi_1}=<-\alpha_{\Phi_1}(x,t,g_u\circ H^{-1})>_{\Phi_1}. \]
  In particular, for each periodic point $p$ of $\Phi_1$ with period $\tau(p,\Phi_1)$, one has that
  \begin{align}
  	\alpha\big(p,\tau(p,\Phi_1), \cF^u_{\Phi_1}, E^{ss},a_0\big)
  	=\int^{\tau(p,\Phi_1)}_0g_s\circ H^{-1}\circ \Phi_1^t (p)dt, \label{eq. flex 3}
  \end{align}
and 
\begin{align}
	\alpha\big(p,\tau(p,\Phi_1), \cF^s_{\Phi_1}, E^{uu},a_0\big)
	=-\int^{\tau(p,\Phi_1)}_0g_u\circ H^{-1}\circ \Phi_1^t (p)dt, \label{eq. flex 4}
\end{align}
Since the periodic orbit $\cO_{\Psi}(p)$ coincides with $\cO_{\Phi_1}(p)$ and by \eqref{eq. flex 3} and \eqref{eq. flex 4}, we have that for all periodic point $p$ of $\Psi$ with period $\tau(p,\Psi)$,
	\begin{align*}
		J^s(p,\Psi)&=\alpha\big(p,\tau(p,\Psi), \cF^u_{\Psi},E^{ss},a_0\big)\\ &=\alpha\big(p,\tau(p,\Phi_1), \cF^u_{\Phi_1}, E^{ss},a_0\big)
		=\int^{\tau(p,\Phi_1)}_0g_s\circ H^{-1}\circ \Phi_1^t (p)dt,
	\end{align*}
and 
	\begin{align*}
	J^u(p,\Psi)&=\alpha\big(p,\tau(p,\Psi), \cF^s_{\Psi},E^{uu},a_0\big)\\ &=\alpha\big(p,\tau(p,\Phi_1), \cF^s_{\Phi_1}, E^{uu},a_0\big)
	=-\int^{\tau(p,\Phi_1)}_0g_u\circ H^{-1}\circ \Phi_1^t (p)dt.
\end{align*}
Since $\Phi_1$ is conjugate to $\Phi$ via $H$,  we have 
	\[	J^s(p,\Psi)
	=\int^{\tau(p,\Phi_1)}_0g_s\circ H^{-1}\circ \Phi_1^t (p)dt=\int^{\tau(H^{-1}(p),\Phi)}_0g_s\circ \Phi^t\circ H^{-1}(p)dt, \]
	and
		\[	J^u(p,\Psi)
	=-\int^{\tau(p,\Phi_1)}_0g_u\circ H^{-1}\circ \Phi_1^t (p)dt=-\int^{\tau(H^{-1}(p),\Phi)}_0g_u\circ \Phi^t\circ H^{-1}(p)dt. \]
	Moreover, by Proposition \ref{prop Jac smooth orbit-equiv}, such a $C^{1+}$-smooth Anosov flow $\Psi$ is unique up to a $C^{1+}$-smooth orbit-equivalence. This completes the proof of the theorem.
\end{proof}
 
\begin{remark}\label{rmk HA-flow reason}
	We remark that by the same argument of Theorem \ref{thm flow flexibility 1}, one can just apply  Proposition \ref{prop ami jacobian 1} for deforming $\Phi$  to get a $C^{1+}$-smooth Anosov flow  $\Psi_1$ such that its stable Jacobian associated with $g_s$. However, we cannot apply Proposition \ref{prop ami jacobian 1} again to $\Psi_1$ to get a $C^{1+}$-smooth Anosov flow  such that its unstable Jacobian associated with $g_u$, since by the orbit-equivalence $H_1$ between $\Psi_1$ and $\Phi$, the pressure of $g_u\circ H_1^{-1}$ may be not zero with respect to flow $\Psi_1$.   Hence, we should first consider the HA-flow in the conjugacy class of $\Phi$ to change both stable and unstable Jacobians, then smoothing it.
\end{remark}

\begin{remark}\label{rmk HA-flow unique 2}
	Here we show that the uniqueness of the HA-flows in Theorem \ref{thm HA-flow flexibility} up to smooth orbit-equivalent, as mentioned in Remark \ref{rmk HA-flow unique 1}. Indeed, let $\Phi_1$ and $\Phi_2$ be two HA-flows deformed from $\Phi$ with same stable and unstable cocycles. Let $\Psi_1$ and $\Psi_2$ be the $C^{1+}$-smooth Anosov flows generated by the smooth vector-fields tangent to $\Phi_1$ and $\Phi_2$, respectively. Since $\Psi_1$ and $\Psi_2$ admit same stable and unstable Jacobians at corresponding periodic orbits, by Proposition \ref{prop Jac smooth orbit-equiv}, $\Psi_1$ is smoothly orbit-equivalent to $\Psi_2$. As the proof of  Theorem \ref{thm flow flexibility 1},   $\Phi_i$ is smoothly orbit-equivalent to $\Psi_i$, for $i=1,2$. Thus, $\Phi_1$ is smoothly orbit-equivalent to $\Phi_2$.
\end{remark}

\subsection{The Teichm\"uller space of Anosov flow}

In this subsection, we prove  Corollary \ref{cor Teich space orbit-eq}, namely, the orbit-equivalence space of a transitive Anosov flow on $3$-manifold module smooth orbit-equivalence classes can be represented by a product of two H\"older function spaces module cohomologous classes.

\begin{proof}[Proof of Corollary \ref{cor Teich space orbit-eq}]
	For $C^{1+}$-smooth transitive Anosov flow $\Phi$ on $3$-manifold $M$, let  map
	\begin{align*}
	 	\widetilde{\mathcal{K}}: \mathbb{F}^{\rm H}(M)\times   \mathbb{F}^{\rm H}(M),&\longrightarrow \mathcal{O}^{1+}(\Phi),\\
		(f_s,f_u)&\longmapsto \Psi,
	\end{align*}
	be defined by Theorem \ref{thm flow flexibility 1}, i.e., $\Psi$ is orbit-equivalent to $\Phi$ via $H$ such that for all $p\in {\rm Per}(\Phi)$,
	\begin{align}
		 J^s(H(p),\Psi)=\int_0^{\tau(p,\Phi)} (f_s-P_s)\circ  \Phi^t(p) dt \quad {\rm and} \quad J^u(H(p),\Psi)=-\int_0^{\tau(p,\Phi)} (f_u-P_u)\circ  \Phi^t(p) dt,\label{eq. teich space 1}
	\end{align}
	Similarly, for $f_\sigma'\ (\sigma=s,u)$ with pressure $P_\sigma'$, one gets flow $\Psi'$. 	Recall that $\Psi'$ is smoothly orbit-equivalent to $\Psi$, if and only if they admit same stable and unstable Jacobians of the return maps at corresponding periodic points, if and only if the functions $f_\sigma'-P_\sigma'$ and $f_\sigma-P_\sigma$ have same integrals along the periodic orbits by\eqref{eq. teich space 1}, if and only if $f_\sigma'-P_\sigma'$ and $f_\sigma-P_\sigma$ are cohomologous by the Livschitz Theorem.  Thus,
		\begin{align*}
		\mathcal{K}: \mathbb{F}^{\rm H}(M)/_{\sim_\Phi}\times   \mathbb{F}^{\rm H}(M)/_{\sim_\Phi},&\longrightarrow \mathcal{O}^{1+}(\Phi)/_{\sim^o},\\
		([f_s],[f_u])&\longmapsto [\Psi],
	\end{align*}
	 is well-defined and injective.
	
	Next, we consider the opposite, we first notice that one can define a map $\Psi\mapsto (J^s_\Psi\circ H, J^u_\Psi\circ H)$ for $\Psi\in \mathcal{O}^{1+}(\Phi)$, where $H$ is the orbit-equivalence from $\Phi$ to $\Psi$. However, this map is not well defined on the quotient spaces. Instead, we will define a map 
	\begin{align*}
		\mathcal{T}^s:\mathcal{O}^{1+}(\Phi)&\longrightarrow  \mathbb{F}^{\rm H}(M),
	\end{align*}
	as follows. We take the pair $(\Psi, H)$ such that the orbit-equivalence $H$ is smooth along the orbit. Then there is a H\"older cocycle $\gamma(\cdot,t)$ over flow $\Phi$ such that
	\[  H\circ \Phi^{t}(x)=\Psi^{\gamma(x,t)}\circ H(x).\]
	Equivalently, there is  a H\"older cocycle $\beta(\cdot,t)$ over flow $\Psi$ such that
	\[  H\circ \Phi^{\beta(x,t)}\circ H^{-1}(x)=\Psi^{t}(x),\]
	here $\beta\big(H(x),\gamma(x,t)\big)=t$ and $\gamma\big(x,\beta(H(x),t)\big)=t$. Moreover, since $H$ is smooth along the orbit, we get a H\"older function
	\[\gamma'(x):=\frac{d \gamma(x,t)}{dt}|_{t=0} \quad \text{such that}\quad \gamma(x,t)=\int^t_0\gamma'\circ \Phi^t(x)dt=\alpha_{\Phi}(x,t,\gamma').\]
	Particularly, $\partial_t \gamma(x,t)=\gamma'\circ \Phi^t(x)$.
	Then we consider the cocycle $J^s_{\Psi}(x,t)$, it has
	\begin{align*}
		J^s_{\Psi}(x,t)&=\int^t_0 J^s_\Psi\circ \Psi^\tau (x)d\tau\\
		&=\int^t_0  J^s_\Psi\circ H\circ \Phi^{\beta(x,\tau)}\circ H^{-1}(x)d\tau\\
		&= \int_0^{\beta(x,t)} J^s_\Psi\circ H\circ \Phi^\kappa\circ H^{-1}(x)d\gamma(H^{-1}(x),\kappa)\\
		&= \int_0^{\beta(x,t)} \big(J^s_\Psi\circ H\circ \Phi^\kappa\circ H^{-1}(x)\big)\cdot \big(\gamma'\circ \Phi^\kappa\circ H^{-1}(x)\big) d\kappa\\
		&=\int_0^{\beta(x,t)} \big(J^s_\Psi\circ H\cdot \gamma'\big)\circ \Phi^\kappa\circ H^{-1}(x) d\kappa.
	\end{align*}
	where  $\kappa=\beta(x,\tau)$ and $\tau=\gamma(H^{-1}(x),\kappa)$. Then  we define the map \[\mathcal{T}^s(\Psi)=J^s_\Psi\circ H\cdot \gamma'. \]
	Hence, $\alpha_{\Phi}\big(H^{-1}(x), \beta(x,t),\mathcal{T}^s(\Psi) \big)=J^s_\Psi(x,t)$. And for any periodic point $p\in {\rm Per}(\Psi)$ and $q=H^{-1}(p)$, one has that the period $\tau(q,\Phi)=\beta\big(p,\tau(p,\Psi)\big)$. This implies that 
	\begin{align}
		\alpha_{\Phi}\big(q, \tau(q,\Phi),\mathcal{T}^s(\Psi) \big)=J^s_\Psi\big(p,\tau(p,\Psi)\big)=J^s(p,\Psi). \label{eq. teich 1}
	\end{align}
	Thus, when we consider two flows $\Psi_1,\Psi_2\in \cO^{1+}(\Phi)$ with  $\Psi_1\sim^o\Psi_2$ via $H_0$, we have that 
	\[\alpha_{\Phi}\big(q, \tau(q,\Phi),\mathcal{T}^s(\Psi_1) \big)=J^s(p,\Psi_1)=J^s(H_0(p),\Psi_2)=\alpha_{\Phi}\big(q, \tau(q,\Phi),\mathcal{T}^s(\Psi_2) \big).  \]
	By Livschitz theorem, the cocycles over flow $\Phi$
	\[  \alpha_{\Phi}\big(x, t,\mathcal{T}^s(\Psi_1) \big) \quad {\rm and}  \quad \alpha_{\Phi}\big(x, t,\mathcal{T}^s(\Psi_2) \big) \]
	are H\"older-cohomologous, namely, $\mathcal{T}^s(\Psi_1)\sim_\Phi \mathcal{T}^s(\Psi_2)$.   Similarly, one can define the map 
	\[	\mathcal{T}^u:\mathcal{O}^{1+}(\Phi)\to  \mathbb{F}^{\rm H}(M),\] by $\mathcal{T}^u(\Psi)=J^u_\Psi\circ H\cdot \gamma'$. And if $\Psi_1\sim^o\Psi_2$, then $\mathcal{T}^u(\Psi_1)\sim_\Phi \mathcal{T}^u(\Psi_2)$. This shows that the map
	\begin{align*}
		\mathcal{T}:\mathcal{O}^{1+}(\Phi)/_{\sim^o}&\longrightarrow  \mathbb{F}^{\rm H}(M)/_{\sim_\Phi}\times  \mathbb{F}^{\rm H}(M)/_{\sim_\Phi},\\
		[\Psi]&\longmapsto \big( [\mathcal{T}^s(\Psi)],  [\mathcal{T}^s(\Psi)]\big),
	\end{align*}
	is well defined. And $	\mathcal{T}$ is also an injection. Indeed, if $\Psi_1,\Psi_2\in \cO^{1+}(\Phi)$ orbit-equivalent via $H_0$ with $\mathcal{T}^s(\Psi_1)\sim_\Phi \mathcal{T}^s(\Psi_2)$ and $\mathcal{T}^u(\Psi_1)\sim_\Phi \mathcal{T}^u(\Psi_2)$, by \eqref{eq. teich 1}, one has that 
	\[J^s(p,\Psi_1)=J^s(H_0(p),\Psi_2) \quad {\rm and} \quad J^u(p,\Psi_1)=J^u(H_0(p),\Psi_2), \]
	for all periodic point $p\in {\rm Per}(\Psi_1)$.  By Proposition \ref{prop Jac smooth orbit-equiv}, $\Psi_1\sim^o\Psi_2$.
	
	Since  the maps $\mathcal{K}$ and $\mathcal{T}$ are injective, by Cantor-Schr\"oder-Bernstein theorem, there is a bijection between  the spaces $\mathcal{O}^{1+}(\Phi)/_{\sim^o}$ and $ \mathbb{F}^{\rm H}(M)/_{\sim_\Phi}\times  \mathbb{F}^{\rm H}(M)/_{\sim_\Phi}$.
\end{proof}

\subsection{Path-connectedness of the space of Anosov flows}

In this subsection, we show Theorem \ref{thm connect}, i.e., two orbit-equivalent (via an orbit-equivalence homotopic to Id$_M$) transitive Anosov flows on $3$-manifold can be connected by a path of Anosov flows. Firstly, we give a version of Anosov vector fields.

\begin{theorem}\label{thm path-connect C1+vector}
	Let $\Phi$ and $\Psi$ be two   Anosov flows generated by $C^{1+}$-smooth  vector-fields on $3$-manifold $M$. Assume that $\Phi$ and $\Psi$ are orbit-equivalent by an orbit-equivalence $H$ homotopic to identity Id$_M$.  Then there is a path of  Anosov flows generated by $C^{1+}$-smooth  vector-fields on $M$ connecting $\Phi$ and $\Psi$.
\end{theorem}

\begin{proof}
	
	Without loss of generality, we assume that the orbit-equivalence $H$ is smooth along the flow direction. There is a H\"older cocycle $\gamma(\cdot,t)$ over flow $\Phi$ such that for all $x\in M$ and $t\in\RR$, $H\circ \Phi^{t}(x)=\Psi^{\gamma(x,t)}\circ H(x)$.  Let 
	\[ f^s_0=J^s_\Phi \quad {\rm and}\quad f^s_1=J^s_\Psi\circ H\cdot \gamma', \]
	where $\gamma'(x):=\frac{d \gamma(x,t)}{dt}|_{t=0}$.  As the proof of Corollary \ref{cor Teich space orbit-eq} (see \eqref{eq. teich 1}), for all $p\in {\rm Per}(\Phi)$, 
	\begin{align}
		\alpha_\Phi\big(p,\tau(p,\Phi),f^s_1\big)=J^s\big(H(p),\Psi\big). \label{eq. path 1}
	\end{align}
	Let $f^s_\kappa=(1-\kappa) f^s_0+\kappa f^s_1$, for $\kappa\in [0,1]$. It is clear that $P_{f^{s}_0}(\Phi)=P_{f^{s}_1}(\Phi)=0$ by SRB property.  
	\begin{figure}[htbp]
		\centering
		\includegraphics[width=8cm]{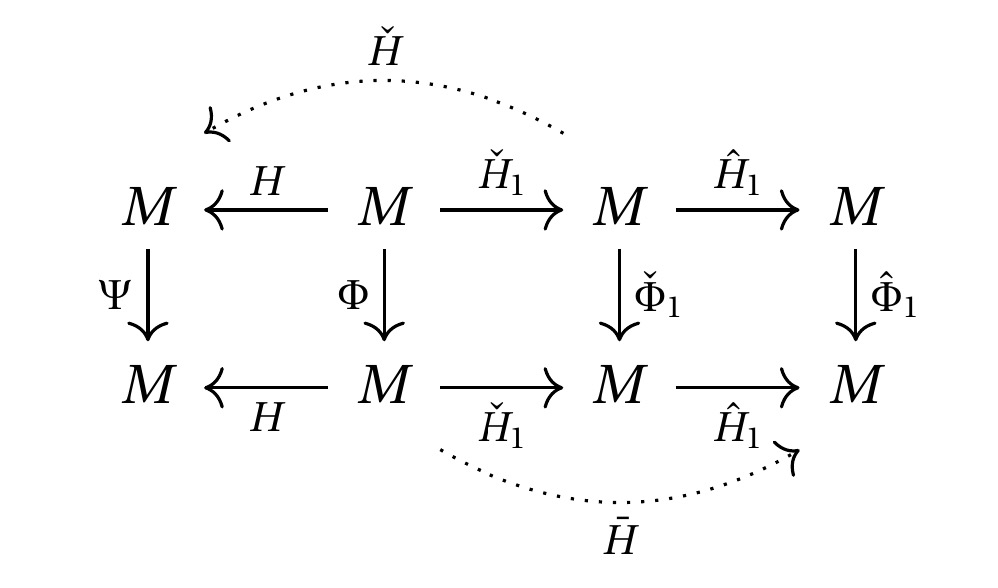}
		\caption{Commutative Diagram}	
	\end{figure}
	
	We apply Proposition \ref{prop path-connect 1} to flow $\Phi$ and potential $f^s_\kappa$. Let $g_\kappa^{s}:=f^{s}_\kappa-P_{f^{s}_\kappa}(\Phi)$.Then we get 
	\begin{itemize}
		\item a path of HA-flow $\check{\Phi}_\kappa$ whose orbit foliation tangent to $C^{1+}$-smooth vector-field $\check{Y}_\kappa$ such that $\check{\Phi}_\kappa$ is conjugate to $\check{\Phi}_0=\Phi$,
		\item a path of H\"older conjugacy $\check{H}_\kappa$ smooth along the flow direction,
		\item and a path of H\"older continuous metric $\check{a}_{\kappa}$,
	\end{itemize}
	such that
	\[	\alpha\big(\check{H}_1(x),t, \cF^s_{\check{\Phi}_1},\check{a}_1\big)=\alpha\big(x,t, \cF^s_\Phi,\check{a}_0\big)\quad {\rm and} \quad \alpha\big(\check{H}_1(x),t, \cF^u_{\check{\Phi}_1},\check{a}_1\big)=\alpha_\Phi(x,t,g^s_1)=\alpha_\Phi(x,t,f^s_1). \]

	Let $\check{H}:=H\circ \check{H}_1^{-1}$. Since $\check{H}_1$ is a conjugacy, we have  $\check{H}\circ \check{\Phi}^t_1(x)=\Psi^{\gamma(\check{H}_1^{-1}(x),t)}\circ \check{H}(x)$. Let 
	\[f^u_0(x)=-\frac{d}{dt}\Big|_{t=0}\alpha\big(x,t, \cF^s_{\check{\Phi}_1},\check{a}_1\big)=-\frac{d}{dt}\Big|_{t=0}\alpha\big(\check{H}_1^{-1}(x),t, \cF^s_\Phi,\check{a}_1\big)=-J^u_\Phi\circ \check{H}^{-1}_1(x)\]
	and \[f^u_1= -J^u_\Psi\circ \check{H}\cdot \gamma'\circ \check{H}^{-1}_1. \]
	Then we have for all $p\in {\rm Per}(\check{\Phi}_1)$,
	\begin{align}
		\alpha_{\check{\Phi}_1}\big(p,\tau(p,\check{\Phi}_1), f^u_1\big)=-J^u\big(\check{H}(p),\Psi \big).\label{eq. path 2}
	\end{align}
	As the above proof, let $f^u_\kappa=(1-\kappa) f^u_0+\kappa f^u_1$ and  $g_\kappa^u=f^u_\kappa-P_{f^u_\kappa}(\check{\Phi}_1)$.  Recall that $g_1^u=f^u_1$ and $g_0^u=f_0^u$. Applying Proposition \ref{prop path-connect 1} to $\check{\Phi}_1$ and $f^u_\kappa$, we get 
	\begin{itemize}
		\item a path of HA-flow $\hat{\Phi}_\kappa$ whose orbit foliation tangent to $C^{1+}$-smooth vector-field $\hat{Y}_\kappa$ such that $\hat{\Phi}_\kappa$ is conjugate to $\hat{\Phi}_0=\check{\Phi}_1$,
		\item a path of conjugacy $\hat{H}_\kappa$ smooth along the flow direction,
		\item  and a path of H\"older continuous metric $\hat{a}_{\kappa}$,
	\end{itemize}
	such that
	\begin{align}
		\alpha(x,t, \cF^u_{\hat{\Phi}_1},\hat{a}_1)=\alpha\big(\hat{H}^{-1}_1(x),t, \cF^u_{\check{\Phi}_1},\hat{a}_0\big)=\alpha_\Phi\big(\bar{H}^{-1}(x),t,f^s_1\big),\label{eq. path 3}
	\end{align}
	where $\bar{H}=\hat{H}_1\circ \check{H}_1$,
	and
	\begin{align}
		\alpha(x,t, \cF^s_{\hat{\Phi}_1},\hat{a}_1)=-\alpha_{\check{\Phi}_1}\big(\hat{H}^{-1}_1(x),t, g^u_1\big)=-\alpha_{\check{\Phi}_1}\big(\hat{H}^{-1}_1(x),t, f^u_1\big). \label{eq. path 4} 
	\end{align}
	Combining \eqref{eq. path 1} with \eqref{eq. path 3} and \eqref{eq. path 2} with \eqref{eq. path 4} respectively, we have that for all $p\in{\rm Per}(\hat{\Phi}_1)$,
	\begin{align}
		\alpha\big(p,\tau(p,\hat{\Phi}_1),\cF^u_{\hat{\Phi}_1},\hat{a}_1  \big)=J^s\big(H\circ \bar{H}^{-1}(p),\Psi \big)  \quad {\rm and}\quad \alpha\big(p,\tau(p,\hat{\Phi}_1),\cF^s_{\hat{\Phi}_1},\hat{a}_1  \big)= J^u\big(H\circ \bar{H}^{-1}(p),\Psi \big).  \label{eq. path 5}
	\end{align}

	Now, we consider a family of $C^{1+}$-smooth vector-field $Y_\kappa$ given by $Y_\kappa=\check{Y}_{2\kappa}$ for $\kappa\in[0,1/2]$, and $Y_\kappa=\hat{Y}_{2\kappa-1}$. It is clear that $Y_\kappa$ is a continuous path with respect to $\kappa$, since $\check{Y}_1=\hat{Y}_0$. Let $\Phi_\kappa$ be the $C^{1+}$-smooth Anosov flow generated by $Y_\kappa$. Then one has that $\Phi_0$ is $C^{1+}$-smoothly orbit-equivalent to $\Phi$. By \eqref{eq. path 5} and Proposition \ref{prop Jac smooth orbit-equiv}, $\Phi_1$ is $C^{1+}$-smooth orbit-equivalent to $\Psi$. Thus, the last thing we need to prove  is that there are two paths consisting of Anosov flows generated by $C^{1+}$-smooth vector-field  from $\Phi_0$ to $\Phi$, and from $\Phi_1$ to $\Psi$, respectively.  We prove the case of $\Phi_1$ and $\Psi$, the other case follows from a same proof. Let $H_1^*$ be a $C^{1+}$-smooth orbit-equivalence from $\Phi_1$ to $\Psi$. Since $H$ is homotopic to identity, so is $H_1^*$.  By the Isotopy Theorem (see Theorem \ref{thm Isotopy}), there is a path $H^*(x,\kappa):M\times [0,1]\to M$ such that $H^*(\cdot,0)={\rm Id}_M$, $H^*(\cdot,1)=H_1^*$ and for $\kappa\in(0,1)$ the  map $H^*_\kappa:=H^*(\cdot,\kappa)$ is a $C^{1+}$-smooth diffeomorphism. Moreover, up to using the standard smoothing methods, we can further assume that  $H^*_\kappa:=H^*(\cdot,\kappa)$ is  $C^{\infty}$-smooth, for  all $\kappa\in(0,1)$. Let $\Psi^t_\kappa:= H^*_\kappa\circ \Phi^t_1\circ (H^*_\kappa)^{-1}$. Then we have that  $\Psi_0=\Phi_1$, and $\Psi_\kappa$ is Anosov flow generated by $C^{1+}$-smooth vector field $DH^*_\kappa\circ Y_1\circ D(H^*_\kappa)^{-1}$. Moreover, $\Psi_1$ is a smooth time change of $\Psi$. Equivalently, there are $C^{1+}$-smooth functions $v_0, v_1:M\to \RR_+$ such that  $\Psi_1$ is generated by $v_1\cdot Y_\Psi$ and $\Psi$ is generated by $v_0\cdot Y_\Psi$, where $Y_\Psi$ is the $C^{1+}$-smooth unit bundle tangent to $\cO_\Psi$.  Let $v_\kappa=\kappa v_1+(1-\kappa) v_0$, for $\kappa\in[0,1]$. Then the the vector-field $v_\kappa\cdot Y_\Psi$ gives a path from $\Psi$ to $\Psi_1$. This finish the proof of Theorem \ref{thm path-connect C1+vector}. 
\end{proof}

\begin{corollary}\label{cor path-connect C1+flow}
	Let $\Phi$ and $\Psi$ be two $C^{1+}$-smooth  Anosov flows on $3$-manifold $M$. Assume that $\Phi$ and $\Psi$ are orbit-equivalent by an orbit-equivalence $H$ homotopic to identity Id$_M$.  Then there is a path of $C^{1+}$-smooth Anosov flows connecting $\Phi$ and $\Psi$.
\end{corollary}
\begin{proof}
	By Proposition \ref{prop flow regularity}, there are  Anosov flows $\Phi_*$ and $\Psi_*$ generated by $C^{1+}$-smooth vector-fields on $M$ and conjugate to $\Phi$ and $\Psi$ via $C^{1+}$-smooth conjugacies $H_\Phi$ and $H_\Psi$ which are $C^1$-close to Id$_M$, respectively.  By Theorem \ref{thm path-connect C1+vector}, there is a path consisting of $C^{1+}$-smooth Anosov flows connecting $\Phi_*$ and $\Psi_*$. Since the $C^{1+}$-smooth diffeomorphism $H_\Phi$ is homotopic to Id$_M$, by the Isotopy Theorem, there is a path of $C^{1+}$-smooth diffeomorphisms $H_\Phi^\kappa\ (\kappa\in[0,1])$ with $H_\Phi^1=H_\Phi$ and $H_\Phi^0={\rm Id}_M$. Then the $C^{1+}$-smooth Anosov flow $\Phi^t_\kappa:= H_\Phi^\kappa\circ \Phi^t\circ (H_\Phi^\kappa)^{-1}$  forms a path from $\Phi$ to $\Phi_*$. Similarly, there is a   path  of $C^{1+}$-smooth Anosov flows connection  $\Psi_*$ and $\Psi$. This completes the proof of corollary.
\end{proof}

Since the Anosov flows and Anosov vector fields are $C^1$-open, one can strengthen Theorem \ref{thm path-connect C1+vector} and Corollary \ref{cor path-connect C1+flow} to  $C^r $-regularity case.
\begin{corollary}
	Let $r\geq1$. Let $\Phi$ and $\Psi$ be two $C^r$-smooth (or generated by $C^{r}$-smooth vector-fields)  Anosov flows  on $3$-manifold $M$. Assume that $\Phi$ and $\Psi$ are orbit-equivalent by an orbit-equivalence $H$ homotopic to identity Id$_M$.  Then there is a path of $C^r$-smooth (or generated by $C^{r}$-smooth vector-fields)  Anosov flows  on $M$ connecting $\Phi$ and $\Psi$.
	In particular, we get Theorem \ref{thm connect}. 
\end{corollary}

\subsection{The homotopy type of the  space of Anosov flows}
In this subsection, we prove Theorem \ref{thm component}. Recall that  for a transitive Anosov flow $\Phi$ on $3$-manifold $M$,
we denote by $\mathcal{A}^r(\Phi)$  the path component of the space of $C^r$-smooth Anosov vector fields on $M$ containing $\Phi$.   We will show that 
 \[ \mathcal{A}^r(\Phi)\simeq {\rm Diff}_0^r(M). \] 
 The sketch of our proof is similar to  \cite{FaG2014} which consider the homotopy type of the space of Anosov diffeomorphisms homotopic to an Anosov automorphism on $\TT^2$. Here, we heavily rely on  $3$-dimensional manifolds. In particular, for $3$-dimensional manifold $M$, one has that  ${\rm Homeo}(M)\simeq {\rm Diff}(M)\simeq {\rm Isom}(M)$. 

\begin{proof}[Proof of Theorem \ref{thm component}]
	
	 Recall that $\cO_0^r(\Phi)$ is the identity component of  the space of $C^r$-smooth Anosov flows which are orbit-equivalent to $\Phi$, namely,
	\[\cO_0^r(\Phi):=\big\{ \Psi\ |\ \Psi:  C^r\ \text{Anosov flow orbit-equivalent to}\ \Phi\ \text{via orbit-equivalence homotopic to}\ {\rm Id}_M\big\}.\]
	In order to avoid abusing symbols, let $\cO_0^r(\Phi)$ consist of $C^r$-smooth vector fields in this proof. Then by the path-connectedness (Theorem \ref{thm connect}) and the structural stability,  $\mathcal{A}^r(\Phi) = \cO_0^r(\Phi)$.

Recall that for $C^1$-generic transitive $C^r$-smooth Anosov flow $\Phi_*$, its centralizer is trivial, see for example \cite{O2021,S1979}. Namely,  
\[C^r(\Phi_*):=\big\{ H\in {\rm Diff}^r(M)\ |\ H\circ \Phi^t_*\circ H^{-1}= \Phi^t_*, \ t\in \RR   \big\},\]
 the centralizer of $\Phi_*$, is the set of  time-$\tau$ maps of $\Phi_*$, where $r> 1$ and $\tau\in\RR$. Hence, up to consider a $C^1$-perturbation $\Phi_*$ of $\Phi$, we can assume that $\Phi$ has trivial centralizer, since $\cO_0^r(\Phi_*)=\cO_0^r(\Phi)$ by the structural stability.  Let 
	\[\mathcal{C}_0^r(\Phi):=\big\{ \Psi\in \cO_0^r(\Phi)  \ |\ \Psi^t:=H\circ \Phi^t\circ H^{-1},\ \text{for some}\  H\in {\rm Diff}^r_0(M) \big\}.\]
	Since the centralizer  $C^r(\Phi)$ is trivial,  $C^r(\Phi)\subset {\rm Diff}_0^r(M)$ and $C^r(\Phi)$ is homeomorphic to $\RR$. Then, by the fact that $\mathcal{C}_0^r(\Phi)\cong{\rm Diff}_0^r(M)/C^r(\Phi)$, one has that ${\rm Diff}^r(M)$ is a fibration over  $\mathcal{C}^r(\Phi)$  with contractible fiber $\RR$. Thus,  
	\[ \mathcal{C}_0^r(\Phi)\simeq {\rm Diff}_0^r(M).  \]
	In the following, we are going to show that $\cO_0^r(\Phi)$ is homotopy equivalent to $\mathcal{C}_0^r(\Phi)$.
	
	Note that both $\cO_0^r(\Phi)$ and $\mathcal{C}_0^r(\Phi)$ have the homotopy type of CW complexes. Indeed, recall that both ${\rm Diff}_0^r(M)$ and  $\mathfrak{X}^r(M)$ the set of $C^r$-smooth vector fields on $M$ are separable absolute neighborhood retracts, so are $\cO_0^r(\Phi)$ and $\mathcal{C}_0^r(\Phi)$, since     $\cO_0^r(\Phi)$ is an open subset of  $\mathfrak{X}^r(M)$ by the structural stability and $\mathcal{C}_0^r(\Phi)$ is homotopy equivalent to ${\rm Diff}_0^r(M)$.  Thus by the Whitehead's Theorem (e.g. see \cite[Theorem 1.13.22 and Theorem 1.13.30]{S2020}), $\cO_0^r(\Phi)$ and $\mathcal{C}_0^r(\Phi)$ have the homotopy type of CW complexes, and they are homotopy equivalent if and only if
	\begin{align}
		\pi_k\big( \cO_0^r(\Phi)\big)\cong \pi_k\big( \mathcal{C}_0^r(\Phi)\big),\quad \forall k\in\NN. \label{eq. whitehead}
	\end{align}

	By the proof of Theorem \ref{thm connect}, we actually get the following lemma.
	\begin{lemma}\label{lem homotopy group}
		The inclusion $\mathcal{C}_0^r(\Phi) \hookrightarrow \cO_0^r(\Phi)$ induces an  epimorphism on their homotopy groups.
	\end{lemma}
	\begin{proof}[Proof of Lemma \ref{lem homotopy group}]
		Let $\mathbb{D}^k$ be a $k$-disk and $\alpha:(\mathbb{D}^k,\partial)\to \big( \cO_0^r(\Phi),\Phi \big)$ be a continuous map with $[\alpha]\in \pi_k\big( \cO_0^r(\Phi)\big)$.  By the path-connectedness of $ \cO_0^r(\Phi)$, for each $x\in\mathbb{D}^k$, there is a path of  $C^r$-smooth Anosov flows from  $\alpha(x)$ to $\mathcal{C}_0^r(\Phi)$. Using the same way of the proof of Theorem \ref{thm connect}, one can check that for all $x\in\mathbb{D}^k$, $\alpha(x)$ can be simultaneously homotopic  to flows in $\mathcal{C}_0^r(\Phi)$.  Since each flow $\alpha(x)$ has a path to $\Phi$ exactly, we can further let  the whole boundary $\alpha|_{\partial\mathbb{D}^k}$ be homotopic to $\Phi$, just by making the path from $\alpha(x)$ to $\Phi$ being continuous with respect to $x\in \partial\mathbb{D}^k$.  Hence, $\alpha$ is homotoped to $\hat{\alpha}:(\mathbb{D}^k,\partial)\to \big( \mathcal{C}_0^r(\Phi), \Phi\big)$, and $ \pi_k\big( \mathcal{C}_0^r(\Phi)\big)\to \pi_k\big( \cO_0^r(\Phi)\big)$ is epic.
	\end{proof}
	
	Recall that  ${\rm Diff}_0^r(M)\simeq {\rm Isom}_0(M)\cong \{{\rm Id}_M\} 
	\ \text{or}\  \mathbb{S}^1$, for 3-manifold admitting Anosov flows, see Remark \ref{rmk isom group}. We get that $\cO_0^r(\Phi)\simeq {\rm Diff}_0^r(M)$ is aspherical, i.e., $ \pi_k\big( \cO_0^r(\Phi)\big)=0$ for $k\geq 2$. Hence, for $k\geq 2$,  the epimorphism  $\pi_k\big( \mathcal{C}_0^r(\Phi)\big)\to \pi_k\big( \cO_0^r(\Phi)\big)$ is a trivial isomorphism.  For $k=1$, we consider the sequence \[\pi_1\big( \mathcal{C}_0^r(\Phi)\big)\to \pi_1\big(\cO_0^r(\Phi) \big)\to \pi_1\big( {\rm Homeo}_0(M)\big).\]
	The first map is induced by the inclusion and the second map is  given by the structural stability. Since $\mathcal{C}_0^r(\Phi)\simeq {\rm Diff}_0^r(M)$ and $\pi_1\big( {\rm Diff}^r_0(M)\big)\cong\pi_1\big( {\rm Homeo}_0(M)\big)$, the composition of above two maps is monic. Thus, the  epimorphism   $\pi_1\big( \mathcal{C}_0^r(\Phi)\big)\to \pi_1\big(\cO_0^r(\Phi) \big)$ is also an monomorphism.  As a result, we get \eqref{eq. whitehead}.  Hence, $\cO_0^r(\Phi)\simeq  \mathcal{C}_0^r(\Phi)\simeq {\rm Diff}^r_0(M)$ and we get Theorem \ref{thm component}.
\end{proof}

\section{Anosov Flows with $C^1$-Smooth Foliations}\label{sec C1 Foliation}

In this section, we focus on the Anosov flows admitting $C^1$-smooth strong hyperbolic foliations on $3$-manifolds. Firstly, we prove the rigidity on the conjugacy preserving $C^1$-strong hyperbolic foliations, i.e., Theorem \ref{thm flow Phi sC1 rigid}. Then we show the conjugacy classification on them, i.e., Theorem \ref{thm flow Phi sC1 flex}. Finally,  we get the representation of  their Teichm\"uller spaces, i.e., Corollary \ref{cor Techmuller space suC1}.

\subsection{Rigidity of Conjugacy Preserving Smooth Foliations}\label{subsec rigidity}

Beyond  Theorem \ref{thm flow Phi sC1 rigid} the 
rigidity of conjugacy preserving $C^1$-smooth strong hyperbolic foliations, we actually consider the following two more general cases in the sense of the rigidity of conjugacy preserving smooth foliations which may be not invariant under the flows. 

\begin{theorem}\label{thm C1 rigid 1}
	Let $\Phi, \Psi$ be two $C^{r}$-smooth $(r>1)$ transitive Anosov flows on $3$-manifold $M$ conjugate via $H$. Assume that there is a  $C^{1+}$-smooth one-dimensional foliation $\cL$ satisfying 
	\begin{itemize}
		\item $\cL$ transversally intersects with $\cF^s_\Phi$,
		\item  $T\cL$ and $E^{ss}_\Phi$ are not jointly integrable.
	\end{itemize}
	If $H(\cL)$ is  a $C^{1+}$-smooth foliation, then the conjugacy $H$ is $C^{r_*}$-smooth  along each leaf of $\cF^s_\Phi$.
\end{theorem}

Note that Theorem \ref{thm flow Phi sC1 rigid}, Corollary \ref{cor rigid C1su} and Corollary \ref{cor dicho} are just  corollaries of the above one.
\begin{proof}[Proof of Theorem \ref{thm flow Phi sC1 rigid}, Corollary \ref{cor rigid C1su} and Corollary \ref{cor dicho}]
	Recall that $\Phi$ is either a suspension with constant roof or $E_\Phi^{ss}\oplus E_\Phi^{uu}$ is not jointly integrable \cite{Pl1972}. Now, we can just assume the second case.
	By Proposition \ref{prop foliation C1 to C1+}, the $C^1$-smooth foliations $\cF^{ss}_\Phi$ and $\cF^{ss}_\Psi$ are actually $C^{1+}$-smooth.  Then Theorem \ref{thm flow Phi sC1 rigid} follows from applying Theorem \ref{thm C1 rigid 1} to $\cL=\cF^{uu}_\Phi$ and $H(\cL)=\cF^{uu}_\Psi$.  Further applying Proposition \ref{prop foliation C1 to C1+} and Theorem \ref{thm flow Phi sC1 rigid}  to both $C^{1+}$-smooth foliations $\cF^{uu/ss}_\Phi$, we get Corollary \ref{cor rigid C1su} and Corollary \ref{cor dicho}, by the Journ\'e Lemma.
\end{proof}

Theorem  \ref{thm C1 rigid 1} considers  the rigidity of conjugacy along the transversal of  preserved $C^{1+}$-smooth foliation. The following one can be viewed as the opposite of the above one, which shows the rigidity on the direction along the foliation itself. Differing from Theorem  \ref{thm C1 rigid 1} , the foliation cannot be invariant.

\begin{theorem}\label{thm C1 rigid 2}
	Let $\Phi, \Psi$ be two $C^{r}$-smooth $(r>1)$ transitive Anosov flows on $3$-manifold $M$ conjugate via $H$. Assume that there is a $C^{1+}$-smooth one-dimensional foliation $\cL$ satisfying
	\begin{itemize}
		\item $\cL$ transversally intersects with $\cF^s_\Phi$ and $\cL\neq \cF^{uu}_\Phi$,
		\item $T\cL\oplus T\cO_\Phi$ is jointly integrable. 
	\end{itemize}
	If $H(\cL)$ is a $C^{1+}$-smooth foliation, then the conjugacy $H$ is $C^{r_*}$-smooth  along each leaf of $\cF^u_\Phi$.
\end{theorem}

\begin{remark}
	By Lemma \ref{lem C1+ L} and Proposition \ref{prop foliation C1+}, for above flow $\Phi$, there exists a foliation $\cL$ satisfying the assumption in Theorem \ref{thm C1 rigid 2}. And we  note that Theorem \ref{thm C1 rigid 2} also holds for suspension case.
\end{remark}

\begin{remark}
	In the proof of Theorem \ref{thm C1 rigid 2}, we can weaken the hypotheses of the regularity of the foliations. Indeed, we just need $\cL$ and $H(\cL)$ be  $C^{1+}$-smooth restricted on each leaf of $\cL\oplus\cO_\Phi$ and $H(\cL\oplus\cO_\Phi)$, respectively. For simplicity, we just state the version of Theorem  \ref{thm C1 rigid 2}.
\end{remark} 

\begin{remark}\label{rmk no flex}
	Theorem \ref{thm C1 rigid 2} implies that there is no unstable Jacobian flexibility, if we deform an Anosov flow along a foliation rather than its strong unstable foliation, see Remark \ref{rmk Asaoka 2}.
\end{remark}

Both Theorem \ref{thm C1 rigid 1} and Theorem \ref{thm C1 rigid 2} will use a fact which is stated as the following lemma.
\begin{lemma}\label{lem first time}
	Let $\Sigma$ be a $C^{1+}$-smooth two-dimensional submanifold subfoliated by orbit foliation $\cO_\Phi$. Let $x,x' \in \Sigma$ such that $x'\in\cO_\Phi(x)$. Let $L_1$ and $L_2$ be two $C^{1+}$-smooth curves transverse to $\cO_\Phi$ and crossing points $x$ and $x'$, respectively.  If $x'$ is on the positive flow of $x$, then there is a neighborhood $L'_1$ of $x$ in $L_1$ such that for each point $y\in L'_1$, $\tau_\Phi^+(y)$ the first positive reaching time from $y$ to $L_2$, i.e.,
	\[\tau^+_\Phi(y):=\inf\big\{ t\geq 0\ |\ \Phi^t(y)\in L_2  \big\}, \]
	is $C^{1+}$-smooth with respect to $y$.  Similarly, if   $x'$ is on the negative flow of $x$, then there is a neighborhood $L'_1$ of $x$ in $L_1$ such that for each point $y\in L'_1$, $\tau_\Phi^-(y)$ the first negative reaching time from $y$ to $L_2$, i.e.,
	\[\tau^-_\Phi(y):=\sup\big\{ t\leq 0\ |\ \Phi^t(y)\in L_2  \big\}, \]
	is $C^{1+}$-smooth with respect to $y$. 	Moreover, when $x=x'$, there is a neighborhood $L'_1$ of $x$ in $L_1$ such that for each point $y\in L'_1$, $\tau_\Phi(y)$ the first reaching time from $y$ to $L_2$, 
	\[ \tau_\Phi(y)=\Big\{ \begin{array}{lr}
		\tau^+_\Phi(y),\quad   \text{if}\ \   \tau^+_\Phi(y)\leq -\tau^-_\Phi(y),\\
		\tau^-_\Phi(y), \quad  \text{if}\ \   \tau^+_\Phi(y)\geq -\tau^-_\Phi(y),
	\end{array}  \]
	is $C^{1+}$-smooth with respect to $y$.
\end{lemma}

\begin{proof}
	We just prove this for $\tau_\Phi^+$, the other cases' proofs are similar. Denote by $U$ the submanifold $\bigcup_{y\in L_1}\cO_{\Phi,{\rm loc}}(y)$, where $\cO_{\Phi,{\rm loc}}(y)$ is orbit from $y\in L_1$ to the first positive reach point on $L_2$, and we can take $L_1$ small such that $U$ is a flow box.  Locally, we can take a $C^{1+}$-smooth chart $\phi: U\to \RR^2$ such that $\phi(x)=0$, $\phi(L_1)$ is parallel to  $e_1$  and the flow $\widetilde{\Phi}^t:=\phi\circ \Phi^t\circ\phi^{-1}$ is parallel to $e_2$, where $\{e_1,e_2\}$ is a orthogonal  basis of $\RR^2$. Then the flow can be represent by 
	\[ \widetilde{\Phi}^t(y)=\big(y, l(y,t)\big),\quad \forall y\in \phi(L_1),\]
	and $l(y,t)$ is $C^{1+}$-smooth with respect to $y$ and $t$.  We assume that the $C^{1+}$-smooth curve $\phi(L_2)$ is a graph of $C^{1+}$-function $g$ defined on $\phi(L_1)$. 
	Since $t=\tau^+_\Phi(\phi^{-1}(y))$ is the solve of the equation
	\[ l(y,t)=g(y), \]
	and $\partial_tl(y,t)$ is the speed at the point $\widetilde{\Phi}^t(y)$ with $| \partial_tl(y,t) |\neq 0$, by the Implicit Function Theorem, we have that $\tau^+_\Phi(y)$ is $C^{1+}$-smooth in a neighborhood $L_1'$ of $x$ in $L_1$.
\end{proof}

Now, we prove Theorem \ref{thm C1 rigid 1} and Theorem \ref{thm C1 rigid 2}. 
Let $\Phi, \Psi$ be two $C^{r}$-smooth $(r>1)$ Anosov flows on $3$-manifold $M$ conjugate via  $H$ such that \[H\circ \Phi^t=\Psi^t\circ H.\] For convenience, we denote the foliation $\cL$ in both theorems by $\cL_\Phi$. Let $\cL_\Psi=H(\cL_\Phi)$.

\begin{proof}[Proof of Theorem \ref{thm C1 rigid 1}]
	
	By the assumption of $\cL_\Phi$, there are points $x,y\in M$ with $y\in \cL_\Phi(x)$ such that the holonomy map induced by $\cL_\Phi$
	\[ {\rm Hol}^{\cL_\Phi}_{x,y}:\cF^s_{\Phi,{\rm loc}}(x) \to \cF^s_{\Phi,{\rm loc}}(y),\]
	maps the local strong stable leaf $I(x):=\cF^{ss}_{\Phi,{\rm loc}}(x)$ to a $C^{1+}$-smooth curve $J(y):={\rm Hol}^{\cL}_{x,y}\big(I(x)\big)$ and $J(y)$ transversally (restricted on $\cF^u_\Phi(y)$) intersects with the local strong stable leaf $I(y):=\cF^{ss}_{\Phi,{\rm loc}}(x)$ at the unique intersection $y$, namely,
	\[J(y)\pitchfork I(y)=\{y\}.\]
	Since $I(y), J(y)\subset \cF^s_{\Phi,{\rm loc}}(y)$,  one can define the  first time $\tau_\Phi$ from $I(y)$ to $J(y)$ as in Lemma \ref{lem first time}, and  $\tau_\Phi(z)$ is $C^{1+}$-smooth with respect to $z\in I(y)$.
	
	By conjugacy $H$, we get two topological curves
	\[J(H(y)):=H(J(y)) \quad {\rm and} \quad I(H(y)):=H(I(y)).  \]
	We claim that both $J(H(y))$ and $I(H(y))$ are $C^{1+}$-smooth curve on $\cF^s_\Psi(H(y))$. Indeed,
	note that $I(H(y))$ is a local leaf of $\cF^{ss}_\Psi$, hence it  is a $C^{1+}$-smooth curve on $\cF^s_\Psi(H(y))$. Considering holonomy map  induced by $\cL_\Psi$
	\[ {\rm Hol}^{\cL_\Psi}_{H(x),H(y)}:\cF^u_{\Psi,{\rm loc}}(H(x)) \to \cF^s_{\Psi,{\rm loc}}(H(y)),\]
	one has that  ${\rm Hol}^{\cL_\Psi}_{H(x),H(y)}\circ H= H\circ {\rm Hol}^{\cL_\Phi}_{x,y}$, and  $J(H(y))={\rm Hol}^{\cL_\Psi}_{H(x),H(y)}\big(H(I(x))\big)$. Since $H(I(x))$ is a local leaf of $\cF^{ss}_\Psi$ and $\cL_\Psi$ is $C^{1+}$-smooth, we get that $J(H(x))$ is also a $C^{1+}$-smooth curve on $\cF^s_\Psi(H(y))$.
	Hence, by Lemma \ref{lem first time}, we can also get the  $C^{1+}$-smooth first time $\tau_\Psi$ from $I(H(y))$ to $J(H(y))$.
	
	By conjugacy $H$ again, one has that 
	\[\tau_\Phi(H^{-1}(w))=\tau_\Psi(w)\quad \forall w\in I(H(y)). \]
	Namely, the restriction of $H^{-1}$ on $I(H(y))$ is the solution of the equation
	\[ \tau_\Phi(z)=\tau_\Psi(w).\]
	Since $I(y)$ is transverse to $J(y)$ at $y$, one has that the derivative $\tau'_\Phi(y)\neq 0$.  By the Implicit Function Theorem, there is a neighborhood of $H(y)$ in $\cF^{ss}_\Psi(H(y))$, denoted by $I'(H(y))=\cF^{ss}_{\Psi, {\rm loc}}(H(y)) \subset I(H(y))$ of point $H(y)$ such that restricted on it, the conjugacy $H^{-1}|_{I'(H(y))}$ has H\"older continuous derivative, namely, $DH^{-1}|_{E^{ss}_\Psi}(w)$ is H\"older continuous with respect to $w\in I'(H(y))$.
	In the following, we will show that $\|DH^{-1}|_{E^{ss}_\Psi}(w)\|$ exists, non-zero and H\"older continuously varies with respect to every point $w\in M$. Then we conclude that $H$ is a $C^{1+}$-smooth along the strong stable leaves.
	
	\begin{lemma}\label{lem trans derivative}
		Let $w\in \cF^u_\Psi(w_0)$. If  $DH^{-1}|_{E^{ss}_\Psi}(w_0)$ exists at point $w_0$, then the derivative $DH^{-1}|_{E^{ss}_\Psi}(w)$ exists at $w$. And $\|DH^{-1}|_{E^{ss}_\Psi}(w)\|=0$ if and only if $\|DH^{-1}|_{E^{ss}_\Psi}(w_0)\|=0$. Moreover, if $DH^{-1}|_{E^{ss}_\Psi}$  is H\"older continuous  on a local leaf $\cF^{ss}_{\Psi,{\rm loc}}(w_0)$, then so is it on a local leaf $\cF^{ss}_{\Psi,{\rm loc}}(w)$.
	\end{lemma}
	\begin{proof}[Proof of Lemma \ref{lem trans derivative}]
		Note that for any $w\in \cF^u_\Psi(w_0)$, there are  local leaves $\cF^{ss}_{\Psi,{\rm loc}}(w_0)$ and $\cF^{ss}_{\Psi,{\rm loc}}(w)$   such that for every point $z\in \cF^{ss}_{\Psi,{\rm loc}}(w)$,
		\[ H^{-1}|_{\cF^{ss}_{\Psi,{\rm loc}}(w)}=     {\rm Hol}^{\cF^u_\Phi}_{H^{-1}(w_0), H^{-1}(w)}  \circ H^{-1}|_{\cF^{ss}_{\Psi,{\rm loc}}(w_0) }\circ {\rm Hol}^{\cF^u_\Psi}_{w,w_0}.\]
		Recall that the above holonomy maps \[{\rm Hol}^{\cF^u_\Psi}_{w,w_0}:\cF^{ss}_{\Psi,{\rm loc}}(w)\to \cF^{ss}_{\Psi,{\rm loc}}(w_0) \quad {\rm and} \quad {\rm Hol}^{\cF^u_\Phi}_{H^{-1}(w_0),H^{-1}(w)}:\cF^{ss}_{\Phi,{\rm loc}}(H^{-1}(w_0))\to \cF^{ss}_{\Phi,{\rm loc}}(H^{-1}(w))\] are $C^{1+}$-smooth. Hence, we get the lemma directly.
	\end{proof}
	
	By Lemma \ref{lem trans derivative}, the minimal  foliation $\cF^{u}_\Psi$ sends the differentiability of $H$ (along the strong stable direction) on $I'(H(y))$ to every point $w\in M$.  By the same reason, one has that there is a point $w_0\in I'(H(y))$ with $\|DH^{-1}|_{E^{ss}_\Psi}(w_0)\|\neq 0$. Otherwise, $\|DH^{-1}|_{E^{ss}_\Psi}(w)\|\equiv 0$ for all $w\in M$. It is a contradiction. Moreover, by the continuous of $DH^{-1}|_{E^{ss}_\Psi}$ on $I'(H(y))$, there is a neighborhood of $w_0$ in $I'(H(y))$ such that $\|DH^{-1}|_{E^{ss}_\Psi}\|\neq 0$. Again by Lemma \ref{lem trans derivative} and the minimality of $\cF^u_\Psi$, we get that for all $w\in M$, $\|DH^{-1}|_{E^{ss}_\Psi}(w)\|$ exists, non-zero and H\"older continuously varies with respect to $w$.
\end{proof}

The proof of Theorem \ref{thm C1 rigid 2} follows the same spirit as Theorem \ref{thm C1 rigid 1}.

\begin{proof}[Proof of  Theorem \ref{thm C1 rigid 2}]
	Since $T\cL\oplus T\cO_\Phi$ is jointly integrable, we denote the integral foliation by $\widetilde{\cL}_\Phi$. Let $\widetilde{\cL}_\Psi=H(\widetilde{\cL}_\Phi)$. Since $\cO_\Psi, \cL_\Psi \subset \widetilde{\cL}_\Psi$ are both $C^{1+}$-smooth foliations, one has that $\widetilde{\cL}_\Psi$ is tangent to the H\"older bundle $T\cO_\Psi\oplus T\cL_\Psi$ and by the Journ\'e Lemma $\widetilde{\cL}_\Psi$ is  also a $C^{1+}$-smooth foliation.
	
	Moreover, since $\cL_\Phi\pitchfork\cF^s_\Phi$ and $\cL_\Phi\neq \cF^{uu}_\Phi$, the foliation $\cL_\Phi$ is not $\Phi^t$-invariant, namely, there is $x_0\in M$ and $t_0>0$ such that $\Phi^{t_0}(\cL_\Phi(x_0))\neq \cL_\Phi(\Phi^{t_0}(x_0))$. Let $y_0=\Phi^{t_0}(x_0)$. We can choose $I(x_0):=\cL_{\Phi,{\rm loc}}(x_0)$ and $I(y_0):=\cL_{\Phi,{\rm loc}}(y_0)$ such that  the first positive reaching time $\tau_\Phi$ from $I(x_0)$ to $I(y_0)$, i.e., 
	\[\tau_\Phi(x):=\inf\big\{t>0\ |\  \Phi^t(x)\in I(y_0) \big\},\]
	is $C^{1+}$-smooth along $I(x_0)$.
	Since  $\cL_\Phi$ is not $\Phi^t$-invariant, we can further assume that $\tau'_\Phi(x_0)\neq 0$. Similarly, one can get a $C^{1+}$-smooth  first positive reaching time $\tau_\Psi$ from $I(H(x_0))\subset \cL_\Psi(H(x_0))$ to $I(H(y_0))\subset \cL_\Psi(H(y_0))$.  By the conjugacy $H$, one has that for each $z\in I(H(x_0))$,
	\[\tau_\Phi(H^{-1}(z))=\tau_\Psi(z). \]
	Particularly, $x=H^{-1}(z)$ restricted on $ I(H(x_0))$ is the solution of 
	\[\tau_\Phi(x)=\tau_\Psi(z). \]
	Since $\tau_\Phi$ and $\tau_\Psi$ are $C^{1+}$-smooth and $\tau'_\Phi(x_0)\neq 0$, there is a neighborhood $I'(H(x_0))\subset I(H(x_0))$ of $H(x_0)$ such that  restricted on the leaves of $\cL_\Psi$, $H^{-1}$ is differentiable in $I'(H(x_0))$ with H\"older continuous derivative. Recall that $\cL_\Phi$ and $\cL_\Psi$ are transverse to $C^{1+}$-smooth foliations $\cF^s_\Phi$ and $\cF^s_\Psi$, respectively. One can apply the same method in the proof of Theorem \ref{thm C1 rigid 1} to show that restricted on each leaf of $\cL_\Psi$, $\|DH^{-1}|_{T\cL_\Psi}\|$ exists, and is nonzero  and H\"older continuous everywhere on the whole manifold. Hence, $H$ is $C^{1+}$-smooth along $\cL_\Phi$. In particular, the cocycles over flow $\Phi$
	\[\alpha_{\Phi}(x,t,\cF^s_\Phi,T\cL_\Phi,a) \quad {\rm and} \quad \alpha_{\Psi}(H^{-1}(x),t,\cF^s_\Psi,T\cL_\Psi,a), \]
	are H\"older-cohomologous for a given metric $a$. By lemma \ref{lem metric to change bundle}, one has that 
	\[J_\Phi^u(x,t)=\alpha_{\Phi}(x,t,\cF^s_\Phi,E^{uu}_\Phi,a) \quad {\rm and} \quad J^u_{\Psi}(H^{-1}(x),t)=\alpha_{\Psi}(H^{-1}(x),t,\cF^s_\Psi,E^{uu}_\Psi,a)\] are cohomologous. Hence by Theorem \ref{thm delaLlave}, $H$ is smooth along the unstable foliation.
\end{proof}

\subsection{The conjugacy class of Anosov flow with $C^1$ foliation}
In this subsection, we prove Theorem \ref{thm flow Phi sC1 flex} and Corollary \ref{cor Techmuller space suC1}.  It is clear that   Theorem \ref{thm flow Phi sC1 flex} follows from the next general case.

\begin{theorem}\label{thm flow Phi sC1 flex 2}
	Let $\Phi$ be a transitive $C^{1+}$-smooth Anosov flow on $3$-manifold $(M,a)$ with $C^1$-smooth strong unstable foliation. Let $f:M\to \RR $ be a H\"older continuous function with topological pressures $P_{f}(\Phi)=P$.  Then there exist  a $C^{1+}$-smooth Anosov flow $\Psi$ of manifold $M$, a bi-H\"older continuous homeomorphism $H$ and a H\"older continuous metric $a_*(\cdot,\cdot)$ such that 
	\begin{enumerate}
		\item $\Psi$ is conjugate to $\Phi$ via $H$.
		\item The strong stable foliation of\, $\Psi$ is $C^{1+}$-smooth.
		\item The stable and unstable Jacobians with respect to metric $a_*(\cdot,\cdot)$ satisfy that
		\begin{align}
				J_\Psi^u(x,t)=-\alpha_\Psi\big(x,t, f\circ H^{-1}-P\big)\quad {\rm and}\quad J_\Psi^s(x,t)=J^s_\Phi\big(H^{-1}(x),t\big), \quad \forall x\in M,\ t\in \RR. \label{eq. Jacobiansu}
		\end{align}
	
	\end{enumerate}
	Moreover, such a $C^{1+}$-smooth Anosov flow $\Psi$ is unique, up to $C^{1+}$-smooth conjugacy.
\end{theorem}

To show the flexibility of unstable Jacobian for 3-dimensional Anosov flow $\Phi$ with $C^1$-smooth strong unstable foliation, we can replace  the $C^{1+}$-foliation $\cL$ in the proof of Proposition \ref{prop ami jacobian 1} by $\cF^{uu}_\Phi$.

\begin{proof}[Proof of Theorem \ref{thm flow Phi sC1 flex 2}]
	Let the $C^{1+}$-foliation $\cL$ in the proof of Proposition \ref{prop ami jacobian 1} to be $\cF^{uu}_\Phi$. Then we already get the bi-H\"older homeomorphism $\check{H}$, HA-flow $\check{\Phi}$ and H\"older metric $\check{a}$ such that
	 \begin{enumerate}
		\item The flow $\check{\Phi}$ is conjugate to $\Phi$ via $\check{H}$, i.e., $\check{H}\circ \Phi^t=\check{\Phi}^t\circ \check{H}$.
		\item The cocycles induced by foliations $\cF^{u/s}_\Phi$ satisfy that for all  $x\in M$ and $t\in \RR$,
		\[ 
	\alpha\big(x,t, \cF^u_{\check{\Phi}},\check{a}\big)=\alpha\big(\check{H}^{-1}(x),t, \cF^u_\Phi,a\big)\quad {\rm and} \quad \alpha\big(x,t, \cF^s_{\check{\Phi}},\check{a}\big)=-\alpha_\Phi\big(\check{H}^{-1}(x),t,f-P\big).\]
	\item The restriction $\check{H}|_{\cF^s_\Phi(x)}:\cF^s_\Phi(x)\to \cF^s_{\check{\Phi}}(\check{H}(x))$ is $C^{1+}$-smooth, for every $x\in M$. Since $\Phi$ is $C^{1+}$-smooth and $\check{H}|_{\cF^s_\Phi}$ is $C^{1+}$-smooth, we get that $\check{\Phi}$ is $C^{1+}$-smooth along each leaf of $\cF^s_{\check{\Phi}}$.
	\item 	$\check{H}$ is just a deviation along the each leaf of $\cL=\cF^{uu}_\Phi$. Thus,	$\cF^{uu}_{\check{\Phi}}:=\check{H}(\cF^{uu}_\Phi)=\cF^{uu}_\Phi$ 
	is still a $C^{1+}$-smooth foliation and is invariant under the flow $\check{\Phi}$.  Recall that $\cO_{\check{\Phi}}$ is a $C^{1+}$-smooth foliation. Hence $\check{\Phi}$ is $C^{1+}$-smooth restricted on the unstable foliation $\cF^u_{\check{\Phi}}$. 
	\end{enumerate}
 Thus $\check{\Phi}$ is a $C^{1+}$-smooth flow with cocycle induced by $\cF^s_{\check{\Phi}}$ H\"older cohomologous to function $(f-P)\circ \check{H}$ satisfying \eqref{eq. final negative}, and cocycle induced by $\cF^u_{\check{\Phi}}$ H\"older cohomologous to $J^s_\Phi\circ \check{H}^{-1}$, namely, the cocycles over $\check{\Phi}$ induced by $\cF^{s/u}_{\check{\Phi}}$ are uniformly expanding and contracting, respectively.   In particular,  $\check{\Phi}$ is a $C^{1+}$-smooth Anosov flow  satisfying this theorem. If it is necessary, one can apply Proposition \ref{prop flow regularity} to $\check{\Phi}$ such that $\check{\Phi}$ is $C^{1+}$-smoothly conjugate to an Anosov flow $\Psi$ generated by $C^{1+}$-smooth vector-field via smooth conjugacy $H_0$. Let  $H=H_0\circ \check{H}$. Then $(\Psi, H,a_*)$ satisfies this theorem , up to change the metric as Claim \ref{claim change metric} such that $a_*$ makes the cohomology relation being exact equations. It follows from the result of de la Llave (Theorem \ref{thm delaLlave}) that up to $C^{1+}$-conjugacy, $\Psi$ is unique.
\end{proof}

Before we prove Corollary \ref{cor Techmuller space suC1}, we give the following Teichm\"uller space of an Anosov $3$-flow in the category of Anosov flows admitting $C^1$-smooth unstable foliations, as a summary of Theorem \ref{thm flow Phi sC1 flex} and Theorem \ref{thm flow Phi sC1 rigid}.	For an Anosov flow $\Phi$ on $3$-manifold $M$, one can  define 
\begin{align*}
		\mathcal{H}^{1+}_u(\Phi):=\Big\{ \Psi\ | \ C^{1+}\ \text{Anosov flow on}\  M\  & \text{conjugate}\  \text{to  }\Phi \  \\ & \text{with}\ C^1\  \text{strong unstable foliation}  \Big\},  
	\end{align*}
	and $\Psi_1\sim_u \Psi_2$, if  $\Psi_1$ being conjugate to $\Psi_2$ via a conjugacy smooth along each unstable manifold.

\begin{corollary}\label{cor teichmulle uC1}
		Let $\Phi$ be a $C^{1+}$-smooth transitive Anosov flow on $3$-manifold $M$ with $C^1$-smooth strong unstable foliation. There is a  nature bijection,
		\[\mathcal{H}^{1+}_u(\Phi)/_{\sim_u} \to \mathbb{F}^{\rm H}(M)/_{\sim_\Phi}. \]
		  Moreover, if  $\Phi$ is not a suspension with constant roof, then there is a nature bijection \[\mathcal{H}^{1+}_u(\Phi)/_{\sim}\to  \mathbb{F}^{\rm H}(M)/_{\sim_\Phi}.\]   
\end{corollary}
\begin{proof}
	Let us consider the map
	$\mathcal{T}:\mathcal{H}^{1+}_u(\Phi)\longrightarrow  \mathbb{F}^{\rm H}(M),$ by $\Psi\mapsto J^{u}_\Psi\circ H$.
	Let $\Psi_i\in \mathcal{H}^{1+}_u(\Phi)$, for $i=1,2$.  By Theorem \ref{thm delaLlave}, $\Psi_1\sim_u\Psi_2$, if and only if $J^{u}_{\Psi_1}\circ H_1$ is cohomologous to $J^{u}_{\Psi_2}\circ H_2$ where $H_i$ is the conjugacy from $\Phi$ to $\Psi_i$. Thus,  the map
	\begin{align*}
		[\mathcal{T}]:\mathcal{H}^{1+}_u(\Phi)/_{\sim_u}&\longrightarrow  \mathbb{F}^{\rm H}(M)/_{\sim_\Phi}, \\
		[\Psi]&\longmapsto [J^{u}_\Psi\circ H],
	\end{align*}
	is well-defined and a injection. By Theorem \ref{thm flow Phi sC1 flex 2}, $[\mathcal{T}]$  is also a surjection, thus a bijection.
	
	If $\Phi$ is not a suspension with constant roof, by Theorem \ref{thm flow Phi sC1 rigid} and the Journ\'e Lemma, $\Psi_1\sim_u\Psi_2$ if and only if $\Psi_1\sim \Psi_2$, namely $\Psi_1$ is smoothly conjugate to $\Psi_2$. This completes the proof .
\end{proof}

Combining  Cawley's work \cite{C1993} (Theorem \ref{thm Cawley}) and the above corollary, we can get Corollary \ref{cor Techmuller space suC1}.

\begin{proof}[Proof of Corollary \ref{cor Techmuller space suC1}]
	
	Let $\Phi$ be a $C^{1+}$-smooth Anosov flow on $3$-manifold $M$ with both $C^{1+}$-smooth strong stable and strong unstable foliations.  Let $\Psi\in \mathcal{H}^{1+}_{s,u}(\Phi)$. By the conjugacy,  we can just consider the following two cases:
	\begin{itemize}
		\item $\Phi$ and $\Psi$ are  suspension flows over  Anosov diffeomorphisms $A=A_\Phi, A_\Psi:\TT^2\to\TT^2$ with  constant roof-functions $\tau_\Phi, \tau_\Psi:\TT^2\to \RR^+$, without loss of generality, we assume that $\tau_\Phi\equiv 1\equiv\tau_\Psi$. Moreover, $A$ and $A_\Psi$ are conjugate via homeomorphism $h_\Psi:\TT^2\to \TT^2$.
		\item Neither $\Phi$ nor $\Psi$  is a suspension with constant roof-function.
	\end{itemize}
	For the second case, by Corollary \ref{cor teichmulle uC1} or Theorem \ref{thm flow Phi sC1 rigid},  the space $\mathcal{H}^{1+}_{s,u}(\Phi)/_\sim=[\Phi]$.

	For the first case, by the work of de la Llave \cite{dL1992} again, the map
	\begin{align*}
		[\mathcal{T}]:\mathcal{H}^{1+}_{s,u}(\Phi)/_{\sim}&\longrightarrow  \mathbb{F}^{\rm H}(\TT^2)/_{\sim_A}\times \mathbb{F}^{\rm H}(\TT^2)/_{\sim_A}, \\
		[\Psi]&\longmapsto \big( [J^s_{A_\Psi}\circ h_\Psi],  [J^u_{A_\Psi}\circ h_\Psi] \big), 
	\end{align*}
	is well-defined and injective.  On the other hand, by the work of Cawley (Theorem \ref{thm Cawley}), for any $\big([\phi_1], [\phi_2] \big)\in \mathbb{F}^{\rm H}(\TT^2)/_{\sim_A}\times \mathbb{F}^{\rm H}(\TT^2)/_{\sim_A}$, there is a $C^{1+}$-Anosov diffeomorphism $A'$ on $\TT^2$  conjugate to $A$ via $h$ such that $J^s_{A'}\circ h \sim_A \phi_1$ and $-J^u_{A'}\circ h \sim_A \phi_2$. Let $\Psi$ be suspension of $A'$ with $1$-roof. Then $\Psi$ is conjugate to $\Phi$ and has $C^{1+}$-smooth strong hyperbolic foliation.
\end{proof}

\appendix

\section{Appendix\quad  Metrics adapted to measures}

In appendix, we prove Proposition \ref{prop adapted metric} which claims that there is a metric ``discretely" adapting the transverse  measures. Recall that  $M$ is a  $3$-manifold with $C^\alpha\ (0<\alpha<1)$ Riemannian metric $a(\cdot,\cdot)$, $\Phi$ is an Anosov flow on $M$, 
 $\cL$ is a  one-dimensional $C^{1+\alpha}$-smooth foliation transverse to $\cF^s_\Phi$ and $\{V_i\}_{1\leq i\leq k}$ is a family of $s$-regular $\cL$-foliation boxes covering $M$.  
 	Firstly, we refine $\{V_i\}_{1\leq i\leq k}$ as follow.
 \begin{claim} \label{claim partition}
 There is a cover $\{U_i\}_{1\leq i\leq k_0}$ such that 
 	\begin{enumerate}
 		\item  For each $1\leq i\leq k_0$, $U_i$ is an $s$-regular $\cL$-foliation box. Namely, $U_i$ is proper,  the upper and lower boundaries $U_i^+$ and $U_i^-$ are local leaves of $\cF^s_\Phi$,  and  each local leaf $\cL(x,U_i)$ is intersecting with both $U_i^+$ and $U_i^-$, for all $x\in U_i$.
 				\item $\{U_i\}_{1\leq i\leq k_0}$ is a refined cover of $\{V_i\}_{1\leq i\leq k}$. Particularly,    
 			\begin{itemize}
 				\item For every $1\leq i\leq k_0$, there exists $1\leq j\leq k$ such that $U_i\subseteq V_j$.
 				\item For every $1\leq j\leq k$, $V_j^{\pm}\subseteq\bigcup_{i=1}^{k_0}U_i^{\pm}$.
 		\end{itemize}   
 		\item ${\rm int}(U_i)\cap {\rm int}(U_j)=\emptyset$, for $1\leq i\neq j\leq k_0$.
 	\end{enumerate}
 \end{claim}
 \begin{proof}[Proof of Claim \ref{claim partition}]
 	The proof follows the construction of Markov Partition. Here we do not need the Markov property,  thus we just need the local product structure of $\cL$ and $\cF^s_\Phi$.
 	For the $s$-regular family of $\cL$-boxes $\{V_i\}_{1\leq i\leq k}$, if $V_i\cap V_j\neq \emptyset$, let 
 	\begin{align*}
 		&V_{i,j}^1=\big\{ x\in T_i\ |\ \cF^s_\Phi(x,V_i)\cap V_j \neq \emptyset, \ \cL(x,V_i)\cap V_j\neq \emptyset  \big\},\\
 		&V_{i,j}^2=\big\{ x\in T_i\ |\ \cF^s_\Phi(x,V_i)\cap V_j \neq \emptyset, \ \cL(x,V_i)\cap V_j= \emptyset  \big\},\\
 		&V_{i,j}^3=\big\{ x\in T_i\ |\ \cF^s_\Phi(x,V_i)\cap V_j =\emptyset, \ \cL(x,V_i)\cap V_j\neq \emptyset  \big\},\\
 		&V_{i,j}^4=\big\{ x\in T_i\ |\ \cF^s_\Phi(x,V_i)\cap V_j= \emptyset, \ \cL(x,V_i)\cap V_j= \emptyset  \big\}.
 	\end{align*}
 	And let 
 	\[ R(x)=\bigcap\big\{  {\rm int}(V_{i,j}^n)\ |\ x\in V_i, V_i\cap V_j\neq \emptyset \ {\rm and}\ x\in V_{i,j}^n   \big\}.  \]
 	By the same proof of \cite[Theorem 3.12]{B1975}, there is an open and dense subset $Z\subset M$  such that  for each $x\in Z$, $R(x)$ is an open set, and $\big\{ \overline{R(x)}\big\}_{x\in Z}$ is actually a finite cover of $M$. Let  $\{U_i\}_{1\leq i\leq k_0}$ be this  refinement.  By the construction, all $U_i$ are still  $s$-regular $\cL$-foliation boxes without joint interiors and  $V_j^{\pm}\subseteq\bigcup_{i=1}^{k_0}U_i^{\pm}$ holds for each $1\leq j\leq k$.
 \end{proof}

 For each $p\in M$, $1\leq i\leq j\leq k_0$ and $\sigma,\tau=\pm$, denote  by $\cL^{\sigma,\tau}_{i,j}(p)$, the curve containing $p$ lying on $\cL(p)$ with two endpoints at $ U_i^\sigma$ and $U^\tau_j$ respectively, when the curve exists. Let $\{\mu_p\}_{p\in M}$ be a family of measures  subordinated to $\cL$.  Assume that for each $1\leq i\leq j\leq k_0$ and $\sigma,\tau=+$ or $-$,  $\mu_p\big(\cL^{\sigma,\tau}_{i,j}(p)\big)$ is H\"older continuous on $p\in U_i$ with H\"older exponent $\alpha$. 

\begin{proposition}\label{prop metric}
	There exists a $C^{\alpha}$ metric $\tilde{a}(\cdot,\cdot)$ of $M$ such that $(T\cL)_{\tilde{a}}^\perp=(T\cL)_a^\perp$, $\tilde{a}|_{(T\cL)_a^\perp}=a|_{(T\cL)_a^\perp}$ and  \[l_{\tilde{a}}\big( \cL^{\sigma,\tau}_{i,j}(p)\big)=\mu_p\big(\cL^{\sigma,\tau}_{i,j}(p)\big),\quad \forall 1\leq i\leq j\leq k_0, \ \  \sigma, \tau=\pm\ \   {\rm and}\ \  \forall p\in M,\]  where  $l_{\tilde{a}}$ is the length induced by metric $\tilde{a}$.
\end{proposition}

\begin{remark}
	Recall that $\{U_i\}_{1\leq i\leq k_0}$ is a refinement of $\{V_i\}_{1\leq i\leq k}$, particularly each $V_j^\pm$ is contained in   $\bigcup_{i=1}^{k_0}U_i^{\pm}$. Thus,
	Proposition \ref{prop adapted metric} follows from the above one, directly.
\end{remark}

\begin{proof}[Proof of Proposition \ref{prop metric}]
Since the bundle $T\cL$ is $C^{\alpha}$, we can assume that the metric $a=a_{i,j}dx^i\otimes dx^j$ such that $\frac{\partial}{\partial x^1}$ coincides with $T\cL$ and the functions $a_{i,j}$ are $C^{\alpha}$. For $\e>0$ and a set $A\subset M$, let  \[B^\cL_\e(A):=\big\{ p\in M\ |\ q\in A, \ p\in \cL(q) \ {\rm and}\  l_a([p,q])\leq \e \big\}.\]
	 Similarly, we denote the set 
	\[B^s_\e(A):=\big\{ p\in M\ |\ q\in A, \ p\in \cF^s_\Phi(q) \ {\rm and}\  l_a([p,q]^s)\leq \e \big\},\]
	where $[p,q]^s$ is the geodesic curve lying on $\cF^s_\Phi(q)$ with endpoints $p$ and $q$.
	
		\begin{figure}[htbp]
		\centering
		\includegraphics[width=8cm]{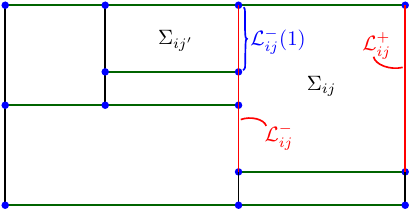}
		\caption{The set $\Sigma_i$, i.e., the piece of the boundary $\partial U_i$ foliated by $\cL$, is divided into rectangles $\Sigma_{ij}$. The curve $\cL^-_{ij}$ is cut by other boundaries of $\Sigma_{ij'}$ into shorter curves. }	
	\end{figure}

	For $1\leq i\leq k_0$, the boundary $\partial U_i$ is the union of close local stable leaves $U_i^\pm$ and a close topological manifold $\Sigma_i$ subfoliated by $\cL$. Since $\cL$ is $C^{1+\alpha}$-smooth, we can assume that $\partial(U_i^\pm)$ are piecewise $C^{1+\alpha}$-smooth curves and $\Sigma_i$ is piecewise $C^{1+\alpha}$-smooth.  For some $1\leq i\neq j\leq k_0$, if \[\Sigma_{ij}:=\Sigma_i\cap \Sigma_j\neq \emptyset,\]
 then it is a rectangle subfoliated by foliations $\cL$ and $\cF^s_\Phi\cap \Sigma_{ij}$.  Even if the case that ${\rm int}_{\Sigma}(\Sigma_{ij})=\emptyset$, where the interior ${\rm int}_{\Sigma}$  is with respect to the topological of surface $\Sigma_i$, it is a set of  finitely many points and curves lying on $\cF^s_\Phi$ or $\cL$, thus it also can be seen as  a union of  degenerate rectangles.    Since ${\rm int}(U_i)\cap {\rm int}(U_j)=\emptyset$, one has that for $j\neq j'$, \[{\rm int}_{\Sigma}(\Sigma_{ij})\cap  {\rm int}_{\Sigma}(\Sigma_{ij'})=\emptyset.\]
It is clear that $\Sigma_i=\bigcup_{1\leq j\leq k_0}\Sigma_{ij}$, and 
 the  boundary of rectangle $\Sigma_{ij}$ is the union of two leaves of $\cL$ denoted by $\cL_{ij}^+$ and $\cL_{ij}^-$ ($\cL_{ij}^+$ and $\cL_{ij}^-$ could degenerate to be one curve),  and two curves (could degenerate to be two points) lying on  $V_i^\sigma$ and $V_j^\tau$ for some $\sigma,\tau=\pm$.  For a curve $\cL_{ij}^\sigma$, it is possible that there exist other $\cL$-boundaries of rectangles, e.g.  $\cL_{ij'}^\tau$,  such that  ${\rm int}_{\cL}(\cL_{ij}^\sigma)\cap {\rm int}_{\cL}(\cL_{ij'}^\tau)\neq \emptyset$. In this case, we cut $\cL_{ij}^\sigma$ into shortest curves given by the intersection of $\cL_{ij}^\sigma$ with some $\cL_{ij'}^\tau$. We number these shortest curves by $\cL_{ij}^\sigma(n)$, $1\leq n\leq n(i,j,\sigma)$, here ${\rm int}_{\cL}(\cL_{ij}^\sigma(n_1))\cap {\rm int}_{\cL}(\cL_{ij}^\sigma(n_2))\neq \emptyset$, for $n_1\neq n_2$, and $\cL_{ij}^\sigma=\bigcup_{1\leq n\leq n(i,j,\sigma)}\cL_{ij}^\sigma(n)$.  See Figure 2.

		Let $\e$ be small enough such that for all $i,i',j,j'\in[1,k_0]$, $\sigma,\tau=+,-$, $n\in[1,n(i,j,\sigma)]$ and  $n'\in[1,n(i',j',\tau)]$ the following two items hold:
	\begin{itemize}
		 \item If $\cL_{ij}^\sigma(n)\cap \cL_{i'j'}^\tau(n')=\emptyset$, then ${\rm int}\big(B^s_\e(\cL_{ij}^\sigma(n))\big)\cap {\rm int}\big(B^s_\e(\cL_{i'j'}^\tau(n'))\big)=\emptyset$.
		\item   $\e\leq  \frac{l_a\big( \cL^{\sigma}_{ij}(n)\big)}{100}$ and $\e\leq \frac{\mu_p\big( \cL^{\sigma}_{ij}(n)\big)}{100}$, where $p\in \cL^{\sigma}_{ij}(n)$. 
	\end{itemize}

	First, we change the metric on each cylinder  $B^s_\e(\cL_{ij}^\sigma(n))$, $1\leq i\neq j\leq k_0$, $\tau=\pm$ and $1\leq n\leq n(i,j,\sigma)$.
		\begin{claim}\label{claim metric 1}
		Let $V\subset M$ be an $s$-regular $\cL$-foliation box.  For any $\e>0$ small enough, there exists  a $C^\alpha$ metric $\tilde{a}(\cdot,\cdot)$ of $V$ such that $(T\cL)_{\tilde{a}}^\perp=(T\cL)_a^\perp$, $\tilde{a}|_{(T\cL)_a^\perp}=a|_{(T\cL)_a^\perp}$,  $\tilde{a}|_{B_\e^\cL(V^\pm)}=a|_{B_\e^\cL(V^\pm)}$ and
		\[l_{\tilde{a}}([p^+,p^-])=\mu_p([p^+,p^-]),\quad \forall p\in V, \]
		where $p^\pm=\cL(p,V)\cap V^\pm$.
	\end{claim}
	
	\begin{proof}[Proof of Claim \ref{claim metric 1}]
	Recall that $\e$ is small such that $B^\cL_\e(V^+)\cap B^\cL_\e(V^-)=\emptyset$ and $\e \leq \frac{\mu_p([p^+,p^-])}{100}$, for all $p\in V$. Let $p^+_\e=\inf \big\{B^\cL_\e(V^+)\cap [p^+,p^-] \big\}$ and $p^-_\e=\sup \big\{B^\cL_\e(V^-)\cap [p^+,p^-] \big\}$. Note that $l_a([p^+,p^+_\e])=\e$ and $l_a([p^-,p^-_\e])=\e$.
		For $p\in V$, we denote $u(p)=\mu_p([p^+,p^-])-2\e$. Let $\gamma_p:[0,1]\to \cL(p)$ be a family of parameterizations of $[p_\e^+,p_\e^-]$. Let $s(p):=\int_0^1(2t-1)^{2K_0}\sqrt{a_{1,1}(\gamma_p(t))}dt$, where $K_0\in \NN$ is big enough such that $s(p)<u(p)$ for all $p\in V$.  Indeed, by $s(p)=\int_0^1(2t-1)^{2K_0}\sqrt{a_{1,1}(\gamma_p(t))}dt\leq (\int_0^1(2t-1)^{4K_0}dt)^{\frac{1}{2}}\cdot (\int_0^1a_{1,1}(\gamma_p(t))dt)^{\frac{1}{2}}$, we can choose such a $K_0\in\NN$.   For $i\neq 1$ or $j\neq 1$, let $\tilde{a}_{i,j}=a_{i,j}$. 
		Let 
		\[ \sqrt{{\tilde{a}}_{1,1}(\gamma_p(t)) }=(2t-1)^{2K_0}\sqrt{a_{1,1}(\gamma_p(t))}+6t(1-t)\big(u(p)-s(p)\big)>0, \]
		for $\gamma_p(t)\in [p^+_\e,p^-_\e]$. For $q\in [p^+,p^-]\setminus [p^+_\e,p^-_\e]$, let $\tilde{a}_{1,1}(q)=a_{1,1}(q)$. Then 
		\[l_{\tilde{a}}([p^+_\e,p^-_\e])=\int^1_0 \sqrt{\tilde{a}_{1,1}(\gamma_p(t)) } dt=u(p),\]
		and hence
		\[ l_{\tilde{a}}([p^+,p^-])=l_{\tilde{a}}\big([p^+,p^-]\setminus [p^+_\e,p^-_\e] \big) +l_{\tilde{a}}([p^+_\e,p^-_\e])=2\e+u(p)=\mu_p([p^+,p^-]).\]
		By the formula of $\tilde{a}_{1,1}$, it is clear that the metric $\tilde{a}$ is $C^{\alpha}$.
	\end{proof}

	 Without loss of generality,  we can assume $B^s_\e(\cL_{ij}^\sigma(n))$  an $s$-regular $\cL$-foliation box. Then applying Claim \ref{claim metric 1} to each cylinder  $B^s_\e(\cL_{ij}^\sigma(n))$,  we get the desired metric $\tilde{a}(\cdot,\cdot)$  in $B^s_\e(\cL_{ij}^\sigma(n))$. Notice that such two cylinders  are either disjoint or the intersection is a local stable manifold. Thus, by Claim \ref{claim metric 1}, the metric  $\tilde{a}(\cdot,\cdot)$ on the intersection of different cylinders are coherent, and actually equals to the original metric $a(\cdot,\cdot)$. Thus, $\tilde{a}(\cdot,\cdot)$ is well-defined in the union of cylinders  $B^s_\e(\cL_{ij}^\sigma)$, for all $1\leq i\neq j\leq k_0$ and  $\sigma=\pm$. In the following, we extend the metric on these cylinders to each $U_i$ and hence to $M$.  
	
	\begin{figure}[htbp]
		\centering
		\includegraphics[width=11cm]{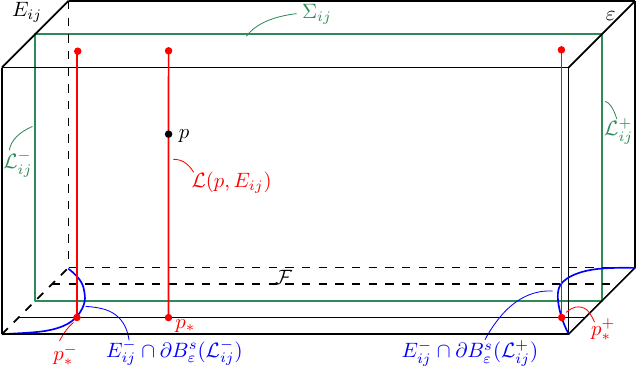}
		\caption{The metric of point $p$ is given by the metric of $p_*^-$ and $p_*^+$.}	
	\end{figure}
	
	We first extend $\tilde{a}(\cdot,\cdot)$ to the tubular neighborhood $B_\e^s(\Sigma_i)$ of  $\Sigma_i$. For $\Sigma_{ij}$, we consider  its $\e$-tubular neighborhood $E_{ij}:= B^s_\e(\Sigma_{ij})$. Recall $B^s_\e(\cL_{ij}^\pm)\subset E_{ij}$.  Without loss of generality,  we  see $E_{ij}$ as a $s$-regular $\cL$-foliation box. Let $E_{ij}^\pm$ be two pieces of the boundary $\partial E_{ij}$ transverse to $\cL$, which are  local stable manifold. And we can choose $E_{ij}^\pm$ such that their boundaries $\partial E_{ij}^\pm$ are piecewise smooth curves.  For $p\in E_{ij}\setminus \big( B^s_\e(\cL_{ij}^+)\cup B^s_\e(\cL_{ij}^-)\big)$, we project point $p$ to $p_*\in E_{ij}^-$ via the holonomy maps induced by  foliation $\cL$, namely, $p_*=E_{ij}^-\cap \cL_{\rm loc}(p)$. Let $\cF$ be a smooth one-dimensional foliation of $E_{ij}^- $ such that one of its leaves coincides with  $\Sigma_{ij}\cap E_{ij}^-$. Let $p^+_*=\inf\{ \cF(p_*)\cap B^s_\e(\cL_{ij}^+)   \}$ and $p^-_*=\sup\{ \cF(p_*)\cap B^s_\e(\cL_{ij}^-)   \}$, where the signs $+, -$  and the infimum, supermum are coincide with the orientation of $\cF$,  see Figure 3.
	Let $\gamma_p:[0,1]\to \cL(p,E_{ij})$. For $p=\gamma_p(t)$ and $p_*=\gamma_p(0)$, let $\tilde{a}_{i',j'}(p)=a_{i',j'}(p)$ for $i'\neq 1$ or $j'\neq 1$, and
	\begin{align}
		\sqrt{\tilde{a}_{1,1}\big( \gamma_p(t) \big)}= \frac{d_{\cF}(p_*,p_*^-)}{d_{\cF}(p_*^-,p_*^+)}\cdot\frac{\mu_p\big(\cL(p,E_{ij})\big)}{\mu_{p^-_*}\big(\cL(p^-_*,E_{ij})\big)}&\sqrt{ \tilde{a}_{1,1}\big(\gamma_{p^-_*}(t)  \big)} \notag\\ 
		 &+ \frac{d_{\cF}(p_*,p_*^+)}{d_{\cF}(p_*^-,p_*^+)}\cdot\frac{\mu_p\big(\cL(p,E_{ij})\big)}{\mu_{p^+_*}\big(\cL(p^+_*,E_{ij})\big)} \sqrt{\tilde{a}_{1,1}\big(\gamma_{p^+_*}(t)  \big)}, \label{eq. metric on E}
	\end{align}
	where $d_\cF(\cdot,\cdot)$ is the distance along $\cF$ induced by $a(\cdot,\cdot)$, and $a_{1,1}\big(\gamma_{p^\pm_*}(t)  \big)$ has already been defined, since $\gamma_{p^\pm_*}(t)  \in B^s_\e(\cL_{ij}^\pm)$. It is cleat that 
	\begin{align}
		\int_0^1\sqrt{\tilde{a}_{1,1}\big( \gamma_p(t) \big)}dt=\mu_p\big(\cL(p,E_{ij})\big), \label{eq. metric 1}
	\end{align}
	and $a_{1,1}$ is $C^\alpha$.   We notice that up to adjust \eqref{eq. metric on E}  by using similar functions $s(p)$ and $u(p)$ in Claim \ref{claim metric 1},  one can reconstruct  the $C^{\alpha}$ metric $\tilde{a}(\cdot,\cdot)$ in $E_{ij}$ such that it further has $\tilde{a}_{1,1}(\gamma_p(t))=a_{1,1}(\gamma_p(t))$ for $t\in [0,\e]\cup [1-\e,1]$. This construction admits that for different $\Sigma_{ij}$ and $\Sigma_{i'j'}$ with $E_{ij}\cap E_{i'j'}\neq \emptyset$, the metric $\tilde{a}(\cdot,\cdot)$ are cohenrece in the intersection. 
	Hence, we actually define the desired metric $\tilde{a}(\cdot,\cdot)$ in a $\e$-tubular neighborhood $B^s_\e(\Sigma_i)$ of each $\Sigma_i$, for $1\leq i\leq k_0$.
	
		\begin{figure}[htbp]
		\centering
		\includegraphics[width=9cm]{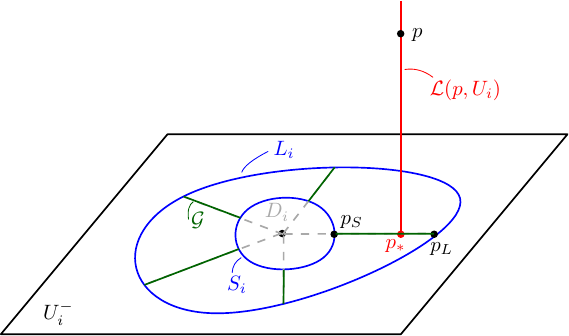}
		\caption{The metric of point $p$ is given by the metric of $p_S$ nad $p_L$.}	
	\end{figure}
	The rest part of the proof is, in each $U_i$,  extending the metric $\tilde{a}|_{B^s_\e(\Sigma_i)}$  to $U_i$. By slightly perturbing the tubular neighborhood of  $\Sigma_i$, we can assume that ${\rm int} (U_i^-)\cap \partial B^s_\e(\Sigma_i)=L_i$ is a simple closed smooth curve.	We take a disk $D_i\subset U_i^-$ inside another disk decided by $L_i$ such that $D_i\cap L_i=\emptyset$, denote $S_i=\partial D_i$. We denote by $T_i$,  the  annulus whose boundaries are $S_i$ and $L_i$. Note that the radials crossing the center point of $D_i$ give us a smooth foliation, denoted by $\mathcal{G}$, of the annulus $T_i$.	Let $E_i:=\bigcup_{p\in D_i}\cL(p,U_i)$, namely, $E_i$ is an $s$-regular $\cL$-foliation box of $D_i$ in $U_i$. By Claim \ref{claim metric 1}, we can define the metric $\tilde{a}$ in $E_i$.  In particular, we have defined the metric in $E_i\cup B^s_\e(\Sigma_i)$.
	Let $\gamma_p:[0,1]\to \cL(p,U_i)$ be  a family of parameterizations.  For $p=\gamma_p(t)$, let $p^*=\gamma_p(0):=\cL(p,U_i)\cap U_i^-$, $p_L=L_i\cap \mathcal{G}(p^*)$ and $p_S=S_i\cap \mathcal{G}(p^*)$, see Figure 4.
   Just like \eqref{eq. metric on E},	for  $p=\gamma_p(t)\in U_i\setminus \big( E_i\cup B^s_\e(\Sigma_i)  \big)$, let $\tilde{a}_{i',j'}(p)=a_{i',j'}(p)$ for $i'\neq 1$ or $j'\neq 1$, 
	\[ \sqrt{\tilde{a}_{1,1}\big( \gamma_p(t) \big)}= \frac{d_{\mathcal{G}}(p^*,p_S)}{d_{\mathcal{G}}(p_L,p_S)}\cdot\frac{\mu_p\big(\cL(p,U_i)\big)}{\mu_{p_S}\big(\cL(p_S,U_i)\big)}\sqrt{ \tilde{a}_{1,1}\big(\gamma_{p_S}(t)  \big)}  + \frac{d_{\mathcal{G}}(p^*,p_L)}{d_{\mathcal{G}}(p_L,p_S)}\cdot\frac{\mu_p\big(\cL(p,U_i)\big)}{\mu_{p_L}\big(\cL(p_L,U_i)\big)} \sqrt{\tilde{a}_{1,1}\big(\gamma_{p_L}(t)  \big)},\]
	where $\tilde{a}_{1,1}\big(\gamma_{p_S}(t)  \big)$ and $\tilde{a}_{1,1}\big(\gamma_{p_L}(t)  \big)$ have been defined in $E_i$ and $B^s_\e(\Sigma_i)$, respectively. Then,	\[ 	\int_0^1\sqrt{\tilde{a}_{1,1}\big( \gamma_p(t) \big)}dt=\mu_p\big(\cL(p,U_i)\big). \]
	Again, combining this construction and the method of Claim \ref{claim metric 1}, we can adjust $a(\cdot,\cdot)$ in $U_i\setminus \big( E_i\cup B^s_\e(\Sigma_i)  \big)$ such that it further satisfies that  $\tilde{a}_{1,1}(\gamma_p(t))=a_{1,1}(\gamma_p(t))$ for $t\in [0,\e]\cup [1-\e,1]$.
	The metric is coherent on the intersection of $U_i$ and $U_j$. Hence, we get the desired metric $\tilde{a}(\cdot,\cdot)$ on whole $M$.
\end{proof}

\bibliographystyle{plain}

\bibliography{ref}

	\flushleft{\bf Ruihao Gu} \\
School of Mathematics, Tongji University, 
Shanghai, 200092, China \\
\textit{E-mail:} \texttt{rhgu@tongji.edu.cn}\\

\flushleft{\bf Yi Shi} \\
School of Mathematics, Sichuan University, Chengdu, 610065, China\\
\textit{E-mail:} \texttt{shiyi@scu.edu.cn}\\

\end{document}